\documentclass[11pt,a4paper,oldfontcommands]{article}
\usepackage[utf8]{inputenc}
\usepackage[T1]{fontenc}
\usepackage{microtype}
\usepackage[dvips]{graphicx}
\usepackage{xcolor}
\usepackage{times}

\usepackage[english]{babel}
\usepackage{amsmath,mathtools}
\usepackage{amsfonts}
\usepackage{mathrsfs}
\usepackage{amsthm}
\usepackage{subfigure}
\usepackage{envmath}
\usepackage{amssymb}
\usepackage{dsfont}
\usepackage{bigints}
\usepackage{textgreek}
\usepackage{authblk}
\usepackage{titlesec}
\usepackage{enumitem}
\usepackage{url}

\usepackage[
breaklinks=true,
colorlinks=false,
bookmarks=true,bookmarksopenlevel=2]{hyperref}

\usepackage{geometry}
\newtheorem{theorem}{Theorem}[section]
\newtheorem{prop}[theorem]{Proposition}
\newtheorem{lem}[theorem]{Lemma}

\newtheorem{defi}[theorem]{Definition}

\theoremstyle{remark}
\newtheorem{rem}[theorem]{Remark}

\numberwithin{equation}{section}

\renewcommand{\Re}{\operatorname{Re}}

\newcommand*\diff{\mathop{}\!\mathrm{d}}

\DeclarePairedDelimiter\abs{\lvert}{\rvert}%
\DeclarePairedDelimiter\norm{\lVert}{\rVert}%
\newcommand{\tnorm}[1]{{\left\vert\kern-0.25ex\left\vert\kern-0.25ex\left\vert #1 
    \right\vert\kern-0.25ex\right\vert\kern-0.25ex\right\vert}}

\makeatletter
\let\oldabs\abs
\def\abs{\@ifstar{\oldabs}{\oldabs*}}
\let\oldnorm\norm
\def\norm{\@ifstar{\oldnorm}{\oldnorm*}}
\makeatother

\newcommand*{\myemail}[1]{%
    \normalsize\href{mailto:#1}{#1}\par
    }

\allowdisplaybreaks

\titleformat{\section}[block]{\centering \scshape \large}{\thesection.}{0.3\baselineskip}{}
\titlespacing{\section}{0pt}{*5}{*2}

\titleformat{\subsection}[block]{\bfseries}{\thesubsection.}{.5em}{}
\titlespacing{\subsection}{0pt}{*2.5}{*1}

\titleformat{\subsubsection}[runin]{\itshape}{\normalfont \thesubsubsection.}{.5em}{}[.]
\titlespacing{\subsubsection}{0pt}{*2.5}{0.5em}

\title{On the stationary solution of the Landau-Lifshitz-Gilbert equation on a nanowire with constant external magnetic field}
\date{\vspace{-1cm}}

\author[]{Guillaume Ferriere}

\affil[]{Univ. Lille, Inria, CNRS, UMR 8524 - Laboratoire Paul Painlevé, F-59000 Lille, France \\ \myemail{guillaume.ferriere@inria.fr}}

\begin{document}

\maketitle

\begin{abstract}
    We consider an infinite straight ferromagnetic nanowire, with an energy functional $E$ with easy-axis in the direction {$e_1$} of the nanowire and a constant external magnetic field {$h_0 e_1$} along the same direction.
    The evolution of its magnetization is governed by the Landau-Lifshitz-Gilbert equation associated to $E$.
    {If $h_0 \neq 0$}, we prove the existence of a stationary solution, {its} uniqueness up to the invariances of the equation and the instability of {its orbit} with respect to the flow.
    {Such solution has the same limits at infinity and looks like a $2$-domain wall when $h_0$ is small.}
    This analytic work is completed by some numerical simulations which are discussed afterwards and gives interesting new insights of the behavior of the solutions, in particular regarding the stability of 2-domain wall structures and more generally the interactions between domain walls.
\end{abstract}

\noindent\textbf{Keywords:} Landau-Lifshitz-Gilbert equation, ferromagnetism, stationary solution, instability, numerical simulation.

\medskip

\noindent\textbf{AMS classification:} Primary : 35B35, 35B38; Secondary : 35B40, 35B20, 35B30, 35Q60, 78M35, 65N06, 37M05.

\section{Introduction}

\subsection{A model for a ferromagnetic nanowire} \label{sec:ferro_model}

We are interested in a model of a ferromagnetic nanowire (of infinite length) as a straight line $\mathbb{R} e_1 \subset \mathbb{R}^3$ where 
\begin{equation*}
    e_1= \begin{pmatrix} 
    1 \\ 0 \\ 0
    \end{pmatrix}, \quad
    e_2= \begin{pmatrix} 
    0 \\ 1 \\ 0
    \end{pmatrix}, \quad
    e_3 = \begin{pmatrix} 
    0 \\ 0 \\ 1
    \end{pmatrix}
\end{equation*}
is the canonical basis of $\mathbb{R}^3$. The magnetization $m = (m_1, m_2, m_3): \mathbb{R} \to \mathbb{S}^2$ of this nanowire takes its values into the unit sphere $\mathbb{S}^2\subset \mathbb{R}^3$,  and we introduce the energy functional
\begin{equation*}
E(m) = \frac{1}{2} \int_{\mathbb{R}} |\partial_x m|^2 + (1-m_1^2) \diff x,
\end{equation*}
where $x$ is the variable in direction $e_1$ of the nanowire.
We refer to \cite{Cote_Ignat__stab_DW_LLG_DM} where this model was derived from the full 3D system by $\Gamma$-convergence in a special regime.
We want to study the evolution of the magnetization under the Landau-Lifshitz-Gilbert flow associated to $E$, that is the equation:
\begin{equation} \tag{LLG} \label{eq:llg}
    \partial_t m = m \wedge H(m) - \alpha m \wedge ( m \wedge H(m) ),
\end{equation}
where now $m: I \times \mathbb{R} \to \mathbb{S}^2$ is the time dependent magnetization ($I$ is an interval of time of $\mathbb{R}$). In this equation, $\wedge$ designates the cross product in $\mathbb{R}^3$, $\alpha >0$ is the damping coefficient, and the magnetic field $H$ is given by
\begin{equation} \label{eq:def_H}
    H(m) = - \delta E (m) + h(t) e_1.
\end{equation}
The last term $H_{ext} (t) \coloneqq h(t) e_1$ is the applied magnetic field, whereas $\delta E (m)$ in the first term is the variation of the energy, which writes (recall that $m_1^2+m_2^2+m_3^2=1$)
\begin{equation} \label{eq:def_delta_E}
    \delta E (m) = - \partial_{xx}^2 m + m_2 e_2 + m_3 e_3.
\end{equation}
The function $h: I \to \mathbb{R}$ is the (given) intensity of an applied external field, oriented on the axis $e_1$. Classically, it solely depends on the time variable $t$.
In this article, we will assume that it is constant: $h (t) = h_0$.

The \eqref{eq:llg} flow is equivariant under the following set of transformations:
\begin{itemize}[label=$\bullet$]
\item translations in space $\tau_y m(x) \coloneqq m(x-y)$ for $y\in \mathbb R$,
\item rotations $R_\phi \coloneqq \begin{pmatrix} 
1 & 0 & 0 \\
0 & \cos \phi & -\sin \phi \\
0 & \sin \phi & \cos \phi
\end{pmatrix}$ around the axis $e_1$ with angle $\phi\in \mathbb R$.
\end{itemize}
We define the group $G \coloneqq \mathbb R \times \mathbb R /2\pi \mathbb Z$ which naturally acts on functions $w: \mathbb R \to \mathbb R^3$ as follows: if $g= (y, \phi) \in G$, $g.w \coloneqq R_\phi \tau_y w = \tau_y R_\phi w$.
The action of $G$ preserves $\mathbb S^2$ valued functions, and so acts on magnetizations; it also extends naturally to functions of space and time, for which it preserves solutions to \eqref{eq:llg}. Also, we endow $G$ with the natural quotient distance over $\mathbb R^2$:
\begin{equation*}
    \forall g = (y,\phi) \in G,\quad \abs{g} \coloneqq \abs{y} + \inf \{ \abs{\phi+2k \pi}, k \in \mathbb Z \}.
\end{equation*}

{We also point out that \eqref{eq:llg} is invariant by reversal in space $m(x) \leftrightarrow m(-x)$. Moreover, for any solution $m = (m_1, m_2, m_3)$ to \eqref{eq:llg} with external magnetic field $h (t) \, e_1$, $(- m_1, m_2, -m_3)$ is solution to \eqref{eq:llg} with external magnetic field $- h(t) \, e_1$.}

\subsection{Functional spaces and Cauchy theory}

We define for $s \ge 1$ the spaces
\begin{align*}
\mathcal H^s &:= \{ m=(m_1, m_2, m_3) \in \mathscr C(\mathbb R, \mathbb S^2) \, | \,  \norm{m}_{\mathcal H^s} < +\infty \}\\ 
\nonumber &\textrm{with}\quad \norm{m}_{\mathcal H^s}:= \norm{m_2}_{L^2} + \norm{m_3}_{L^2} + \norm{m}_{\dot H^s}.
\end{align*}
The $\mathcal H^s$ spaces are modelled on the usual Sobolev spaces $H^s$ (see Section \ref{subsec:gen_notations} for general notations), but adapted to the geometry of the target manifold $\mathbb S^2$ and to the energy functional $E$: the main point is that $\abs{m_1} \to 1$ at $\pm \infty$ so that $m_1 \notin L^2$. The space $\mathcal H^1$ corresponds to the set of finite energy configurations $E(m) < +\infty$, in which case the energy gradient $\delta E(m)$ {belongs to} $H^{-1}$.

The Cauchy theory of \eqref{eq:llg} has been a challenging question, prompting numerous studies. Notably, Alouges and Soyeur \cite{Alouges_Soyeur__weak_LLG} investigated weak solutions, Lakshmanan and Nakamura \cite{Lakshmanan_Nakamura__LLG} employed the stereographic projection to transform \eqref{eq:llg} into a quasilinear dissipative Schrödinger equation, and Gutiérrez and de Laire \cite{Gutierrez_deLaire__Cauchy_LLG} delved into the Cauchy problem in the BMO space.
The following well-posedness result, quoted from \cite{Cote_Ignat__stab_DW_LLG_DM}, may not be the most optimal solution for initial data in $\mathcal{H}^s$ with $s \geq 1$, as discussed in \cite[Section~4]{Cote_Ignat__stab_DW_LLG_DM} {(see the footnote at the end of the proof of Theorem 4.1 in this article)}. Despite this, it suffices for our purpose and, notably, yields an energy dissipation identity of significant interest.
{In this result and in the rest of the paper, we denote by $\cdot$ the scalar product in $\mathbb{R}^3$.}

\begin{theorem}[Local well-posedness in $\mathcal H^s$] \label{th:lwp} Let $\alpha >0$ and $h \in L^\infty((0,+\infty), \mathbb R)$. Assume $s \ge 1$ and $m_0 \in \mathcal H^s$. Then there exist a maximal time $T_+= T_+(m_0) \in (0, +\infty]$ and a unique solution $m \in\mathscr C([0, T_+), \mathcal H^s)$ to \eqref{eq:llg} with initial data $m_0$. 
Moreover,
\begin{enumerate}
\item if $T_+ <+\infty$, then $\| m(t) \|_{\mathcal H^1} \to +\infty$ as $t \uparrow T_+$;
\item \label{item:continuity_flow} for $T<T_+$ (with $T_+$ finite or infinite), the map $\tilde m_0\in \mathcal H^s \to \tilde m\in \mathscr C([0,T], \mathcal H^s)$ is continuous in a small $\mathcal H^s$ neighbourhood of $m_0$ (for every initial data $\tilde m_0$ in that neighborhood, the maximal time of the corresponding solution $\tilde m$ satisfies $T_+(\tilde m_0) > T$);
\item if $s \ge 2$, one has the energy dissipation identity
: $t \mapsto E(m(t))$ is a locally Lipschitz function in $[0,T_+)$ (even $\mathscr C^1$ provided $h$ is continuous) and for all $t \in [0,T_+)$,
\begin{equation*}
\frac{d}{dt} E(m) = - \alpha \int ( |\delta E(m)|^2 - |m \cdot \delta E(m)|^2 )\, \diff x + \alpha h(t) \int (m \wedge e_1) \cdot (m \wedge \delta E(m)) \, \diff x.
\end{equation*}
\end{enumerate}
\end{theorem}

\subsection{Main result: existence, uniqueness and instability of stationary solutions}

Following the establishment of well-posedness, a natural question arises regarding the dynamics of \eqref{eq:llg}. In particular, the large-time behavior of solutions becomes an important subject, extensively explored in the context of many partial differential equations, as well as the study of stationary solutions, solitons and progressive waves as a first step.
It is indeed expected that generic solutions of this equation decompose into a superposition of independent structures called \textit{domain walls}. These domain walls are special structures, solution to \eqref{eq:llg} with limit $\pm e_1$ at $- \infty$ and the opposite at $+ \infty$, defined (up to the aforementioned invariances) by
\begin{equation*}
    w_*(x) \coloneqq \begin{pmatrix}  \cos ( \theta_*(x))  \\ \sin (\theta_*(x)) \\ 0 \end{pmatrix}  \quad
\textrm{with} \quad  \theta_*(x) \coloneqq 2\arctan (e^{x}).
\end{equation*}
This behavior aligns with the \textit{Soliton Resolution Conjecture}, commonly invoked in the context of dispersive PDEs. This conjecture (vaguely formulated) suggests that any global solution of a nonlinear dispersive PDE will eventually decompose at large time into a combination of non-scattering structures and a radiative term.

{The study of domain walls in this context has been started by Carbou and Labbé \cite{Carbou_Labbe__StaticWalls} without any external magnetic field, and improved with external magnetic field by Côte and Ignat \cite{Cote_Ignat__stab_DW_LLG_DM}.}
These authors proved that domain walls are stationary solutions of \eqref{eq:llg} when $h \equiv 0$, they are precessing explicitly with respect to the external magnetic field $h$, and they exhibit asymptotic stability under a smallness assumption on $h$. Moreover, they established an exponential convergence rate for both the solutions and the associated parameters (translations and rotations).
Building on this groundwork, Côte and the author of this paper \cite{Cote_Ferriere__2DW} delved into the study of \textit{2-domain walls}. These structures consist of two independent domain walls that move away from each other. While not being exact solutions to \eqref{eq:llg}, these authors proved that these structures are still asymptotically stable under a suitable external magnetic field.
These findings contribute to our understanding of the intricate behavior of solutions to \eqref{eq:llg}.

This paper is focused on the study of the stationary solutions of \eqref{eq:llg} in the presence of a constant external magnetic field $h (t) = h_0 \neq 0$.
As previously explained, the inclusion of stationary solutions is imperative for any comprehensive understanding of the (generic) large-time behavior.
However, to the best of the author's knowledge, such stationary solutions have not yet undergone dedicated numerical or analytic study, as they present different properties compared to domain walls.
The first result of this paper asserts the existence of these stationary solutions and, notably, establishes their uniqueness up to the invariances of the equation.

\begin{theorem} \label{th:stat_sol_unique}
    Let $h_0 \in \mathbb{R}$.
    \begin{itemize}
        \item If $h_0 \leq -1$, then the only stationary solution $w$ in $\mathcal{H}^1$ to \eqref{eq:llg} with $h (t) \equiv h_0$ satisfying $\lim_{- \infty} w = e_1$ is the constant solution $w \equiv e_1$.
        \item If $h_0 = 0$, then the only non-constant stationary solution $w$ in $\mathcal{H}^1$ to \eqref{eq:llg} with $h (t) \equiv 0$ satisfying $\lim_{- \infty} w = e_1$ is the domain wall $w \equiv w_*$, up to the invariances.
        \item If $h_0 \in (-1, 0)$ or $h_0 > 0$, then there exists a unique non-constant stationary solution $w_{h_0}$ in $\mathcal{H}^1$ to \eqref{eq:llg} with $h \equiv h_0$ such that $\lim_{{-} \infty} w_{h_0} = e_1$, up to the invariances.
    \end{itemize}
\end{theorem}

\begin{rem}
    The second point of this theorem is already known thanks to \cite{Cote_Ignat__stab_DW_LLG_DM}, so we will focus on $h_0 \neq 0$.
\end{rem}

In the last case of this theorem, the unique (up to the invariances) non-constant stationary solution $w_{h_0}$ is not explicit unlike the domain wall $w_*$. However, it satisfies several interesting properties, explicited hereafter. In particular, these solutions are planar and can be lifted with the help of an angular function $\theta_{h_0}$ solution to an explicit ODE.

\begin{theorem} \label{th:stat_sol}
    \begin{itemize}
        \item Let $h_0 > 0$. {The function $w_{h_0}$ given by Theorem \ref{th:stat_sol_unique} satisfies $\lim_{+ \infty} w_{h_0} = e_1$ and} can be taken with values in $\mathbb S^1 \times \{ 0 \}$ and symmetric: $w_{h_0} = (\cos \theta_{h_0}, \sin \theta_{h_0}, 0)$ with $\theta_{h_0}$ smooth and satisfying
        \begin{itemize}
            \item $\theta_{h_0} (0) = \pi$ and $\partial_x \theta_{h_0} (0) = \sqrt{4 h_0}$,
            \item $\lim_{- \infty} \theta_{h_0} = 0$ and $\lim_{+ \infty} \theta_{h_0} = 2 \pi$,
            \item $\theta_{h_0} - \pi$ is odd,
            \item $\theta_{h_0}$ is a solution of the following ODEs:
                \begin{equation} \label{eq:ODE_theta_v1}
                    - \partial_{xx} \theta + \sin \theta \cos \theta + h_0 \sin \theta = 0,
                \end{equation}
                \begin{equation} \label{eq:ODE_theta_2_v1}
                    \partial_x \theta = \sqrt{\sin^2 \theta + 2 h_0 (1 - \cos \theta)},
                \end{equation}
            \item $\sin \theta_{h_0} \in L^2$, $1 - \cos \theta_{h_0} \in L^1$, $\partial_x \theta_{h_0} \in H^k$ for any $k \geq 1$,
            \item {Lastly, with $H$ given by \eqref{eq:def_H}}, there holds
            \begin{equation} \label{eq:expr_H_w_h_0_v1}
                H (w_{h_0}) {(x)} = \Lambda (x) \, w_{h_0} {(x)}, \qquad {\forall x \in \mathbb{R},}
            \end{equation}
            where
            \begin{equation} \label{eq:def_Lambda_v1}
                \Lambda \coloneqq - 2 \sin^2 \theta_{h_0} + 3 h_0 \cos \theta_{h_0} - 2 h_0.
            \end{equation}
        \end{itemize}
        \item Let $h_0 \in (- 1, 0)$. {The function $w_{h_0}$ given by Theorem \ref{th:stat_sol_unique} satisfies $\lim_{+ \infty} w_{h_0} = e_1$ and} can be taken with values in $\mathbb S^1 \times \{ 0 \}$ and symmetric: $w_{h_0} = (\cos \theta_{h_0}, \sin \theta_{h_0}, 0)$ with $\theta_{h_0}$ smooth and satisfying
        \begin{itemize}
            \item $\theta_{h_0} (0) = \arccos{(- 1 - 2 h_0)}$ and $\partial_x \theta_{h_0} (0) = 0$,
            \item $\lim_{\pm \infty} \theta_{h_0} = 0$,
            \item $\theta_{h_0}$ is even,
            \item $\theta_{h_0}$ satisfies \eqref{eq:ODE_theta_v1} and:
                \begin{equation} \label{eq:ODE_theta_2_v2}
                    \partial_x \theta = \operatorname{sgn} (x) \sqrt{\sin^2 \theta + 2 h_0 (1 - \cos \theta)},
                \end{equation}
            \item $\sin \theta_{h_0} \in L^2$, $1 - \cos \theta_{h_0} \in L^1$, $\partial_x \theta_{h_0} \in H^k$ for any $k \geq 1$,
            \item {Lastly}, there holds \eqref{eq:expr_H_w_h_0_v1} with \eqref{eq:def_Lambda_v1}.
        \end{itemize}
    \end{itemize}
\end{theorem}

The proof of {the existence, the uniqueness and the properties} of such solutions is similar to the proof of \cite[Theorem~1.1]{Cote_Ignat__stab_DW_LLG_DM}. We adapted it to the case with an external magnetic field, which in turn implies different limits at infinity for {$w_{h_0}$} and $\theta_{h_0}$ {compared to the domain wall $w_*$}.

\begin{rem}
    Similar statements hold for
    {
    \begin{itemize}
        \item $\lim_{+ \infty} w_{h_0} = e_1$ (and the same assumptions on $h_0$),
        \item $\lim_{- \infty} w_{h_0} = - e_1$ (or $\lim_{+ \infty} w_{h_0} = - e_1$) and $h_0 < 0$ or $h_0 \in (0, 1)$.
    \end{itemize}}
    This can be easily seen by {the two symmetries of \eqref{eq:llg} mentioned at the end of Section \ref{sec:ferro_model}.}
    Therefore, we will focus on the case $\lim_{{-} \infty} m = e_1$ in this paper.
\end{rem}

\begin{rem}
    In Section \ref{sec:numerical}, we will see that these stationary solutions look like some kind of stationary $2$-domain walls when $h_0$ is small, roughly meaning that they perform 2 transitions between the constants $e_1$ and $- e_1$ (or almost). The uniqueness proved in Theorem \ref{th:stat_sol_unique} shows that there is no "stationary $k$-domain wall" (roughly defined as stationary solutions performing $k$ transitions between the constants $e_1$ and $- e_1$, or almost) for $k > 2$ in this context.
\end{rem}

With the existence and uniqueness of stationary solutions of \eqref{eq:llg} established, the natural progression of inquiry leads to the question of their stability. In other words, we now turn our attention to the discerning whether the stationary structures identified in the previous result are stable configurations, maintaining their form over time, or if they are sensitive to perturbations and evolve into different states. This quest for stability, whose definition is stated hereafter 
in the nonlinear setting, is crucial for characterizing the large-time behavior and practical implications of the identified stationary solutions.

\begin{defi} \label{def:nonlin_(un)stable}
    The stationary solution $w_{h_0}$ is \textit{(nonlinearly) orbitally stable} if for all $\varepsilon > 0$, there exists $\delta > 0$ such that, for any initial data $m_0 \in \mathcal{H}^{{1}}$ satisfying
    \begin{equation*}
        \norm{m_0 - w_{h_0}}_{H^1} < \delta,
    \end{equation*}
    the solution $m$ to \eqref{eq:llg} satisfies, for any $t > 0$,
    \begin{equation*}
        \inf_{g \in G} \norm{m(t) - g.w_{h_0}}_{H^1} < \varepsilon.
    \end{equation*}
    {It is said to be \textit{(nonlinearly) orbitally unstable} if it is not (nonlinearly) orbitally stable.}
\end{defi}

The second main result of this paper is the nonlinear instability of the stationary solution $w_{h_0}$.

\begin{theorem} \label{th:instability_simple}
    For any $h_0 > 0$ or $h_0 \in (-1, 0)$, $w_{h_0}$ is 
    (nonlinearly) orbitally unstable.
\end{theorem}

The instability of these stationary structures implies that, to some extent, an initial data which is a perturbation of this stationary structure yields a solution which will not remain near the orbit of this structure (constructed by translation and rotation around $e_1$).
This explains why such solutions {have} not been observed in numerical simulations for very general initial data.

We point out that the stability of the domain wall $w_*$ (corresponding to the case $h_0 \equiv 0$) was treated by Carbou and Labbé \cite{Carbou_Labbe__StaticWalls} and completed by Côte and Ignat \cite{Cote_Ignat__stab_DW_LLG_DM}. On the other hand, as formally pointed out in \cite{Gou_2011}, the constant solution $e_1$ becomes linearly unstable when $h_0 < -1$, and it is therefore expected that it is also nonlinearly unstable (and similarly for $- e_1$ with $h_0 > 1$). 

\subsection{A more precise statement}

Actually, the previous statement can be precised thanks to another energy: the Zeeman energy, which is usually given by the integral of $- H_{\textnormal{ext}} \cdot m$ along the wire. Here, such an integral is not finite. However, for any $m \in \mathcal{H}^1$ satisfying $\lim_{\pm \infty} = e_1$ and using the fact that $m_2^2 + m_3^2 = 1 - m_1^2 = (1 - m_1) (1 + m_1)$, there holds $m_1 - 1 \in L^1$. Thus, we can define the Zeeman energy by
\begin{equation*}
    E_{h_0}^Z (m) \coloneqq - h_0 \int_{\mathbb{R}} (m_1 - 1) \diff x.
\end{equation*}
From this, we can define the total energy $E_{h_0}$ which takes into account the external magnetic field:
\begin{equation} \label{eq:mod_en}
    E_{h_0} (m) \coloneqq E(m) + E_{h_0}^Z (m) = \frac{1}{2} \int_{\mathbb{R}} |\partial_x m|^2 + (1-m_1^2) \diff x - h_0 \int_{\mathbb{R}} (m_1 - 1) \diff x.
\end{equation}
Its variation is related not only to the variation $\delta E$ of the initial energy $E$, but also to the effective magnetic field $H$:
\begin{equation} \label{eq:delta_E_h0}
    \delta E_{h_0} (m) = \delta E (m) - h_0 e_1 = - H(m).
\end{equation}

\begin{rem}
    Any stationary solution $w$ of \eqref{eq:llg} satisfies $w \wedge H(w) = 0$. {In particular, for $w \in \mathcal{H}^1$, it is equivalent to be a stationary solution to \eqref{eq:llg} or to be a critical point of $E_{h_0}$.} Therefore, $w_{h_0}$ is also the unique non-constant critical point of $E_{h_0}$ in $\mathcal{H}^1$, up to the invariances.
\end{rem}

This total energy plays an important role in the evolution of the solution.
In particular, the statement of Theorem \ref{th:instability_simple} can {be} refined by an estimate: the time during which the solution stays close to the orbit of the stationary solution $w_{h_0}$ is bounded {from} above by an explicit value.
This estimate is valid as long as the total energy of the initial data is smaller than the total energy of the stationary solution.
We also prove that such initial data can be found as close to $w_{h_0}$ as one wants.

\begin{theorem} \label{th:instability}
    {Let $h_0 \in (-1, 0)$ or $h_0 > 0$, and let $w_{h_0}$ be given by Theorem 1.2.}
    There exist $\lambda, \varepsilon > 0$ such that for any $0 < \delta < \varepsilon$, the following property holds. Define
    \begin{equation*}
        V_\delta \coloneqq \{ m_0 \in \mathcal{H}^2 \, | \, \norm{m_0 - w_{h_0}}_{H^1} < \delta, E_{h_0} (m_0) < E_{h_0} (w_{h_0}) \}.
    \end{equation*}
    Then $V_\delta \neq \emptyset$, and {for} any $m_0 \in V_\delta$, {the solution $m$ to \eqref{eq:llg} with initial value $m_0$ satisfies $\inf_{g \in G} \norm{m(t) - g.w_{h_0}}_{H^1} \geq \varepsilon$ for some $t > 0$. Furthermore, $t$ can be chosen} such that
    \begin{equation} \label{eq:upper_bound_min_time_instability}
        t \leq \frac{1}{\lambda^2} \ln{\biggl( \frac{\varepsilon}{\lambda \abs{E_{h_0} (m_0) - E_{h_0} (w_{h_0})}} \biggr)}.
    \end{equation}
\end{theorem}

{The total} energy is actually related to the effective magnetic field $H (m)$, and is decreasing through the flow of \eqref{eq:llg}.
Although $w_{h_0}$ is a critical point with respect to the total energy $E_{h_0}$, it is only a saddle point, which explains its instability.
The proof of this theorem is, once again, inspired by \cite[Theorem~1.4]{Cote_Ignat__stab_DW_LLG_DM}.
First, a moving frame, adapted to the geometry of the equation, is introduced. 
It allows to reduce the dimension of the problem, from 3d (with a constraint) to 2d, thanks to a smallness assumption.
Then, expansions in the two remaining unknowns are performed on the total energy $E_{h_0}$ and its dissipation through the flow of the equation (Proposition \ref{prop:evolution_E_h_0}).
Two Schrödinger operators appear at the first order, 
from which some estimates are deduced.
Such estimates allow to get a global estimate on the evolution of the energy, from which a Gronwall lemma allows to conclude.

\begin{rem}
    The proof of Theorem \ref{th:instability} shows an even stronger property: {with $\lambda$ and $\varepsilon$ as in the statement, and $m_0 \in V_\delta$, if for some $\tau > 0$ there holds}
    \begin{equation*}
        \inf_{g \in G} \norm{m(t) - g.w_{h_0}}_{H^1} < \varepsilon
    \end{equation*}
    for all $t \in [0, \tau]$, then this infimum also satisfies the following lower bound {for all $t \in [0, \tau]$}:
    \begin{equation*}
        \inf_{g \in G} \norm{m (t) - g.w_{h_0}}_{H^1} \geq \lambda \abs{E_{h_0} (m_0) - E_{h_0} (w_{h_0})} e^{\lambda^2 t}.
    \end{equation*}
\end{rem}

\begin{rem}
    For the {instability result}, the initial data $m_0$ are taken in $\mathcal{H}^2$, whereas the assumption regarding the closeness of this initial data to the stationary solution only involves the $H^1$ norm.
    This simplification is employed here in order to use the energy dissipation identity in Theorem \ref{th:lwp} and so that $m \wedge H(m)$ is in $L^2$ for all time $t \geq 0$.
    It is purely technical: a limiting argument can be used to restrain this assumption to $m_0 \in \mathcal{H}^1$, in the same way as in \cite{Cote_Ignat__stab_DW_LLG_DM, Cote_Ferriere__2DW}.
\end{rem}

\begin{rem}
    Similarly as in \cite{Cote_Ignat__stab_DW_LLG_DM}, the analysis of this paper can also be performed with the addition of the Dzya\-lo\-shin\-skii-Mo\-riya interaction.
    Like the domain walls, the stationary solutions would be modified by a {dilation} and a rotation around $e_1$ with an angle linearly dependent of the space variable $x$, both transformations depending on the intensity of the aforementioned interaction.
    Apart from this change, we believe that our arguments still apply in this case.
\end{rem}

\subsection{Outline of the paper}

This paper is organized as follows.
Existence and uniqueness of the stationary solutions are proved in Section \ref{sec:ex_uni_stat_sol}, completed by exponential decay estimates at infinity.
The proof of the instability of the stationary solution is inspired from \cite{Cote_Ignat__stab_DW_LLG_DM}. In particular, we use a special space-dependent basis similar to the one used for the stability of the domain wall in the latter. This basis is introduced in Section \ref{sec:basis_schro_op}, along with two Schrödinger operators and their properties. These two operators appear in the expansion at the main order of the total energy and its dissipation, performed in Section \ref{sec:mod_en_disp_exp}.
The proof of Theorem \ref{th:instability} is then performed in Section \ref{sec:instability}.
We complete our study with numerical simulations, displayed and discussed in Section \ref{sec:numerical}.
This paper ends with a discussion about the conclusions, the open problems and some conjectures related to this study.

\subsection{Notations}

\subsubsection{General notations} \label{subsec:gen_notations}

\begin{itemize}
    
    \item $\mathcal{C}_b^\infty$ designates the space of $\mathcal{C}^\infty$ functions with all derivatives bounded.
    \item $(L^p)^k$ is usually used as short notation for the Lebesgue space $(L^p(\mathbb R, \mathbb R))^k$ (or equivalently $L^p(\mathbb R, \mathbb R^k)$) with $p\in [1, \infty]$.
    Similarly, $(\dot H^s)^k$ is classically used for the homogeneous Sobolev space $(\dot H^s(\mathbb R, \mathbb R))^k$ (or equivalently $\dot H^s(\mathbb R, \mathbb R^k)$) with $s \ge 0$ and $k \in \mathbb{N}$ whose seminorm is given through the Fourier transform:
    \begin{equation*} 
    \norm{f}_{(\dot H^s)^k}^2 \coloneqq \frac{1}{2\pi} \int_\mathbb R |\hat f (\xi)|^2 |\xi|^{2s} \diff \xi, \quad \text{where} \quad \hat f(\xi) = \int_{\mathbb R} e^{-ix\xi} f(x)\, \diff x.
    \end{equation*}
    In particular, $\norm{f}_{(\dot H^n)^k} = \norm{\partial_{x}^n f}_{(L^2)^k}$ for any $n \in \mathbb{N}$.
    Last, the Sobolev spaces $(H^s(\mathbb R, \mathbb R))^k$ (with $s \ge 0$ and $k \in \mathbb{N}$) are also denoted $(H^s)^k$.
    Though, most of the time, we will omit the $k$ exponent and simply denote these spaces $L^p$, $\dot H^s$ and $H^s$ when there is no possible confusion.
    {
    \item We denote $\langle \cdot, \cdot \rangle$ the $L^2$-scalar product (and more generally the $(L^2)^k$-scalar product):
    \begin{equation*}
        \langle f, g \rangle = \int_{\mathbb{R}} f \cdot g \diff x.
    \end{equation*}
    \item We denote $\langle \cdot, \cdot \rangle_{H^{-1}, H^1}$ the $H^{-1}-H^1$ duality bracket.}
    \item We denote $\mathbb S^2 \subset \mathbb{R}^3$ the 2-dimensional sphere: $\mathbb S^2 = \{ X \in \mathbb{R}^3 \, | \, \abs{X} = 1 \}$.
    
    \item We define for $s \ge 1$ the spaces
    \begin{align*}
    \mathcal H^s &:= \{ m=(m_1, m_2, m_3) \in \mathscr C(\mathbb R, \mathbb S^2) \, | \,  \norm{m}_{\mathcal H^s} < +\infty \}\\ 
    \nonumber &\textrm{with}\quad \norm{m}_{\mathcal H^s}:= \norm{m_2}_{L^2} + \norm{m_3}_{L^2} + \norm{m}_{\dot H^s}.
    \end{align*}
    For any interval $I \subset \mathbb{R}$, the space $\mathscr C(I, \mathcal H^s)$ is the set of maps $m$
    \begin{align*}
        m: I &\rightarrow \mathcal{H}^s \\
        t &\mapsto m(t)
    \end{align*}
    continuous with respect to the $H^s$ norm.
    
\end{itemize}

\subsubsection{Multilinear estimates in Sobolev spaces}

We will use the same notation $O_k^\ell (f)$ as in \cite{Cote_Ignat__stab_DW_LLG_DM} to express pointwise bounds that turn into Sobolev bounds with linear dependence in the highest term.

\begin{defi}
    For $k \geq 0$ and $\ell \geq 1$, and given a (possibly vector valued) function $f = (f_j)_{1 \leq j \leq J}$, we use the notation
    \begin{equation*}
        g = O_k^\ell (f)
    \end{equation*}
    for a (possibly vector valued) function $g$ if (each component of) $g$ is an homogeneous polynomial of degree $\ell$ in the components of $f$ and their derivatives such that the total number of derivatives in each term is at most $k$, and whose coefficients are $\mathscr{C}_b^\infty (\mathbb{R})$ functions. $g$ is then the sum of terms of the form
    \begin{equation*}
        {\mathfrak{a}} \prod_{j=1}^J \prod_{\kappa=0}^{k} (\partial_x^\kappa f_j)^{\ell_{j, \kappa}}, \qquad
        \textnormal{where} \quad
        \sum_{j, \kappa} \ell_{j, \kappa} = \ell, \quad
        %
        \quad
        \sum_{j, \kappa} {\ell}_{j, \kappa} \kappa \leq k, \quad
        \textnormal{and} \quad
        {\mathfrak{a}} \in \mathscr{C}_b^\infty.
    \end{equation*}
\end{defi}

\begin{lem} \label{lem:calc_negl_terms}
    \begin{enumerate}
        \item If $k' \geq k$, then $O_k^\ell (f) = O_{k'}^\ell (f)$.
        \item If ${\mathfrak{a}} \in \mathscr{C}^\infty_b$, then ${\mathfrak{a}} O_k^\ell (f) = O_k^\ell (f)$.
        \item $O_k^\ell (f_1) O_{k'}^{\ell'} (f_2) = O_{k + k'}^{\ell + \ell'} (f_1, f_2)$.
        \item $\partial_x O_k^\ell (f) = O_{k+1}^\ell (f)$,
        \item $O_k^\ell (f_1 + f_2) = O_k^\ell (f_1, f_2)$.
    \end{enumerate}
\end{lem}

\begin{lem} \label{lem:est_negl_terms}
    Assume that $u = O_k^\ell (f)$ with $\ell \geq 1$.
    \begin{itemize}
        \item If $k \geq 2$ and $f \in H^k$, then there holds
            \begin{equation*}
                \norm{u}_{L^2} \lesssim \norm{f}_{H^{k-1}}^{\ell - 1} \norm{f}_{H^k}.
            \end{equation*}
        \item If $k \in \{ 0, 1 \}$ and $f \in H^1$, then
            \begin{equation*}
                \norm{u}_{L^2} \lesssim \norm{f}_{H^1}^\ell.
            \end{equation*}
        \item If $\ell \geq 2$, $k=2$ and $f \in H^1$, then
            \begin{equation*}
                \abs{\int u \diff x} \lesssim \norm{f}_{H^1}^\ell.
            \end{equation*}
        \item If $\ell \geq 2$, $k \in \{ 3, 4 \}$ and $f \in H^2$, then
            \begin{equation*}
                \abs{\int u \diff x} \lesssim \norm{f}_{H^1}^{\ell-k+2} \norm{f}_{H^2}^{k-2}.
            \end{equation*}
    \end{itemize}

\end{lem}

The proof of such properties and bounds can be found in Claims 4.7 to 4.9 of \cite{Cote_Ignat__stab_DW_LLG_DM}.

\section{Existence, uniqueness and properties of the stationary solution} \label{sec:ex_uni_stat_sol}

This section is devoted to the general properties of the stationary solution.

\subsection{Existence and uniqueness}

We begin by proving Theorem \ref{th:stat_sol}, which asserts the existence and uniqueness of the stationary solution.

\begin{proof}[Proof of Theorem \ref{th:stat_sol}]
    The proof is quite similar to the proof of the existence of stationary domain walls in \cite[Theorem~1.1]{Cote_Ignat__stab_DW_LLG_DM}, and we use the same tactic.
    
    \medskip
    
    \textit{Step 1: Equation on $m$.}
    
    Let $m \in \mathcal{H}^1$ be {a stationary solution of \eqref{eq:llg} such that $\lim_{{-} \infty} m = e_1$. It is} a critical point to $E_{h_0}$ in $\mathcal{H}^1$ {and it} satisfies $m \wedge H(m) = 0$ and $\abs{m} = 1$. The first equation can be re-written as $H(m) = \Lambda (x) m$, which is
    \begin{equation} \label{eq:ODE_m}
        \partial_{xx} m - m_2 e_2 - m_3 e_3 + h_0 e_1 = \Lambda (x) m.
    \end{equation}
    The scalar function $\Lambda$, defined on $\mathbb{R}$, can be explicited thanks to the second equation:
    \begin{equation} \label{eq:Lambda}
        \Lambda (x) = H(m) \cdot m = - (\abs{\partial_x m}^2 + m_2^2 + m_3^2) + h_0 m_1.
    \end{equation}
    Since $m \in \mathcal{H}^1$, by a bootstrap argument, there holds $m \in \mathscr{C}^\infty (\mathbb{R})$.
    
    \medskip

    \textit{Step 2: $m$ is planar.}
    
    First, we show that
    \begin{equation} \label{eq:col_expr}
        m_2 \partial_x m_3 - m_3 \partial_x m_2 = 0 \quad \text{on } \mathbb{R}.
    \end{equation}
    Indeed, using \eqref{eq:ODE_m}, there holds
    \begin{align*}
        \partial_x \Bigl( m_2 \partial_x m_3 - m_3 \partial_x m_2 \Bigr) &= m_2 \partial_{xx} m_3 - m_3 \partial_{xx} m_2 \\
            &= m_2 ( \partial_{xx} m_3 - m_3 ) - m_3 ( \partial_{xx} m_2 - m_2 ) \\
            &= m_2 ( \Lambda (x) m_3 ) - m_3 ( \Lambda (x) m_2 ) \\
            &= 0,
    \end{align*}
    which shows that $m_2 \partial_x m_3 - m_3 \partial_x m_2$ is constant on $\mathbb{R}$.
    On the other hand, since $m \in \mathcal{H}^1$, we have $m_2, m_3 \in H^1$, and thus $m_2 \partial_x m_3 - m_3 \partial_x m_2 \in L^1$.
    This proves \eqref{eq:col_expr}, which implies that $(m_2 (x), \partial_x m_2 (x))$ and $(m_3 (x), \partial_x m_3 (x))$ are collinear in $\mathbb R^2$ for every $x$ in $\mathbb R$.
    On the other hand, by \eqref{eq:ODE_m}, these two vector fields solve the same first order linear ODE system in $(u, v)$:
    \begin{equation*}
        \partial_x u = v, \quad \partial_x v = (1 + \Lambda (x)) u.
    \end{equation*}
    By uniqueness in the Cauchy-Lipschitz theorem, the collinearity factor is constant in $x$, and thus there exists $\beta \in \mathbb R$ such that $m_2 = \beta m_3$ (or $m_3 = \beta m_2$, respectively) in $\mathbb R$.
    Therefore, choosing $\phi$ such that $\cot \phi = \beta$ (or $\tan \phi = \beta$, respectively), we conclude that $R_{- \phi} m$ satisfies the same equation and regularity as $m$, with the additional property that its 3rd component vanishes.
    
    \medskip
    
    \textit{Step 3: Lifting.}
    
    Since $m$ can be taken as a smooth function with values in $\mathbb S^1 \times \{ 0 \}$, there exists a smooth lifting $\theta: \mathbb R \rightarrow \mathbb R$ such that $m = (\cos \theta, \sin \theta, 0)$. For such a lifting, there holds
    \begin{align*}
        \abs{\partial_x \theta}^2 &= \abs{\partial_x m}^2 \in L^1, \\
        \sin^2 (\theta) &= m_2^2 \in L^1, \\
        \cos \theta - 1 &= m_1 - 1 \, {\xrightarrow[x \to - \infty]{} 0}.
    \end{align*}
    Therefore, we get
    \begin{align*}
        E (m) &= \frac{1}{2} \int \Bigl[ (\partial_x \theta)^2 + \sin^2 \theta \Bigr] \diff x < \infty,
    \end{align*}
    and the equation \eqref{eq:ODE_m} can be re-written in terms of $\theta$ as \eqref{eq:ODE_theta_v1}.
    Moreover, since $\lim_{{-} \infty} m = e_1$ {and since $\theta$ is defined up to a $2 \pi \mathbb{Z}$ additive constant, we can assume $\lim_{{-} \infty} \theta = {0}$}.
    We also point out that, conversely, if $\theta$ satisfies \eqref{eq:ODE_theta_v1} and $\lim_{{-} \infty} \theta = 0$, then $m = (\cos \theta, \sin \theta, 0)$ satisfies \eqref{eq:ODE_m} and $\lim_{{-} \infty} m = e_1$.

    Since $\theta$ is smooth, we can multiply and integrate \eqref{eq:ODE_theta_v1} by $\theta$, which yields
    \begin{equation*}
        - (\partial_x \theta)^2 + \sin^2 \theta - 2 h_0 ( \cos \theta - 1 ) = cst.
    \end{equation*}
   {The first two terms} are integrable on $\mathbb R$ {and the last one goes to $0$ at $- \infty$}, so this constant is zero and we get
    \begin{equation} \label{eq:theta_en}
        - (\partial_x \theta)^2 + \sin^2 \theta + 2 h_0 ( 1 - \cos \theta ) = 0.
    \end{equation}

    \medskip
    
    \textit{Step 4 : No non-constant solution when $h_0 \leq - 1$}
    
    In \eqref{eq:theta_en}, we see that the last two terms can be rewritten as
    \begin{align*}
        \sin^2 \theta + 2 h_0 ( 1 - \cos \theta ) &= 1 - \cos^2 \theta + 2 h_0 ( 1 - \cos \theta ) \\
            &= ( 1 - \cos \theta ) ( 1 + \cos \theta + 2 h_0 ) \\
            &= ( 1 - \cos \theta ) ( \cos \theta - 1 + 2 (h_0 + 1) ).
    \end{align*}
    Then, we see that, if $h_0 + 1 \leq 0$, then $\cos \theta - 1 + 2 (h_0 + 1) < 0$ as soon as $\theta \in \mathbb{R} \setminus 2 \pi \mathbb{Z}$, and therefore so is $\sin^2 \theta + 2 h_0 ( 1 - \cos \theta )$.
    Using this property in \eqref{eq:theta_en}, we can easily deduce that $\theta$ cannot take any value in $\mathbb{R} \setminus 2 \pi \mathbb{Z}$. As $\theta$ is continuous and satisfies $\lim_{- \infty} \theta = 0$, we conclude that $\theta \equiv 0$ and $m \equiv e_1$.
    
    \medskip
    
    \textit{Step {5}: Trajectories of the lifting for $h_0 > 0$.}

    Let $X \coloneqq (\theta, \partial_x \theta)$ and set the Hamiltonian
    \begin{equation*}
        \operatorname{Ham} (X_1, X_2) = \frac{1}{2} \Bigl( X_2^2 - \sin^2 X_1 - 2 h_0 (1 - \cos X_1) \Bigr), \qquad X = (X_1, X_2) \in \mathbb R^2.
    \end{equation*}
    From the previous equations, $X$ satisfies the dynamical system
    \begin{equation*}
        \partial_x X = \Bigl( \frac{\partial \operatorname{Ham}}{\partial X_2} (X), - \frac{\partial \operatorname{Ham}}{\partial X_1} (X) \Bigr),
    \end{equation*}
    and the trajectory $\{ X (x) \}_{x \in \mathbb R}$ is included in the zero set of $\operatorname{Ham}$.
    We emphasize that $1 - \cos \theta \geq 0$, and vanishes only on $2 \pi \mathbb Z$.
    For $h_0 > 0$, the zero set of $\operatorname{Ham}$ can be denoted as $Z^- \cup Z^0 \cup Z^+$ where
    \begin{equation*}
        Z^\pm \coloneqq \{ (X_1, X_2) \in \mathbb R^2 \, | \, \pm X_2 > 0, \operatorname{Ham} (X_1, X_2) = 0 \},
    \end{equation*}
    and
    \begin{equation*}
        Z^0 \coloneqq \{ (X_1, 0) \, | \, \operatorname{Ham} (X_1, 0) = 0 \} = 2 \pi \mathbb Z \times \{ 0 \}.
    \end{equation*}
    In particular, any connected component of $Z^+$ and $Z^-$ ends at two consecutive points of $Z^0$, which is actually the set of constant solutions of the previous dynamical system. Therefore, by uniqueness in the Cauchy-Lipschitz theorem, the trajectory $\{ X (x) \}_{x \in \mathbb R}$ either:
    \begin{itemize}
        \item is included in $Z^0$: the trajectory is stationary, i.e. $\theta \equiv {0}$ {with the condition $\lim_{- \infty} \theta = 0$}, corresponding to a constant solution of \eqref{eq:ODE_m},
        \item begins and ends at two consecutive points of $Z^0${, one of them being $(0, 0)$ with the assumption $\lim_{- \infty} \theta = 0$, so that the other one is either $(2 \pi, 0)$ or $(- 2 \pi, 0)$}. In particular, the sign of $\partial_x \theta$ does not change, and the total rotation of a critical point $\theta$ is given by $\int_{\mathbb R} \partial_x \theta \diff x = \pm 2 \pi$. This also means that {$\lim_{+ \infty} \theta = \pm 2 \pi$}.
    \end{itemize}

    \medskip
    
    \textit{Step {6}: Conclusion on existence and uniqueness for $h_0 > 0$.}

    Assuming that $m$ is not constant, we are in the second case of the previous disjunction.
    If $\lim_{+ \infty} \theta = - 2 \pi$, we can consider $- \theta$ instead of $\theta$: it satisfies the same ODE \eqref{eq:ODE_theta_v1}, and it corresponds to a rotation of angle $\pi$ around $e_1$ for $m$. Therefore, we can assume $\lim_{+ \infty} \theta = 2 \pi$.
    Thus, there exists some $x_0 \in \mathbb R$ such that $\theta (x_0) = \pi$, and we can assume that $x_0 = 0$ up to a translation $\tau_{x_0}$.
    Since the sign of $\partial_x \theta$ does not change, there holds $\partial_x \theta > 0$.
    Thus, with this property along with \eqref{eq:theta_en}, $\theta$ satisfies
    \begin{System} \label{sys:ODE_theta}
        \partial_x \theta = \sqrt{\sin^2 \theta + 2 h_0 (1 - \cos \theta)}, \\
        \theta (0) = \pi.
    \end{System}
    The function $f: y \mapsto \sqrt{\sin^2 y + 2 h_0 (1 - \cos y)}$ is smooth on $(0, 2 \pi)$, and we know that $\theta (x) \in (0, 2 \pi)$ for all $x \in \mathbb R$. Therefore, by the uniqueness in the Cauchy-Lipschitz theorem, the previous properties uniquely {define} $\theta$.
    %
    
    
    
    For the existence, let $\theta_{h_0}$ {be the maximal solution of \eqref{sys:ODE_theta}} given by the existence part of Cauchy-Lipschitz theorem. Then $\theta_{{h_0}}$ also satisfies \eqref{eq:ODE_theta_v1}.
    Moreover, one can compute an expansion of $f$ at $0^+$ and $2 \pi^-$.
    For instance, at {$0^+$}, we have $\sin y = y + O(y^3)$, $\sin^2 y = y^2 + O(y^4)$ and $1 - \cos y = \frac{y^2}{2} + O(y^4)$. Therefore, $f(y) = \sqrt{1 + h_0} \, y + O(y^3)$.
    Thus, it is easy to prove that not only $\theta_{h_0}$ does not reach $0$ {in finite time, which induces the global existence of $\theta_{h_0}$}, but {also that} $\theta_{h_0}$ and $\partial_x \theta_{h_0}$ converge exponentially to $0$ at $- \infty$ by Gronwall lemma, and so {is} $\partial_x^k \theta_{h_0}$ for any $k \geq 1$ thanks to \eqref{eq:ODE_theta_v1}.
    A similar conclusion can be achieved for $\theta_{h_0} - 2 \pi$ at $+ \infty$, therefore $\theta$ is globally defined and $\theta (x) \in (0, 2 \pi)$ for all $x \in \mathbb R$.
    This shows that $w_{h_0} = ( \cos \theta_{h_0}, \sin \theta_{h_0}, 0 ) \in \mathcal{H}^1$ is a non-constant critical point to $E_{h_0}$ such that $\lim_{\pm \infty} w_{h_0} = e_1$. Last, we emphasize that \eqref{eq:theta_en} along with \eqref{eq:Lambda} yields \eqref{eq:expr_H_w_h_0_v1}.
    
    \medskip
    
    \textit{Step {7}: Trajectories of the lifting for $h_0 \in (- 1, 0)$.}
    
    In this case, the trajectories are different. Indeed, $\operatorname{Ham} (\pi, X_2) = \frac{1}{2} X_2^2 - 2 h_0 > 0$ for all $X_2 \in \mathbb R$, which shows that $\theta$ cannot cross $\pi$.
    However, defining $g (y) \coloneqq \sin^2 y + 2 h_0 (1 - \cos y) = 1 - \cos^2 (y) + 2 h_0 (1 - \cos y)$, there holds $g (y) = g_2 (\cos y)$ where $g_2 (z) \coloneqq 1 - z^2 + 2 h_0 (1 - z)$ is a polynomial function of degree $2$ which has two roots: $1$ and $-1 - 2h_0$.
    By a further analysis, it is easy to check that $g > 0$ on $(0, \theta_c)$ and $g < 0$ on $(\theta_c, \pi)$ where $\theta_c = \arccos{(- 1 - 2 h_0)}$. Thus, $\theta$ cannot reach $(\theta_c, \pi)$, neither $(- \pi, - \theta_c)$ since $g$ is even. Similarly as before, $\theta$ is defined up to a $2 \pi$ additive constant, so that we can assume that {$\lim_{- \infty} \theta = 0$}, and thus $- \theta_c \leq \theta (x) \leq \theta_c$ for all $x \in \mathbb R$.
    Moreover, similarly as before, if $\theta (x_0) = 0$ for some $x_0$, then $\theta \equiv 0$ is constant.
    
    Assuming now that $\theta$ is not constant, this means that the sign of $\theta$ cannot change. We can assume that $\theta > 0$ (up to a rotation of angle $\pi$ about the axis $e_1$ for $w_{h_0}$).
    Then, $\theta$ reaches $\theta_c$ at some point $x_0$. Indeed, if not, then $\partial_x \theta = g(\theta)$ does not vanish, which means that the sign of $\partial_x \theta$ does not change, that $\theta$ is strictly monotone and has two different limits at infinity, limits which are in $[0, \theta_c]$ and must be constant solutions of \eqref{eq:ODE_theta_v1} and \eqref{eq:ODE_theta_2_v1}. Thus, these limits must be $0$ and $\theta_c$. However, if $\theta_c$ is a constant solution to \eqref{eq:ODE_theta_v1}, then $\cos \theta_c = - h_0$, whereas $\cos \theta_c = - 1 - 2 h_0$ from the expression of $\theta_c$, leading to a contradiction as $h_0 > - 1$.
    %
    
    Up to a translation in space, we can assume that $\theta (0) = \theta_c$. Then, $\partial_x \theta (0) = 0$ and $\partial_{xx} \theta (0) = \sin(\theta_c) (h_0 + \cos \theta_c) < 0$.
    Since $\theta$ cannot cross $0$ and $\partial_x \theta$ cannot change sign unless $\theta = \theta_c$, it is easy to prove that $\partial_x \theta > 0$ on $(- \infty, 0)$ and $\partial_x \theta < 0$ on $(0, \infty)$.
    Therefore, by \eqref{eq:theta_en}, we obtain \eqref{eq:ODE_theta_2_v2}.
    
    \medskip
    
    \textit{Step {8}: Conclusion on the existence and uniqueness for $h_0 \in (- 1, 0)$.}
    
    The equation \eqref{eq:ODE_theta_v1} along with $\theta (0) = \theta_c$ and $\partial_x \theta (0) = 0$ ensures that $\theta$ is global and unique by Cauchy-Lipschitz theorem.

    On the other hand, if we take $\theta_{h_0}$ satisfying \eqref{eq:ODE_theta_v1} along with $\theta_{h_0} (0) = \theta_c$ and $\partial_x \theta_{h_0} (0) = 0$, which exists by Cauchy-Lipschitz theorem, then $\theta_{h_0}$ satisfies all the above properties.
    Moreover, it is easy to compute that $f = \sqrt{g}$ satisfies $f(y) = \sqrt{1 + h_0} \, y + O (y^3)$ for $y \rightarrow 0$, which proves (thanks to \eqref{eq:ODE_theta_2_v2} and the fact that $\lim_{\pm \infty} \theta = 0$) that $\theta_{h_0}$ and $\partial_x \theta_{h_0}$ decays exponentially, so that $\theta_{h_0} \in H^1$.
    Last, the same arguments show that $w_{h_0} = (\cos \theta_{h_0}, \sin \theta_{h_0}, 0)$ satisfies all the above equations and is thus a critical point of $E_{h_0}$ in $\mathcal{H}^1$.
\end{proof}

\subsection{Exponential convergence at infinity}

Next, we explore additional properties of the stationary solution outlined in Theorem \ref{th:stat_sol}. Specifically, we focus on deriving decay estimates for the solution and its derivatives as $\abs{x}$ tends to $\pm \infty$.

\begin{lem} \label{lem:w_h_0_exp}
    For any $k \in \mathbb N$, there exists $C_k > 0$ such that, for any $x \in \mathbb R$,
    \begin{itemize}
        \item if $h_0 > 0$,
            \begin{gather*}
                \abs{\partial_x^k \theta_{h_0} {(x)}} \leq C_k e^{- \sqrt{1 + h_0} \abs{x}} \qquad \text{for } x < 0, \\
                \abs{\partial_x^k (\theta_{h_0} - 2 \pi) {(x)}} \leq C_k e^{- \sqrt{1 + h_0} \abs{x}} \qquad \text{for } x > 0,
            \end{gather*}
        \item if $h_0 \in (-1, 0)$,
            \begin{equation*}
                \abs{\partial_x^k \theta_{h_0} {(x)}} \leq C_k e^{- \sqrt{1 + h_0} \abs{x}},
            \end{equation*}
        \item in both cases,
            \begin{equation*}
                \abs{\partial_x^k (w_{h_0} - e_1) {(x)}} \leq C_k e^{- \sqrt{1 + h_0} \abs{x}}.
            \end{equation*}
    \end{itemize}
\end{lem}

\begin{proof}
    Let $f$ defined by
    \begin{align*}
        f: (0, 2 \pi) &\rightarrow \mathbb R \\
            y &\mapsto \sqrt{\sin^2 y + 2 h_0 (1 - \cos y)}.
    \end{align*}
    Then $\partial_x \theta_{h_0} = f (\theta_{h_0})$.
    Moreover, by the previous expansion provided in the proof of Theorem \ref{th:stat_sol}, we know that $f(y) = \sqrt{1 + h_0} \, y + O(y^3)$ near $0^+$.
    If $h_0 > 0$, since $\lim_{- \infty} \theta_{h_0} = 0$, this shows that there exists $C > 0$ and $A < 0$ such that, for every $x < A$,
    \begin{equation} \label{eq:ineg_theta}
        (\sqrt{1 + h_0} - C \theta_{h_0} (x)^2) \, \theta_{h_0} (x) \leq \partial_x \theta_{h_0} (x) \leq (\sqrt{1 + h_0} + C \theta_{h_0} (x)^2) \, \theta_{h_0} (x).
    \end{equation}
    By a backward Gronwall lemma, we can estimate that, for all $x < A$,
    \begin{equation*}
        \theta_{h_0} (x) \leq \theta_{h_0} (A) \exp{\Bigl( - \sqrt{1 + h_0} (A - x) + C \int_x^A \theta_{h_0} (y)^2 \diff y \Bigr)}.
    \end{equation*}
    We obtain the announced estimate for $\theta_{h_0}$ by using the fact that $\theta_{h_0} \in L^2 (- \infty, 0)$, since $\sin \theta_{h_0} \in L^2$ and $\lim_{- \infty} \theta_{h_0} = 0$.
    Then, using again \eqref{eq:ineg_theta}, we obtain the conclusion for $\partial_x \theta_{h_0}$. As for $\partial_{xx} \theta_{h_0}$, we use \eqref{eq:ODE_theta_v1}, and more generally the estimate for $\partial_x^k \theta_{h_0}$ (for $k \geq 2$) can be reached by induction, by differentiating \eqref{eq:ODE_theta_v1} $k-2$ times and applying the estimates of the lower derivatives.

    A similar conclusion can be reached for $2 \pi - \theta_{h_0}$ at $+ \infty$, for $\theta_{h_0}$ at $\pm \infty$ if $h_0 \in (-1, \infty)$, and finally the estimate for $w_{h_0}$ in both cases comes from these estimates on $\theta_{h_0}$ and the fact that $w_{h_0} = (\cos \theta_{h_0}, \sin \theta_{h_0}, 0)$.
\end{proof}

\section{Associated basis and Schrödinger operators} \label{sec:basis_schro_op}

In this section, we introduce an associated space-dependent basis and two Schrödinger operators related to our problem. This specialized basis is carefully selected to capture the essential features of the system dynamics near the stationary solution $w_{h_0}$, providing a valuable framework for the analysis of the evolution and for the computation of the expansions of some quantities.
But first, we introduce and analyze two significant Schrödinger operators. They play a fundamental role in our analysis, as they appear both in the operator of the linearized equation and in the nonlinear analysis through the total energy.

\subsection{Two important Schrödinger operators} \label{sec:schro_op}

\subsubsection{Definition and first properties}

We begin by defining the following operators, denoted $L_1$ and $L_2$:
\begin{gather*}
    L_1 \coloneqq - \partial_{xx} + 1 - 2 \sin^2 \theta_{h_0} + h_0 \cos \theta_{h_0}, \\
    L_2 \coloneqq L_1 - 2 h_0 (1 - \cos \theta_{h_0}) = - \partial_{xx} + 1 - 2 \sin^2 \theta_{h_0} + 3 h_0 \cos \theta_{h_0} - 2 h_0.
\end{gather*}
As it is classical for Schrödinger operators, they are unbounded self-adjoint operators in $L^2 (\mathbb R)$ with domain $H^2 (\mathbb R)$, and their quadratic form can be extended to $H^1 (\mathbb R)$.
Our primary interest lies in the spectrum of these operators, particularly the discrete spectrum contained in $(- \infty, 0]$. As an initial observation, we establish the presence of certain functions in their kernel.

\begin{prop} \label{prop:kernel_schro_op}
    There holds $L_1 \partial_x \theta_{h_0} = L_2 \sin \theta_{h_0} = 0$.
\end{prop}

\begin{proof}
    Coming back to \eqref{eq:ODE_theta_v1} and differentiating with respect to $x$ since $\theta_{h_0}$ is smooth, we get
    \begin{equation*}
        0 = - \partial_{xx} (\partial_x \theta_{h_0}) + (\cos^2 \theta_{h_0} - \sin^2 \theta_{h_0}) \partial_x \theta_{h_0} + h_0 \, \partial_x \theta_{h_0} \, \cos \theta_{h_0} = L_1 \partial_x \theta_{h_0}.
    \end{equation*}
    On the other hand, as $w_{h_0} = (\cos \theta_{h_0}, \sin \theta_{h_0}, 0)$ satisfies \eqref{eq:expr_H_w_h_0_v1}, taking the second component of this ODE leads to
    \begin{equation*}
        \partial_{xx} (\sin \theta_{h_0}) - \sin \theta_{h_0} = \Lambda (x) \sin \theta_{h_0},
    \end{equation*}
    and therefore $L_2 \sin \theta_{h_0} = 0$ thanks to \eqref{eq:def_Lambda_v1}.
\end{proof}

\begin{rem} \label{rem:invariance_link}
    We point out that the functions $\partial_x \theta_{h_0}$ and $\sin \theta_{h_0}$ are actually related to the invariances of the equation and of the energies $E$ and $E_{h_0}$. Indeed, $\partial_x \theta_{h_0}$ is related to the invariance by translation as there holds
    \begin{equation*}
        w_{h_0} (\cdot + \varepsilon) = w_{h_0} + \varepsilon \partial_x \theta_{h_0} n_{h_0} + O (\varepsilon^2),
    \end{equation*}
    and $\sin \theta_{h_0}$ is related to the invariance by rotation along $e_1$:
    \begin{equation*}
        R_{\varepsilon} w_{h_0} = 
        \begin{pmatrix}
            \cos \theta_{h_0} \\
            \sin \theta_{h_0} \cos \varepsilon \\
            \sin \theta_{h_0} \sin \varepsilon
        \end{pmatrix}
        = w_{h_0} + \varepsilon \sin \theta_{h_0} e_3 + O (\varepsilon^2).
    \end{equation*}
\end{rem}

\subsubsection{Abstract approach}

The functions identified in the kernel provide deeper insights than initially apparent. Given that these operators are in dimension $d=1$, the Sturm-Liouville theory offers a more precise understanding of the nature of the eigenvalues of the operators $L_1$ and $L_2$.
To this end, we introduce two abstract lemmas in the following subsection. These lemmas not only address the existence (or absence) of eigenvalues in the interval $(-\infty, 0)$ but also provide precise estimates regarding the lack of coercivity of these operators.

\begin{lem} \label{lem:abstract_schro_op_1}
    Let $L = - \partial_{xx} + V$ where $V \in L^\infty (\mathbb R)$ such that $\lim_{\pm \infty} V > 0$.
    Assume there exists $\phi \in H^2$ such that $L \phi = 0$ and $\phi > 0$ in $\mathbb R$.
    Then $\operatorname{ker} L = {\mathbb{R}} \phi$ and there exists $\lambda > 0$ such that for all $v \in H^1$,
    \begin{gather*}
        0 \leq \langle L v, v \rangle_{{H^{-1}, H^1}} \leq \frac{1}{\lambda} \norm{v}_{H^1}^2, \\
        \langle L v, v \rangle_{{H^{-1}, H^1}} \geq 4 \lambda \norm{v}_{H^1}^2 - \frac{1}{\lambda} \langle v, \phi \rangle^2,
    \end{gather*}
    and, for all $v \in H^2$,
    \begin{equation*}
        \norm{L v}_{L^2}^2 \geq 4 \lambda \norm{v}_{H^2}^2 - \frac{1}{\lambda} \langle v, \phi \rangle^2.
    \end{equation*}
\end{lem}

\begin{proof}
    See \cite[Lemma~C.1]{Cote_Ignat__stab_DW_LLG_DM}.
\end{proof}

\begin{lem} \label{lem:abstract_schro_op_2}
    Let $L = - \partial_{xx} + V$ where $V \in L^\infty (\mathbb R)$ such that $\ell \coloneqq \lim_{\pm \infty} V > 0$.
    Assume there exists $\phi \in H^2$ such that $L \phi = 0$ and that $\phi$ vanishes only once in $\mathbb R$. 
    Then it has a unique negative eigenvalue $\gamma < 0$, which is simple, and we denote $\psi \in H^2$ a normalized eigenfunction related to this eigenfunction: $\norm{\psi}_{L^2} = 1$ and $L \psi = \gamma \psi$.
    Its second eigenvalue is $0$, which is also simple: $\operatorname{ker} L = {\mathbb{R}} \phi$.
    As a consequence, there exists $\lambda > 0$ such that for all $v \in H^1$, there holds
    \begin{gather}
        \abs{\langle L v, v \rangle_{{H^{-1}, H^1}}} \leq \frac{1}{\lambda} \norm{v}_{H^1}^2, \label{eq:est_abs_L_1} \\
        \langle L v, v \rangle_{{H^{-1}, H^1}} \leq \frac{1}{\lambda} \norm{v - \langle v, \psi \rangle \psi}_{H^1}^2 - \lambda \langle v, \psi \rangle^2, \label{eq:est_abs_L_2}
    \end{gather}
    and, for all $v \in H^2$,
    \begin{equation*}
        \norm{L v}_{L^2}^2 \geq 4 \lambda \norm{v}_{H^2}^2 - \frac{1}{\lambda} \langle v, \phi \rangle^2.
    \end{equation*}
\end{lem}

\begin{proof}
    Using Sturm-Liouville theory on $\mathbb R$ (see \cite[Chapter~B.5]{Angulo_Pava_nonlin_disp}, especially the proof of Theorem B.61 which can be easily adapted to this case), the fact that $\phi$ vanishes once shows that $0$ is the second eigenvalue for $L$, there is a unique negative eigenvalue $\gamma < 0$ so that $L - \gamma$ is a positive operator, and both {$0$ and $\gamma$} are simple {eigenvalues}. It also means that there exists $\delta > 0$ such that every $\lambda_0 \in \sigma (L) \setminus \{ \gamma, 0 \}$ (where $\sigma (L)$ is the spectrum of $L$) satisfies $\lambda_0 \geq \delta$.

    Denoting $\psi$ a normalized eigenfunction related to $\gamma$, we can easily derive the first two estimates. From classic estimates, there holds $\psi \in H^2$. First, \eqref{eq:est_abs_L_1} is an easy consequence of the fact that $L$ is a Schrödinger operator with an $L^\infty$ potential. Therefore, for any $v \in H^1$, denoting $\tilde v = v - \langle v, \psi \rangle \psi$, we get:
    \begin{align*}
        \frac{1}{\lambda} \norm{\tilde v}_{{H^1}}^2 \geq \langle L \tilde v, \tilde v \rangle_{{H^{-1}, H^1}} &= \langle L v, v \rangle_{{H^{-1}, H^1}} - \langle v, \psi \rangle \langle L \psi, v \rangle_{{H^{-1}, H^1}} - \langle v, \psi \rangle \langle \psi, L v \rangle_{{H^{-1}, H^1}} + \langle v, \psi \rangle^2 \langle L \psi, \psi \rangle_{{H^{-1}, H^1}} \\
            &= \langle L v, v \rangle_{{H^{-1}, H^1}} - 2 \langle v, \psi \rangle \langle L \psi, v \rangle + \langle v, \psi \rangle^2 \langle L \psi, \psi \rangle \\
            &= \langle L v, v \rangle_{{H^{-1}, H^1}} - 2 \gamma \langle v, \psi \rangle^2 + \gamma \langle v, \psi \rangle^2 \\
            &= \langle L v, v \rangle_{{H^{-1}, H^1}} - \gamma \langle \psi, v \rangle^2,
    \end{align*}
    which gives \eqref{eq:est_abs_L_2} by the fact that $\gamma < 0$, up to decreasing $\lambda$.
    Last, let
    \begin{equation*}
        a \coloneqq \inf \{ \langle L v, v \rangle_{{H^{-1}, H^1}} \, | \, v \in H^1, \norm{v}_{L^2} = 1, \langle v, \phi \rangle = \langle v, \psi \rangle = 0 \}.
    \end{equation*}
    We aim to prove that $a > 0$.
    This quantity is actually related to the third Rayleigh quotient (see Lemma \ref{lem:link_a_rayleigh_quotient}):
    \begin{equation*}
        a = \mu_3 (L) \coloneqq \sup_{\psi_1, \psi_2} \inf_{u \in \operatorname{span} (\psi_1, \psi_2)^\perp \cap H^1 \setminus \{ 0 \}} \frac{\langle L u, u \rangle_{{H^{-1}, H^1}}}{\norm{u}_{L^2}^2}.
    \end{equation*}
    Moreover, we know that, either $\mu_3 (L)$ is the third eigenvalue (counted with multiplicity) of $L$ or the bottom of its essential spectrum.
    By the assumption $\lim_{\pm \infty} V = \ell > 0$, it is known that the essential spectrum of $L$ is $[\ell, \infty)$. On the other hand, $0$ is the second eigenvalue and it is simple. Thus, we conclude that $a > 0$.
    Therefore, if $v \in H^2$ with $\langle v, \psi \rangle = \langle v, \sin \theta_{h_0} \rangle = 0$, there holds
    \begin{equation*}
        \norm{L v}_{L^2} \norm{v}_{L^2} \geq \langle L v, v \rangle \geq a \norm{v}_{L^2}^2,
    \end{equation*}
    which yields $\norm{L v}_{L^2} \geq a \norm{v}_{L^2}$.
    Let $b \coloneqq \frac{a}{1 + 2 a + 2 \norm{V}_{L^\infty}} \in (0, \frac{1}{2})$. We thus get
    \begin{align*}
        \norm{L v}_{L^2}^2 &= 2 b \norm{\Delta v}_{L^2}^2 + (1 - 2 b) \langle L v, V v \rangle + 2 b \langle \Delta v, V v \rangle + 2 b \norm{V v}_{L^2}^2 \\
            &\geq 2 b \norm{\Delta v}_{L^2}^2 + a (1 - 2 b) \norm{v}_{L^2}^2 - 4 b \norm{\Delta v}_{L^2} \norm{V v}_{L^2} + 2 b \norm{V v}_{L^2}^2 \\
            &\geq 2 b \norm{\Delta v}_{L^2}^2 + a (1 - 2 b) \norm{v}_{L^2}^2 - b \norm{\Delta v}_{L^2}^2 - 4 b \norm{V v}_{L^2}^2 + 2 b \norm{V v}_{L^2}^2 \\
            &\geq b \norm{\Delta v}_{L^2}^2 + a (1 - 2 b) \norm{v}_{L^2}^2 - 2 b \norm{V v}_{L^2}^2 \\
            &\geq b \norm{\Delta v}_{L^2}^2 + \underbrace{(a (1 - 2 b) - 2 b \norm{V}_{L^\infty}^2)}_{= b} \norm{v}_{L^2}^2 = b (\norm{\Delta v}_{L^2}^2 + \norm{v}_{L^2}^2).
    \end{align*}
    As $\norm{\Delta v}_{L^2} + \norm{v}_{L^2}$ controls the $H^2$ norm, we can conclude that, for some $b' > 0$,
    \begin{equation*}
        \norm{L v}_{L^2}^2 \geq b' \norm{v}_{H^2}^2.
    \end{equation*}
    Let us emphasize that we can reduce $b'$, therefore we can assume that $b'$ is as small as we want it to be.
    Now, if we take a general $v \in H^2$, we can apply the previous estimate to $v - \langle v, \psi \rangle \psi - \frac{1}{\norm{\phi}_{L^2}^2} \langle v, \phi \rangle \phi$. Moreover,
    \begin{align*}
        \norm{L \Bigl( v - \langle v, \psi \rangle \psi - \frac{1}{\norm{\phi}_{L^2}^2} \langle v, \phi \rangle \phi \Bigr)}_{L^2}^2 &= \norm{L v - \gamma \langle v, \psi \rangle \psi}_{L^2}^2 \\
            &= \norm{L v}_{L^2}^2 - 2 \gamma \langle v, \psi \rangle \underbrace{\langle L v, \psi \rangle}_{= \langle v, L \psi \rangle = \gamma \langle v, \psi \rangle} + \gamma^2 \langle v, \psi \rangle^2 \\
            &= \norm{L v}_{L^2}^2 - \gamma^2 \langle v, \psi \rangle^2.
    \end{align*}
    On the other hand,
    \begin{equation*}
        \norm{v - \langle v, \psi \rangle \psi - \frac{1}{\norm{\phi}_{L^2}^2} \langle v, \phi \rangle \phi}_{H^2}^2 \geq \frac{1}{2} \norm{v}_{H^2}^2 - 2 \langle v, \psi \rangle^2 \norm{\psi}_{H^2}^2 - \frac{2}{\norm{\phi}_{L^2}^4} \langle v, \phi \rangle^2 \norm{\phi}_{H^2}^2,
    \end{equation*}
    where we have used $(a - b)^2 \geq \frac{1}{2} a^2 - b^2$ and $(b + c)^2 \leq 2 b^2 + 2 c^2$, so that $(a - b - c)^2 \geq \frac{1}{2} a^2 - 2 b^2 - 2 c^2$.
    Therefore, we obtain
    \begin{align*}
        \norm{L v}_{L^2}^2 &\geq \frac{b'}{2} \norm{v}_{H^2}^2 + \Bigl( \gamma^2 - 2 b' \norm{\psi}_{H^2}^2 \Bigr) \langle v, \psi \rangle^2 - \frac{2 b'}{\norm{\phi}_{L^2}^4} \langle v, \phi \rangle^2 \norm{\phi}_{H^2}^2 \\
            &\geq \frac{b'}{2} \norm{v}_{H^2}^2 - \frac{2 b'}{\norm{\phi}_{L^2}^4} \langle v, \phi \rangle^2 \norm{\phi}_{H^2}^2,
    \end{align*}
    by assuming $b' \leq \frac{\gamma^2}{2 \norm{\psi}_{H^2}^2}$.
\end{proof}

\subsubsection{Application to $L_1$ and $L_2$}

The previous lemmas, Lemma \ref{lem:abstract_schro_op_1} and Lemma \ref{lem:abstract_schro_op_2}, when applied to the two operators, offer valuable insights into the spectral properties of $L_1$ and $L_2$. However, the behaviour of the two functions in the kernel ($\partial_x \theta_{h_0}$ and $\sin \theta_{h_0}$) depends on the sign of $h_0$ (see Theorem \ref{th:stat_sol}). Therefore, we have to distinguish the two cases.

\paragraph{For $h_0 > 0$.}

In this case, $\partial_x \theta_{h_0}$ remains positive on $\mathbb R$, whereas $\sin \theta_{h_0}$ vanishes at $x=0$.

\begin{lem} \label{lem:schro_op_1}
    Assume that $h_0 > 0$.
    \begin{itemize}
        \item $L_1$ is a self-adjoint positive operator on $L^2$ with dense domain $H^2$, and has $0$ as first simple eigenvalue with eigenfunction $\partial_x \theta_{h_0} > 0$:
        \begin{equation} \label{eq:kernel_L1}
            L_1 \partial_x \theta_{h_0} = 0.
        \end{equation}
        As a consequence, there exists $\lambda_1 > 0$ such that for all $v \in H^1$,
        \begin{gather*}
            0 \leq \langle L_1 v, v \rangle_{{H^{-1}, H^1}} \leq \frac{1}{\lambda_1} \norm{v}_{H^1}^2, \\
            \langle L_1 v, v \rangle_{{H^{-1}, H^1}} \geq 4 \lambda_1 \norm{v}_{H^1}^2 - \frac{1}{\lambda_1} \langle v, \partial_x \theta_{h_0} \rangle^2,
        \end{gather*}
        and for all $v \in H^2$
        \begin{equation*}
            \norm{L_1 v}_{L^2}^2 \geq 4 \lambda_1 \norm{v}_{H^2}^2 - \frac{1}{\lambda_1} \langle v, \partial_x \theta_{h_0} \rangle^2.
        \end{equation*}
        \item $L_2$ is a self-adjoint operator on $L^2$ with dense domain $H^2$. It has a unique negative eigenvalue $\gamma_2 < 0$, which is simple, and we denote $\psi_{h_0} \in H^2$ a normalized eigenfunction related to this eigenvalue: $\norm{\psi_{h_0}}_{L^2} = 1$ and
        \begin{equation*}
            L_2 \psi_{h_0} = \gamma_2 \psi_{h_0}.
        \end{equation*}
        Its second eigenvalue is $0$, which is also simple, with eigenfunction $\sin \theta_{h_0}$:
        \begin{equation} \label{eq:kernel_L2}
            L_2 \sin \theta_{h_0} = 0.
        \end{equation}
        As a consequence, there exists $\lambda_2 > 0$ such that for all $v \in H^1$, there holds
        \begin{gather*}
            \langle L_2 v, v \rangle_{{H^{-1}, H^1}} \leq \frac{1}{\lambda_2} \norm{v - \langle v, \psi_{h_0} \rangle \psi_{h_0}}_{H^1}^2 - \lambda_2 \langle v, \psi_{h_0} \rangle^2, \\
            \abs{\langle L_2 v, v \rangle_{{H^{-1}, H^1}}} \leq \frac{1}{\lambda_2} \norm{v}_{H^1}^2,
        \end{gather*}
        and for all $v \in H^2$,
        \begin{equation*}
            \norm{L_2 v}_{L^2}^2 \geq 4 \lambda_2 \norm{v}_{H^2}^2 - \frac{1}{\lambda_2} \langle v, \sin \theta_{h_0} \rangle^2.
        \end{equation*}
    \end{itemize}
\end{lem}

\begin{proof}
    First of all, we emphasize that $L_1$ and $L_2$ are two self-adjoint Schrödinger operators, whose potentials have finite limits at infinity. Indeed, since $\lim_{- \infty} \theta_{h_0} = 0$ and $\lim_{+ \infty} \theta_{h_0} = 2 \pi$, we have
    \begin{equation*}
        \lim_{\pm \infty} 1 - 2 \sin^2 \theta_{h_0} + h_0 \cos \theta_{h_0} = \lim_{\pm \infty} 1 - 2 \sin^2 \theta_{h_0} + 3 h_0 \cos \theta_{h_0} - 2 h_0 = 1 + h_0.
    \end{equation*}
    We also have \eqref{eq:kernel_L1} by Proposition \ref{prop:kernel_schro_op}.
    Since $\partial_x \theta_{h_0} \in H^2$ and $\partial_x \theta_{h_0} > 0$ (see \eqref{eq:ODE_theta_2_v1}), we can apply Lemma \ref{lem:abstract_schro_op_1}.

    On the other hand, we have \eqref{eq:kernel_L2} by Proposition \ref{prop:kernel_schro_op}.
    However, $\sin \theta_{h_0}$ vanishes at $x=0$ and its sign changes at this point. Nonetheless, as $\theta_{h_0}$ is increasing with values in $(0, 2 \pi)$, we know that it is the only point where $\sin \theta_{h_0}$ vanishes.
    Thus, we can apply Lemma \ref{lem:abstract_schro_op_2}, leading to the conclusion.
\end{proof}

\begin{rem} \label{rem:eigenvalue_L2_h0_pos}
    Let us point out that, due to \eqref{eq:kernel_L1} and since $2 h_0 (1 - \cos \theta_{h_0}) > 0$ on $\mathbb R$, we obtain that
    \begin{equation*}
        \langle L_2 \partial_x \theta_{h_0}, \partial_x \theta_{h_0} \rangle = - 2 h_0 \int (1 - \cos \theta_{h_0}) \abs{\partial_x \theta_{h_0}}^2 \diff x < 0.
    \end{equation*}
    Moreover, since $L_1$ has no negative eigenvalue, it is positive, so that for any $f \in H^2$,
    \begin{equation*}
        \langle L_2 f, f \rangle = \langle L_1 f, f \rangle - 2 h_0 \int (1 - \cos \theta_{h_0}) \abs{f}^2 \diff x \geq - 2 h_0 \int (1 - \cos \theta_{h_0}) \abs{f}^2 \diff x \geq - 4 h_0 \norm{f}_{L^2}^2.
    \end{equation*}
    With its link with the first Rayleigh quotient (see Appendix \ref{sec:app_rayleigh}), we can probably estimate that the first eigenvalue has the same order as $- h_0$, at least when $h_0$ is small.
\end{rem}

\paragraph{For $h_0 \in (-1, 0)$.}

If $h_0 \in (-1, 0)$, conversely, $\sin \theta_{h_0}$ remains positive on $\mathbb R$, whereas $\partial_x \theta_{h_0}$ vanishes at $x=0$.

\begin{lem} \label{lem:schro_op_2}
    Assume that $h_0 \in (-1, 0)$.
    \begin{itemize}
        \item $L_1$ is a self-adjoint operator on $L^2$ with dense domain $H^2$. It has a unique negative eigenvalue $\gamma_1 < 0$, which is simple, and we denote $\psi_{h_0} \in H^2$ a normalized eigenfunction related to this eigenvalue: $\norm{\psi_{h_0}}_{L^2} = 1$ and
        \begin{equation*}
            L_1 \psi_{h_0} = \gamma_1 \psi_{h_0}.
        \end{equation*}
        Its second eigenvalue is $0$, which is also simple, with eigenfunction $\partial_x \theta_{h_0}$:
        \begin{equation} \label{eq:kernel_L1_2}
            L_1 \partial_x \theta_{h_0} = 0.
        \end{equation}
        As a consequence, there exists $\lambda_1 > 0$ such that for all $v \in H^1$,
        \begin{gather*}
            \langle L_1 v, v \rangle \leq \frac{1}{\lambda_1} \norm{v - \langle v, \psi_{h_0} \rangle \psi_{h_0}}_{H^1} - \lambda_1 \langle v, \psi_{h_0} \rangle^2, \\
            \abs{\langle L_1 v, v \rangle} \leq \frac{1}{\lambda_1} \norm{v}_{H^1}^2,
        \end{gather*}
        and for all $v \in H^2$,
        \begin{equation*}
            \norm{L_1 v}_{L^2}^2 \geq 4 \lambda_1 \norm{v}_{H^2}^2 - \frac{1}{\lambda_1} \langle v, \sin \theta_{h_0} \rangle^2.
        \end{equation*}
        \item $L_2$ is a self-adjoint positive operator on $L^2$ with dense domain $H^2$, and has $0$ as first simple eigenvalue with eigenfunction $\sin \theta_{h_0} > 0$:
        \begin{equation} \label{eq:kernel_L2_2}
            L_2 \sin \theta_{h_0} = 0.
        \end{equation}
        As a consequence, there exists $\lambda_2 > 0$ such that for all $v \in H^1$, there holds
        \begin{gather*}
            0 \leq \langle L_2 v, v \rangle \leq \frac{1}{\lambda_2} \norm{v}_{H^1}^2, \\
            \langle L_2 v, v \rangle \geq 4 \lambda_2 \norm{v}_{H^1}^2 - \frac{1}{\lambda_2} \langle v, \partial_x \theta_{h_0} \rangle^2,
        \end{gather*}
        and for all $v \in H^2$
        \begin{equation*}
            \norm{L_2 v}_{L^2}^2 \geq 4 \lambda_2 \norm{v}_{H^2}^2 - \frac{1}{\lambda_2} \langle v, \partial_x \theta_{h_0} \rangle^2.
        \end{equation*}
    \end{itemize}
\end{lem}

\begin{proof}
    Once again, $L_1$ and $L_2$ are two self-adjoint Schrödinger operators, whose potentials have finite limits at infinity. Indeed, since $\lim_{\pm \infty} \theta_{h_0} = 0$, we have
    \begin{equation*}
        \lim_{\pm \infty} 1 - 2 \sin^2 \theta_{h_0} + h_0 \cos \theta_{h_0} = \lim_{\pm \infty} 1 - 2 \sin^2 \theta_{h_0} + 3 h_0 \cos \theta_{h_0} - 2 h_0 = 1 + h_0.
    \end{equation*}
    We also have \eqref{eq:kernel_L1_2} and \eqref{eq:kernel_L2_2} by Proposition \ref{prop:kernel_schro_op}.
    Since $\sin \theta_{h_0} \in H^2$ and $\sin \theta_{h_0} > 0$ (see \eqref{eq:ODE_theta_2_v1}), we can apply Lemma \ref{lem:abstract_schro_op_1} to $L_2$.
    %
    On the other hand, $\partial_x \theta_{h_0}$ vanishes at $x=0$. Nonetheless, we also know that $\partial_x \theta_{h_0} > 0$ on $(- \infty, 0)$ and $\partial_x \theta_{h_0} < 0$ on $(0, \infty)$.
    Thus, we can apply Lemma \ref{lem:abstract_schro_op_2} to $L_1$.
\end{proof}

\begin{rem} \label{rem:eigenvalue_L1_h0_neg}
    Let us point out that, due to \eqref{eq:kernel_L2} and since $2 h_0 (1 - \cos \theta_{h_0}) < 0$ on $\mathbb R$, we obtain that
    \begin{equation*}
        \langle L_1 \sin \theta_{h_0}, \sin \theta_{h_0} \rangle = 2 h_0 \int (1 - \cos \theta_{h_0}) \abs{\sin \theta_{h_0}}^2 \diff x < 0.
    \end{equation*}
    Moreover, since $L_2$ has no negative eigenvalue, it is positive, so that for any $f \in H^2$,
    \begin{equation*}
        \langle L_1 f, f \rangle = 2 h_0 \int (1 - \cos \theta_{h_0}) \abs{f}^2 \diff x + \langle L_2 f, f \rangle \geq 2 h_0 \int (1 - \cos \theta_{h_0}) \abs{f}^2 \diff x \geq - 4 h_0 \norm{f}_{L^2}^2.
    \end{equation*}
    With its link with the first Rayleigh quotient (see Appendix \ref{sec:app_rayleigh}), we can probably estimate that the first eigenvalue is of order of $h_0$, at least when $h_0$ is small.
\end{rem}

\subsection{The associated basis} \label{sec:associated_basis}

\subsubsection{Definition and equivalence}

Let us introduce the following vector, depending on $x$:
\begin{equation*}
    n_{h_0} = 
    \begin{pmatrix}
        - \sin \theta_{h_0} \\
        \cos \theta_{h_0} \\
        0
    \end{pmatrix}.
\end{equation*}

Similarly to \cite{Cote_Ignat__stab_DW_LLG_DM}, such a vector allows us to introduce the frame $(w_{h_0} (x), n_{h_0} (x), e_3)$. This frame is a direct {orthonormal} basis for any $x \in \mathbb R$, and it is better adapted to a $\mathbb{S}^2$-valued magnetisation $m$ close to a stationary solution $w_{h_0}$.
Indeed, if $m = w_{h_0} + \eta \in \mathbb S^2$ with $\eta$ small (in some sense), we can decompose $\eta$ in this frame: $\eta = \mu w_{h_0} + \nu n_{h_0} + \rho e_3$.
Then $\mu$ is quadratic in $\eta$, whose norm is thus equivalent to that of $\nu$ and $\rho$.
The precise statement is as follows.

\begin{lem} \label{lem:expand_eta}
    {Let $h_0 \in (-1, 0)$ or $h_0 > 0$.}
    There exists $\delta_3 > 0$ and $C_2 > 0$ such that the following holds.
    Let $m = w_{h_0} + \eta: \mathbb{R} \rightarrow \mathbb{S}^2$ be such that
    \begin{equation*}
        \norm{\eta}_{H^1} < \delta_3.
    \end{equation*}
    We decompose $\eta$ in the $(w_{h_0}, n_{h_0}, e_3)$ basis pointwise in $x$:
    \begin{equation*}
        \eta = \mu w_{h_0} + \nu n_{h_0} + \rho e_3 \quad
        \text{where} \quad
        \mu \coloneqq \eta \cdot w_{h_0}, \quad
        \nu = \eta \cdot n_{h_0}, \quad
        \rho = \eta \cdot e_3.
    \end{equation*}
    Then $\mu, \nu, \rho \in H^1$, with
    \begin{equation} \label{eq:ineg1}
        \norm{\mu}_{H^1} \leq C_2 \norm{\eta}_{H^1}^2, \qquad
        \frac{1}{C_2} \norm{\eta}_{H^1} \leq \norm{(\nu, \rho)}_{H^1} \leq C_2 \norm{\eta}_{H^1}.
    \end{equation}
    In particular, $\mu = - \frac{1}{2} \abs{\eta}^2 = - \frac{1}{2} (\nu^2 + \rho^2) + O_0^4 (\eta)$, and $\mu \in L^1$.
    We also have
    \begin{equation} \label{eq:eta_1_L1}
        \eta \cdot e_1 = \cos \theta_{h_0} \, \mu - \sin \theta_{h_0} \, \nu \in L^1.
    \end{equation}
    If furthermore $\eta \in H^2$, then $\mu, \nu, \rho \in H^2$ and
    \begin{equation*} \label{eq:equiv_eta_nu2}
        \norm{(\nu, \rho)}_{H^2} \leq C_2 \norm{\eta}_{H^2}
    \end{equation*}
    Last, there also hold
    \begin{equation*}
        \rho \sin \theta_{h_0} = \eta \cdot (e_1 \wedge w_{h_0}), \qquad
        \nu \, \partial_x \theta_{h_0} = \eta \cdot \partial_x w_{h_0}.
    \end{equation*}
\end{lem}

\begin{proof}
    The proof is similar to the first step of the proof of \cite[Proposition~4.16]{Cote_Ignat__stab_DW_LLG_DM}.
    First, the relations between $\mu$, $\nu$, $\rho$ and $\eta$
    along with Lemma \ref{lem:est_negl_terms} and the regularity and integrability result on $\theta_{h_0}$ in Theorem \ref{th:stat_sol} give
    \begin{equation*}
        \norm{\mu}_{H^k} + \norm{\nu}_{H^k} + \norm{\rho}_{H^k} \lesssim \norm{\eta}_{H^k}.
    \end{equation*}
    On the other side, $\eta = \mu w_{h_0} + \nu n_{h_0} + \rho e_3$ and therefore
    \begin{equation*}
        \norm{\eta}_{H^k} \lesssim \norm{\mu}_{H^k} + \norm{\nu}_{H^k} + \norm{\rho}_{H^k}.
    \end{equation*}
    The equality $\mu = {-} \frac{1}{2} \abs{\eta}^2$ comes from the expansion of $\abs{w_{{h_0}} + \eta}^2 = 1$, which gives the first inequality of \eqref{eq:ineg1} with Lemma \ref{lem:est_negl_terms}. It also gives the fact that $\mu \in L^1$, and then \eqref{eq:eta_1_L1} is deduced by the relation
    \begin{equation*}
        e_1 = \cos \theta_{h_0} w_{h_0} - \sin \theta_{h_0} n_{h_0}.
    \end{equation*}
    We conclude that $\eta \cdot e_1 \in L^1$ thanks to $\mu \in L^1$, $\sin \theta_{h_0} \in L^2$ and $\nu \in L^2$.
    As soon as $\norm{\eta}_{H^1}$ is small enough the second inequality in \eqref{eq:ineg1} is then straightforward.
    Eventually, the last equality comes from the formulas
    \begin{equation*}
        \partial_x w_{h_0} = \partial_x \theta_{h_0} \, n_{h_0}, \qquad
        e_1 \wedge w_{h_0} = \sin \theta_{h_0} \, e_3. \qedhere
    \end{equation*}
\end{proof}

\subsubsection{Expansions of some expressions} \label{sec:expand_expressions}

In a similar vein to the approach taken in \cite{Cote_Ignat__stab_DW_LLG_DM}, we now turn our attention to a lemma providing expansions in $\eta$ (or equivalently, in $\nu$ and $\rho$) for several crucial expressions. 

\begin{lem}
    Let $m \in \mathcal{H}^2$ satisfying the assumptions of Lemma \ref{lem:expand_eta}, and let $\eta, \mu, {\nu}, \rho$ as defined in the same Lemma. Then there holds
    \begin{gather}
        \begin{aligned}
            \delta E (\eta) = O_2^2 (\eta) &+ ( \partial_{xx} \theta_{h_0} \nu + 2 \partial_x \theta_{h_0} \partial_x \nu + \cos \theta_{h_0} \sin \theta_{h_0} \nu ) \, w_{h_0} \\
                &+ ( - \partial_{xx} \nu + (\partial_x \theta_{h_0})^2 \nu + \cos^2 \theta_{h_0} \nu ) \, n_{h_0} \\
                &+ (- \partial_{xx} \rho + \rho) \, e_3,
        \end{aligned} \label{eq:delta_E_eta} \\
        \begin{aligned}
            \delta E (\eta) \cdot \eta = O_2^3 (\eta) &+ (- \partial_{xx} \nu + (1 + 2 h_0 (1 - \cos \theta_{h_0})) \nu) \, \nu \\
                &+ (- \partial_{xx} \rho + \rho) \, \rho.
        \end{aligned} \label{eq:delta_E_eta_eta}
    \end{gather}
\end{lem}

\begin{proof}
    First, from the definition of $w_{h_0}$ and $n_{h_0}$, there holds $\partial_x w_{h_0} = \partial_x{\theta_{h_0}} n_{h_0}$ and $\partial_x n_{h_0} = - \partial_x \theta_{h_0} w_{h_0}$. Therefore, we can compute the derivatives of $\eta = \mu \, w_{h_0} + \nu \, n_{h_0} + \rho \, e_3$:
    \begin{gather}
        \partial_x \eta = (\partial_x \mu - \partial_x \theta_{h_0} \nu) \, w_{h_0} + (\partial_x \nu + \partial_x \theta_{h_0} \mu) \, n_{h_0} + \partial_x \rho \, e_3, \notag \\
        \begin{multlined}
            \partial_{xx} \eta = (\partial_{xx} \, \mu - \partial_{xx} \theta_{h_0} \, \nu - 2 \partial_x \theta_{h_0} \partial_x \nu - (\partial_x \theta_{h_0})^2 \mu) w_{h_0} \\
            \begin{aligned}
                &+ (\partial_{xx} \nu + \partial_{xx} \theta_{h_0} \, \mu + \partial_x \theta_{h_0} \partial_x \mu - (\partial_x \theta_{h_0})^2 \nu) \, n_{h_0} \\
                &+ \partial_{xx} \rho \, e_3.
            \end{aligned}
        \end{multlined} \label{eq:first_comp_dxx_eta}
    \end{gather}
    Moreover, we see that $e_2 = \sin \theta_{h_0} \, w_{h_0} + \cos \theta_{h_0} \, n_{h_0}$. Thus, $\eta_2 = \eta \cdot e_2 = \sin \theta_{h_0} \, \mu + \cos \theta_{h_0} \, \nu$ and therefore
    \begin{equation*}
        \eta_2 e_2 = (\sin \theta_{h_0} \, \mu + \cos \theta_{h_0} \, \nu) \sin \theta_{h_0} \, w_{h_0} + ( \sin \theta_{h_0} \, \mu + \cos \theta_{h_0} \, \nu) \cos \theta_{h_0} \, n_{h_0}.
    \end{equation*}
    Last, $\eta_3 e_3 = \rho e_3$.
    Thus, using the fact that $\mu = - \frac{1}{2} \abs{\eta}^2 = O_0^2 (\eta)$ (see Lemma \ref{lem:expand_eta}) and Lemma \ref{lem:est_negl_terms}, all the terms involving $\mu$ in \eqref{eq:first_comp_dxx_eta} are $O_2^2 (\eta)$. Then, we obtain \eqref{eq:delta_E_eta} by the expression of $\delta E$ in \eqref{eq:def_delta_E}. As for \eqref{eq:delta_E_eta_eta}, it is obtained by expanding $\delta E (\eta)$ with \eqref{eq:delta_E_eta}, by neglecting once again all the terms involving $\mu$ and by using \eqref{eq:ODE_theta_2_v1}.
\end{proof}

\subsubsection{Modulation and decomposition of the magnetisation}

In our framework, decomposing directly $m$ with the aforementioned associated basis is not the best option. Indeed, in order to prove Theorem \ref{th:instability}, we need to prove not only that $w_{h_0}$ is (nonlinearly) unstable, but also its orbit. Therefore, formally, this decomposition should be used on $(-g_m).w_{h_0}$ where $g_m$ minimizes the following minimization problem:
\begin{equation*}
    \min_{g \in G} \norm{m - g.w_{h_0}}_{H^1}.
\end{equation*}
However, such a $g_m$ is hard to track down, and might not be smooth enough in time. Therefore, we will more generally use this decomposition on $(-g).m$ for some $g$ well chosen. To be more specific, we would like $g$ to be near $g_m$, to be smooth enough and to make vanish the bad negative terms in the estimations of Lemma \ref{lem:schro_op_1} or Lemma \ref{lem:schro_op_2}.
This is the goal of the following result of modulation and decomposition of a magnetization $m$ near $w_{h_0}$.

\begin{lem} \label{lem:modulation}
    There exists $\delta_1 > 0$ and $C_1 \geq 1$ such that the following holds. For all $w \in \mathcal{M}$ where $\mathcal{M} \coloneqq (e_1 + H^1) \cap \mathscr{C} (\mathbb R, \mathbb S^2)$ such that
    \begin{equation} \label{eq:m_close_m_h_0}
        \norm{w - w_{h_0}}_{H^1} \leq \delta_1,
    \end{equation}
    there exists a gauge $g = (y, \phi) \in G$ such that
    the following orthogonality constraints hold:
    \begin{equation*}
        \int \eta \cdot \partial_x w_{h_0} \diff x = \int \eta \cdot (e_1 \wedge w_{h_0}) \diff x = 0,
    \end{equation*}
    where $\eta \in H^1$ is defined by
    \begin{equation*}
        \eta \coloneqq (- g).w - w_{h_0},
    \end{equation*}
    so that $w = g.(w_{h_0} + \eta)$.
    Furthermore,
    \begin{equation*}
        \abs{g} + \norm{\eta}_{H^1} \leq C_1 \norm{w - w_{h_0}}_{H^1}.
    \end{equation*}
    The gauge $g$ is unique with the above properties and the map $w \in \mathcal{M} \mapsto g \in G$ is of class $\mathscr{C}^\infty$ in a neighborhood of $w_{h_0}$.
\end{lem}

Of course, such a lemma can be applied to $(-g_m).m$, which leads to the requested properties. This lemma is proved in Appendix \ref{sec:proof_mod_lem}.

\section{The total energy: {dissipation} and expansion} \label{sec:mod_en_disp_exp}

We first recall the definition of the total energy $E_{h_0}$ given in \eqref{eq:mod_en}
\begin{equation*}
    E_{h_0} (m) \coloneqq E (m) - h_0 \int (m_1 - 1) \diff x.
\end{equation*}
We also recall that its variation is related to the variation $\delta E$ of the initial energy $E$ and to the effective magnetic field $H$:
\begin{equation*}
    \delta E_{h_0} (m) = \delta E (m) - h_0 e_1 = - H(m).
\end{equation*}
In particular, $\delta E_{h_0} (m)$ and $H(m)$ are in $L^2 + L^\infty$ if $m \in \mathcal{H}^2$, and in $H^{-1} + L^\infty$ if $m$ is only $\mathcal{H}^1$.
Another property of $E_{h_0}$ can be deduced from \eqref{eq:delta_E_h0}: this total energy is decreasing through the flow of \eqref{eq:llg}.

\begin{prop} \label{prop:evolution_E_h_0_1}
    With the assumption of Theorem \ref{th:lwp}, assuming furthermore that $h (t) = h_0$ is constant, $m_0 \in \mathcal{H}^s$ for some $s \geq 2$ and $\lim_{\pm \infty} m_0 = e_1$, then for all $t \in [0, T_+ (m))$,
    \begin{equation*}
        \frac{\diff}{\diff t} E_{h_0} (m (t)) = - \alpha \int \abs{m \wedge H (m)}^2 \diff x.
    \end{equation*}
\end{prop}

\begin{proof}
    From \eqref{eq:llg}, there holds
    \begin{align*}
        \partial_t m_1 &= (m \wedge H(m)) \cdot e_1 - \alpha \Bigl( m \wedge (m \wedge H(m)) \Bigr) \cdot e_1, \\
        \partial_t (m_1 - 1) &= - H(m) \cdot (m \wedge e_1) + \alpha (m \wedge e_1) \cdot (m \wedge H(m)) \\
            &= - \partial_{xx} m \cdot (m \wedge e_1) - \alpha (m \wedge e_1) \cdot (m \wedge \delta E (m)) + h_0 \alpha (m \wedge e_1) \cdot (m \wedge e_1).
    \end{align*}
    As we know that $m \in \mathcal{H}^2$ and thus $\delta E (m) \in L^2$ on $[0, T_+ (m))$, with uniform bounds on every compact subintervals, we can easily infer that all the terms in the right-hand side of the last equality are in $L^1$. Therefore, since $m_1 (t) - 1 \in L^1$ for all $t \geq 0$, we get $m_1 - 1 \in C^1 ([0, T_+ (m)), L^1)$ and
    \begin{align*}
        \frac{\diff}{\diff t} \int (m_1 - 1) \diff x &= - \int \partial_{xx} m \cdot (m \wedge e_1) \diff x - \alpha \int (m \wedge e_1) \cdot (m \wedge \delta E (m)) \diff x + h_0 \alpha \int \abs{m \wedge e_1}^2 \diff x \\
            &= \int \partial_{x} m \cdot ( \partial_x m \wedge e_1) \diff x - \alpha \int (m \wedge e_1) \cdot (m \wedge \delta E (m)) \diff x + h_0 \alpha \int \abs{m \wedge e_1}^2 \diff x \\
            &= - \alpha \int (m \wedge e_1) \cdot (m \wedge \delta E (m)) \diff x + h_0 \alpha \int \abs{m \wedge e_1}^2 \diff x.
    \end{align*}
    The conclusion is reached by using this equality into the relation between $E (m)$ and $E_{h_0} (m)$, which yields
    \begin{equation*}
        \frac{\diff}{\diff t} E_{h_0} (m (t)) = \frac{\diff}{\diff t} E (m (t)) - h_0 \frac{\diff}{\diff t} \int (m_1 - 1) \diff x,
    \end{equation*}
    and using the fact that $ |\delta E(m)|^2 - |m \cdot \delta E(m)|^2 = \abs{\delta E (m) \wedge m}^2$.
\end{proof}

This property shows that $E_{h_0}$ is therefore better suited to describe the behavior of any solution near its critical point $w_{h_0}$, and more generally near the orbit of $w_{h_0}$ since $E_{h_0}$ is invariant by translation and rotation around $e_1$.
The results in Theorem \ref{th:stat_sol} and in Section \ref{sec:associated_basis} allow to expand the total energy $E_{h_0}$ and its {dissipation} through the flow of the equation (see Proposition \ref{prop:evolution_E_h_0_1}) near $w_{h_0}$ in both cases $h_0 > 0$ and $h_0 \in (-1, 0)$.
First, we begin by the expansion of the total energy.

\begin{lem} \label{lem:expand_energy}
    Let $m \in \mathcal{H}^2$ satisfying the assumptions of Lemma \ref{lem:expand_eta}, and let $\eta, \mu, {\nu}, \rho$ as defined in the same Lemma. Then there exists $C > 0$ such that
    \begin{equation} \label{eq:est_en_mod}
        \abs{E_{h_0} (m) - E_{h_0} (w_{h_0}) - \frac{1}{2} \Bigl( \langle L_1 \nu, \nu \rangle + \langle L_2 \rho, \rho \rangle \Bigr)} \leq C \norm{\eta}_{H^1}^3.
    \end{equation}
\end{lem}

\begin{proof}
    By the definition of $E$ and $\delta E$, we know that $E (m) = \frac{1}{2} \int \delta E (m) \cdot m \diff x$.
    Since $m = w_{h_0} + \eta$, there holds
    \begin{align*}
        E_{h_0} (m) &= E_{h_0} (w_{h_0} + \eta) = E (w_{h_0} + \eta) - h_0 \int ( w_{h_0} \cdot e_1 - 1 + \eta \cdot e_1) \diff x \\
            &= \begin{multlined}[t][12cm]
                    \underbrace{\frac{1}{2} \int \Bigl( \delta (w_{h_0}) \cdot w_{h_0} - 2 h_0 (w_{h_0} \cdot e_1 - 1) \Bigr) \diff x}_{I_1} \\
                    \begin{aligned}
                        &+ \frac{1}{2} \Bigl( \underbrace{\int \delta E (w_{h_0}) \cdot \eta \diff x + \int \delta E (\eta) \cdot w_{h_0} \diff x - 2 h_0 \int \eta \cdot e_1 \diff x}_{I_2} \Bigr) \\
                        &+ \frac{1}{2} \underbrace{\int \delta E (\eta) \cdot \eta \diff x}_{I_3}.
                    \end{aligned}
                \end{multlined}
    \end{align*}
    First, there holds $I_1 = E_{h_0} (w_{h_0})$.
    As for $I_2$, we emphasize that $\eta \cdot e_1 \in L^1$ by Lemma \ref{lem:expand_eta}, thus the last integral is well-defined. By grouping all the integrals, using \eqref{eq:expr_H_w_h_0_v1} and that $\delta E$ is a linear operator which is self-adjoint with respect to the $L^2$ scalar product, we get
    \begin{align*}
        I_2 &= \int \Bigl( 2 \delta E (w_{h_0}) \cdot \eta - 2 h_0 \eta \cdot e_1 \Bigr) \diff x \\
            &= - 2 \int H(w_{h_0}) \cdot \eta \diff x \\
            &= - 2 \int \Lambda (x) w_{h_0} \cdot \eta \diff x \\
            &= 2 \int (2 \sin^2 \theta_{h_0} - 3 h_0 \cos \theta_{h_0} + 2 h_0) \mu \diff x \\
            &= - \int (2 \sin^2 \theta_{h_0} - 3 h_0 \cos \theta_{h_0} + 2 h_0) (\nu^2 + \rho^2) \diff x + \int O_0^4 (\eta) \diff x.
    \end{align*}
    Last, using \eqref{eq:delta_E_eta_eta}, $I_3$ can be expanded as follows:
    \begin{equation*}
        I_3 = \int (- \partial_{xx} \nu + (1 + 2 h_0 (1 - \cos \theta_{h_0})) \nu) \, \nu \diff x + \int (- \partial_{xx} \rho + \rho) \, \rho \diff x + \int O_2^3 (\eta) \diff x.
    \end{equation*}
    The last term for both $I_2$ and $I_3$ can be estimated thanks to Lemma \ref{lem:est_negl_terms}. Summing the other terms {leads} to \eqref{eq:est_en_mod}.
\end{proof}

Then, we continue with the expansion of the {dissipation}.

\begin{lem} \label{lem:expand_dissipation}
    Let $m \in \mathcal{H}^2$ satisfying the assumptions of Lemma \ref{lem:expand_eta}, and let $\eta, \mu, {\nu}, \rho$ as defined in the same Lemma. Then there exists $C > 0$ such that
    \begin{equation} \label{eq:est_diss}
        \abs{\int \abs{m \wedge H(m)}^2 \diff x - ( \norm{L_1 \nu}_{L^2}^2 + \norm{L_2 \rho}_{L^2}^2 )} \leq C \norm{\eta}_{H^1} \norm{\eta}_{H^2}^2. 
    \end{equation}
\end{lem}

\begin{proof}
    First, since $H(m) = - \delta E (m) + h_0 e_1$ with $\delta E$ linear in $m = w_{h_0} + \eta$, then
    \begin{equation*}
        m \wedge H(m) = w_{h_0} \wedge H(w_{h_0}) + \eta \wedge H (w_{h_0}) - w_{h_0} \wedge \delta E (\eta) - \eta \wedge \delta E (\eta).
    \end{equation*}
    Since $H(w_{h_0}) = \Lambda (x) w_{h_0}$, the first term of this expansion vanishes.
    Therefore, we can expand $\abs{m \wedge H(m)}^2$ as follows:
    \begin{align*}
        \abs{m \wedge H(m)}^2 &= \abs{\eta \wedge H (w_{h_0}) - w_{h_0} \wedge \delta E (\eta) - \eta \wedge \delta E (\eta)}^2 \\
            &= 
                \begin{multlined}[t][7cm]
                    \abs{\Lambda (x) \eta \wedge w_{h_0} - w_{h_0} \wedge \delta E (\eta)}^2 \\
                    \begin{aligned}
                        &+ 2 (\eta \wedge \delta E (\eta)) \cdot (w_{h_0} \wedge \delta E (\eta) - \Lambda (x) \eta \wedge w_{h_0}) \\
                        &+ \abs{\eta \wedge \delta E (\eta)}^2.
                    \end{aligned}
                \end{multlined}
    \end{align*}
    From \eqref{eq:delta_E_eta}, we have
    \begin{align*}
        w_{h_0} \wedge \delta E (\eta) = O_2^2 (\eta) &- (- \partial_{xx} \rho + \rho) \, n_{h_0} \\
            &+ (- \partial_{xx} \nu + (\partial_x \theta_{h_0})^2 \nu + \cos^2 \theta_{h_0} \nu ) \, e_3.
    \end{align*}
    On the other hand, $\eta \wedge w_{h_0} = \rho n_{h_0} - \nu e_3$, so that, using \eqref{eq:ODE_theta_2_v1},
    \begin{align*}
        \Lambda (x) \eta \wedge w_{h_0} - w_{h_0} \wedge \delta E (\eta) = O_2^2 (\eta) &+ \Bigl( - \partial_{xx} \rho + (1 + \Lambda (x)) \rho \Bigr) \, n_{h_0} \\
            &- \Bigl( - \partial_{xx} \nu + \Bigl( (\partial_x \theta_{h_0})^2 + \cos^2 \theta_{h_0} + \Lambda (x) \Bigr) \nu \Bigr) \, e_3.
    \end{align*}
    Using \eqref{eq:def_Lambda_v1} in Theorem \ref{th:stat_sol}, we get
    \begin{align*}
        1 + \Lambda (x) &= 1 - 2 \sin^2 \theta_{h_0} + 3 h_0 \cos \theta_{h_0} - 2 h_0, \\
        (\partial_x \theta_{h_0})^2 + \cos^2 \theta_{h_0} + \Lambda (x) &= 1 - 2 \sin^2 \theta_{h_0} + h_0 \cos \theta_{h_0}.
    \end{align*}
    Therefore, we get
    \begin{equation*}
        \Lambda (x) \eta \wedge w_{h_0} - w_{h_0} \wedge \delta E (\eta) = O_2^2 (\eta) - L_2 \rho \, n_{h_0} + L_1 \nu \, e_3.
    \end{equation*}
    This leads to
    \begin{equation*}
        \abs{\Lambda (x) \eta \wedge w_{h_0} - w_{h_0} \wedge \delta E (\eta)}^2 = \abs{L_2 \rho}^2 + \abs{L_1 \nu}^2 + O_4^3 (\eta) + O_4^4 (\eta).
    \end{equation*}
    Last, we also have $(\eta \wedge H (\eta)) \cdot (w_{h_0} \wedge H (\eta) + \Lambda (x) \eta \wedge w_{h_0}) = O_4^3 (\eta)$ and $\abs{\eta \wedge \delta E (\eta)}^2 = O_4^4 (\eta)$. \eqref{eq:est_diss} is then obtained by combining and integrating all these equalities and by using Lemma \ref{lem:est_negl_terms} for any $\int O_4^\ell (\eta) \diff x$.
\end{proof}

\section{Proof of the instability} \label{sec:instability}

We now delve into the proof of Theorem \ref{th:instability}.
In brief, the idea of the proof consists in tracking down the evolution of the total energy. We know it is decreasing thanks to Proposition \ref{prop:evolution_E_h_0_1} (which is recalled in Section \ref{sec:evo_mod_en}), but the evolution can be better described thanks to the expansions of the energy (Lemma \ref{lem:expand_energy}) and of the {dissipation} (Lemma \ref{lem:expand_dissipation}). As soon as this energy is smaller than the energy of the stationary solution, it decreases exponentially in time and this estimate can be pushed back to the distance between the solution and $g.w_{h_0}$ for any $g \in G$.



\subsection{Summary of the main properties}

\subsubsection{Evolution of the total energy} \label{sec:evo_mod_en}


First, we recall here the evolution of the total energy, \textit{i.e.} Proposition \ref{prop:evolution_E_h_0_1}.

\begin{prop} \label{prop:evolution_E_h_0}
    With the assumption of Theorem \ref{th:lwp}, assuming furthermore that $h (t) = h_0$ is constant, $m_0 \in \mathcal{H}^s$ for some $s \geq 2$ and $\lim_{\pm \infty} m_0 = e_1$, then for all $t \in [0, T_+ (m))$,
    \begin{equation*}
        \frac{\diff}{\diff t} E_{h_0} (m (t)) = - \alpha \int \abs{m \wedge H (m)}^2 \diff x.
    \end{equation*}
\end{prop}

\subsubsection{Expansion of the energy for $h_0 > 0$}

Now, we recall the expansion of the energy in terms of $\nu$ and $\rho$, starting with the case $h_0 > 0$.

\begin{lem} \label{lem:summary_expand_energy}
    Assume that $h_0 > 0$.
    There exists $C, \lambda > 0$ such that, for any $m = w_{h_0} + \eta \in \mathcal{H}^2$ satisfying
    \begin{equation*}
        \norm{\eta}_{H^1} < \delta_3,
    \end{equation*}
    {where $\delta_3 > 0$ is defined in Lemma \ref{lem:expand_eta},} there holds
    \begin{gather*}
        E_{h_0} (m) - E_{h_0} (w_{h_0}) \leq C \Bigl( \norm{\nu}_{H^1}^2 + \norm{\rho - \langle \rho, \psi_{h_0} \rangle \psi_{h_0}}_{H^1}^2 + \norm{\eta}_{H^1}^3 \Bigr) - \lambda \langle \rho, \psi_{h_0} \rangle^2, \\
        \abs{E_{h_0} (m) - E_{h_0} (w_{h_0})} \leq \frac{1}{\lambda} \norm{\eta}_{H^1}^2,
    \end{gather*}
    where $\nu, \rho$ are defined in Lemma \ref{lem:expand_eta} and $\psi_{h_0}$ in Lemma \ref{lem:schro_op_1}.
\end{lem}

This lemma is straightforwardly obtained by Lemma \ref{lem:expand_energy} and Lemma \ref{lem:schro_op_1}.

\subsubsection{For $h_0 \in (-1, 0)$}

Now, we recall the same result for the case $h_0 \in (-1, 0)$.

\begin{lem} \label{lem:summary_expand_energy_2}
    Assume that $h_0 \in (-1, 0)$.
    There exists $C, \lambda > 0$ such that, for any $m = w_{h_0} + \eta \in \mathcal{H}^2$ satisfying
    \begin{equation*}
        \norm{\eta}_{H^1} < \delta_3,
    \end{equation*}
    {where $\delta_3 > 0$ is defined in Lemma \ref{lem:expand_eta},} there holds
    \begin{gather*}
        E_{h_0} (m) - E_{h_0} (w_{h_0}) \leq C \Bigl( \norm{\rho}_{H^1}^2 + \norm{\nu - \langle \nu, \psi_{h_0} \rangle \psi_{h_0}}_{H^1}^2 + \norm{\eta}_{H^1}^3 \Bigr) - \lambda \langle \nu, \psi_{h_0} \rangle^2, \\
        \abs{E_{h_0} (m) - E_{h_0} (w_{h_0})} \leq \frac{1}{\lambda} \norm{\eta}_{H^1}^2,
    \end{gather*}
    where $\nu, \rho$ are defined in Lemma \ref{lem:expand_eta} and $\psi_{h_0}$ in Lemma \ref{lem:schro_op_2}.
\end{lem}

Once again, this lemma is easily obtained by Lemma \ref{lem:expand_energy} and Lemma \ref{lem:schro_op_2}.

\subsubsection{Expansion of the {dissipation}}

Last, we summarize the expansion of the {dissipation} in terms of $\nu$ and $\rho$.

\begin{lem} \label{lem:summary_expand_dissipation}
    Assume that $h_0 > 0$ or $h_0 \in (-1, 0)$.
    There exists $C, \lambda > 0$ such that, for any $m = w_{h_0} + \eta \in \mathcal{H}^2$ satisfying
    \begin{equation*}
        \norm{\eta}_{H^1} < \delta_3,
    \end{equation*}
    {where $\delta_3 > 0$ is defined in Lemma \ref{lem:expand_eta},} there holds
    \begin{equation*}
        - \int \abs{m \wedge H(m)}^2 \diff x \leq - \lambda \norm{\eta}_{H^1}^2 + \frac{1}{\lambda} (\langle \nu, \partial_x \theta_{h_0} \rangle^2 + \langle \rho, \sin \theta_{h_0} \rangle^2)
    \end{equation*}
    where $\nu, \rho$ are defined in Lemma \ref{lem:expand_eta} and $\psi_{h_0}$ in Lemma \ref{lem:schro_op_1}.
\end{lem}

This last result is obtained thanks to Lemma \ref{lem:expand_dissipation} along with Lemma \ref{lem:schro_op_1} (when $h_0 > 0$) and Lemma \ref{lem:schro_op_2} (when $h_0 \in (-1, 0)$).

\subsection{$V_\delta$ is non-empty}
\label{subsec:V_delta_not_empty}

We are now in position to prove Theorem \ref{th:instability}. We begin by the first part, which is the existence of initial data arbitrarily close to $m_{h_0}$ with a strictly smaller total energy.
We recall this here.

\begin{prop}
    Let $h_0 \in (-1, 0) \cup (0, \infty)$.
    For any {$\delta > 0$}, the following set
    \begin{equation*}
        V_\delta \coloneqq \{ m_0 \in \mathcal{H}^2 \, | \, \norm{m_0 - w_{h_0}}_{H^1} < \delta, E_{h_0} (m_0) < E_{h_0} (w_{h_0}) \}
    \end{equation*}
    is non-empty.
\end{prop}

The proof is constructivist, although it can be easily shown that an infinite number of functions are in the set $V_\delta$, which takes the form of a cone near $m_{h_0}$.

\begin{proof}
    \underline{1st case}: $h_0 > 0$
    
    {Let $\psi_{h_0}$ as defined in Lemma \ref{lem:schro_op_1}.}
    First, we point out that $\psi_{h_0} \in H^1 \subset L^\infty$. 
    Let $\varepsilon_0 \neq 0$ small enough to be fixed later, with ${\abs{\varepsilon_0}} < \frac{1}{2 \norm{\psi}_{L^\infty}}$.
    Let ${\eta_0} \coloneqq \varepsilon_0 (\psi_{h_0} e_3 + \Phi w_{h_0})$ where $\Phi$ is defined by
    \begin{equation*}
        {\Phi \coloneqq \frac{- 1 + \sqrt{1 - \varepsilon_0^2 \abs{\psi_{h_0}}^2}}{\varepsilon_0} = - \frac{\varepsilon_0 \abs{\psi_{h_0}}^2}{1 + \sqrt{1 - \varepsilon_0^2 \abs{\psi_{h_0}}^2}}}.
    \end{equation*}
    By assumption on $\varepsilon_0$, we know that
    \begin{equation*}
        \frac{1}{2} < 1 - 2 \varepsilon_0^2 \abs{\psi_{h_0}}^2 < 1.
    \end{equation*}
    Moreover, since $\psi_{h_0} \in H^k$ for all $k \in \mathbb N$, then so is $\abs{\psi_{h_0}}^2$.
    This proves that $\Phi$ is well defined and it is of the same regularity as $\psi_{h_0}$: $\Phi \in H^k$ for all $k \in \mathbb N$, with
    \begin{equation*}
        \norm{\Phi}_{H^k} \leq C_k {\abs{\varepsilon_0}},
    \end{equation*}
    for some $C_k > 0$ depending on $k$ and on $\norm{\psi_{h_0}}_{H^k}$.
    By the definition of $\Phi$, we have
    \begin{equation*}
        \Phi = - \frac{\varepsilon_0}{2} (\abs{\psi_{h_0}}^2 + \abs{\Phi}^2),
    \end{equation*}
    and
    \begin{equation*}
        \abs{w_{h_0} + \eta_0}^2 = 1 + 2 w_{h_0} \cdot \eta_0 + \abs{\eta_0}^2 = 1 + 2 \varepsilon_0 \Phi + \varepsilon_0^2 (\abs{\psi_{{h_0}}}^2 + \abs{\Phi}^2) = 1.
    \end{equation*}
    Thus, there holds
    \begin{equation*}
        m_0 \coloneqq w_{h_0} + \eta_0 \in \mathcal{H}^k
    \end{equation*}
    and
    \begin{equation*}
        \norm{m_0 - w_{h_0}}_{H^k} \leq C_k {\abs{\varepsilon_0}}
    \end{equation*}
    for all $k \geq 1$. Assuming $C_1 {\abs{\varepsilon_0}} \leq \delta$ leads to $\norm{m_0 - w_{h_0}}_{H^1} \leq \delta$.
    Last, using Lemma \ref{lem:summary_expand_energy}, we get, for some ${C, \lambda} > 0$, 
    \begin{equation*}
        E_{h_0} (m_0) - E_{h_0} (w_{h_0}) \leq - \lambda \varepsilon_0^2 + C \varepsilon_0^3.
    \end{equation*}
    By taking $\abs{\varepsilon_0}$ small enough, we get
    \begin{equation*}
        E_{h_0} (m_0) - E_{h_0} (w_{h_0}) \leq - \frac{1}{2} \lambda_2 \varepsilon_0^2 < 0.
    \end{equation*}
    Therefore, $m_0 \in V_\delta$.
    
    \medskip
    
    \underline{2nd case}: $h_0 \in (-1, 0)$
    
    {Let $\psi_{h_0}$ as defined in Lemma \ref{lem:schro_op_2}.}
    Once again, $\psi_{h_0} \in H^1 \subset L^\infty$, and for $\varepsilon_0$ small enough with ${\abs{\varepsilon_0}} < \frac{1}{2 \norm{\psi}_{L^\infty}}$, we define $\eta_0 \coloneqq \varepsilon_0 (\psi_{h_0} n_{h_0} + \Phi w_{h_0})$ where $\Phi$ is defined like previously, by
    \begin{equation*}
        \Phi \coloneqq \frac{- 1 + \sqrt{1 - \varepsilon_0^2 \abs{\psi_{h_0}}^2}}{\varepsilon_0} = - \frac{\varepsilon_0 \abs{\psi_{h_0}}^2}{1 + \sqrt{1 - 2 \varepsilon_0^2 \abs{\psi_{h_0}}^2}}.
    \end{equation*}
    The same arguments as for $h_0 > 0$ hold here, and we get $m_0 \in V_\delta$ for any $\varepsilon_0$ small enough
\end{proof}

\begin{rem} \label{rem:more_explicit_V_delta}
    {In the proof above, one} can also substitute $\psi_{h_0}$ by $\partial_x \theta_{h_0}$ in the case $h_0 > 0$ and by $\sin \theta_{h_0}$ in the case $h_0 \in (-1, 0)$.
    Indeed, Remarks \ref{rem:eigenvalue_L2_h0_pos} and \ref{rem:eigenvalue_L1_h0_neg} show that, in their respective case, $\langle L_2 \partial_x \theta_{h_0}, \partial_x \theta_{h_0} \rangle < 0$ and $\langle L_1 \sin \theta_{h_0}, \sin \theta_{h_0} \rangle < 0$, and the same arguments still apply.
    This remark will be important for numerical simulations, as $\psi_{h_0}$ is not explicit whereas $\partial_x \theta_{h_0}$ and $\sin \theta_{h_0}$ can be easily computed once we have (numerically) computed $\theta_{h_0}$.
\end{rem}

\subsection{Nonlinear instability of the orbit of $w_{h_0}$}

Now, we are in position to {complete} the proof of Theorem \ref{th:instability}.

\begin{proof}[Proof of Theorem \ref{th:instability}]
    Let $\varepsilon \coloneqq \min(\frac{\delta_1}{2}, \frac{\delta_3}{1 + 2 C_1})$ and $0 < \delta < \varepsilon$, where $\delta_1$ and $C_1$ are defined in Lemma \ref{lem:modulation} and $\delta_3$ is defined in Lemma \ref{lem:expand_eta}.
    Let $m_0 \in V_\delta$.
    By Proposition \ref{prop:evolution_E_h_0}, there obviously holds that $E_{h_0} (m (t))$ is non-increasing, and thus
    \begin{equation*}
        E_{h_0} (m(t)) \leq E_{h_0} (m_0) < E_{h_0} (w_{h_0}) = E_{h_0} (g.w_{h_0}),
    \end{equation*}
    for any $g \in G$ and for all $t \geq 0$.
    In particular, $\mathcal{E} (t) \coloneqq E_{h_0} (m(t)) - E_{h_0} (w_{h_0}) < 0$.
    Define
    \begin{equation*}
        T \coloneqq {\sup} \{ \tau \geq 0 \, | \, \forall t \in [0, \tau], \inf_{g \in G} \norm{m (t) - g.w_{h_0}}_{H^1} \leq \varepsilon \}.
    \end{equation*}
    Thanks to the continuity of the flow in $\mathcal{H}^1$ (see part \ref{item:continuity_flow} of Theorem \ref{th:lwp}) and with the assumptions on the initial data, there holds $T > 0$.
    Let $0 < \tau < T$. By definition of $T$, we have $\inf_{g \in G} \norm{m (t) - g.w_{h_0}}_{H^1} \leq \varepsilon \leq \frac{1}{2} \delta_1$ for all $t \in [0, \tau]$.
    Therefore, for any $t \in [0, T]$, we can find some $\tilde g (t) \in G$ such that $\norm{m (t) - \tilde g (t).w_{h_0}}_{H^1} = \norm{(- \tilde g (t)).m (t) - w_{h_0}}_{H^1} < 2 \varepsilon \leq \delta_1$
    Then, from Lemma \ref{lem:modulation} applied to $(- \tilde g (t)).m (t)$ for any $t \in [0, \tau]$, we get some $\hat g (t)$ such that, by defining $g \coloneqq \tilde g + \hat g$ and {$\eta (t) \coloneqq (- g(t)) . m(t) - w_{h_0}$,}
    \begin{equation} \label{eq:bound_m_mod_vs_wh0_H1}
        \norm{m (t) - g(t) . w_{h_0}}_{H^1} = \norm{\eta}_{H^1} \leq C_1 \norm{(- \tilde g (t)).m (t) - w_{h_0}}_{H^1} < 2 C_1 \varepsilon < \delta_3,
    \end{equation}
    with the orthogonality conditions
    \begin{equation*}
        \int \eta (t) \cdot \partial_x w_{h_0} \diff x = \int \eta (t) \cdot (e_1 \wedge w_{h_0}) \diff x = 0, \qquad \text{for all } t \in [0, \tau].
    \end{equation*}
    {For any fixed $t \in [0, \tau]$,} define $\mu (t), \nu (t), \rho (t)$ from Lemma \ref{lem:expand_eta} applied to $(-g (t)).m (t)${. We can} apply {either} Lemma \ref{lem:summary_expand_energy} if $h_0 > 0$ {or} Lemma \ref{lem:summary_expand_energy_2} if $h_0 \in (-1, 0)$ {on $(- g(t)).m(t)$ again, as this function satisfies the assumption of these lemmas thanks to \eqref{eq:bound_m_mod_vs_wh0_H1}. Therefore, for $\lambda > 0$ defined in Lemma \ref{lem:summary_expand_energy} or Lemma \ref{lem:summary_expand_energy_2}}, there holds 
    \begin{equation} \label{eq:first_est_main_proof}
        \norm{\eta (t)}_{H^1}^2 \geq \lambda \abs{E_{h_0} ((- g(t)).m (t)) - E_{h_0} (w_{h_0})} = - \lambda \mathcal{E} (t), \qquad {\text{for all } t \in [0, \tau]},
    \end{equation}
    since $\mathcal{E} (t) < 0$.
    On the other hand, by {Proposition \ref{prop:evolution_E_h_0} and} Lemma \ref{lem:summary_expand_dissipation},
    \begin{equation*}
        \frac{\diff}{\diff t} E_{h_0} (m(t)) \leq - \lambda \norm{\eta (t)}_{H^1}^2,
    \end{equation*}
    thus there also holds for all $t \in [0, \tau]$
    \begin{equation*}
        \mathcal{E} (t) - \mathcal{E} (0) \leq - \lambda \int_0^t \norm{\eta (s)}_{H^1}^2 \diff s.
    \end{equation*}
    Using \eqref{eq:first_est_main_proof}, we obtain
    \begin{equation} \label{eq:est_energy}
        \mathcal{E} (t) \leq \mathcal{E} (0) + \lambda^2 \int_0^t \mathcal{E} (s) \diff s,
    \end{equation}
    for all $t \in [0, \tau]$. This proves that ${\mathfrak{G}} (t) \coloneqq e^{- \lambda^2 t} \int_0^t \mathcal{E} (s) \diff s$ is decreasing, and more exactly
    \begin{equation*}
        {\mathfrak{G}}'(t) \leq \mathcal{E} (0) e^{- \lambda^2 t}.
    \end{equation*}
    By integration and since ${\mathfrak{G}} (0) = 0$, it yields
    \begin{equation*}
        {\mathfrak{G}} (t) \leq \frac{1 - e^{- \lambda^2 t}}{\lambda^2} \mathcal{E} (0),
    \end{equation*}
    which leads to
    \begin{equation*}
        \int_0^t \mathcal{E} (s) \diff s \leq \frac{e^{\lambda^2 t} - 1}{\lambda^2} \mathcal{E} (0).
    \end{equation*}
    Putting this estimate into \eqref{eq:est_energy}, we finally obtain
    \begin{equation} \label{eq:est_lower_bound_energy}
        \mathcal{E} (t) \leq e^{\lambda^2 t} \mathcal{E} (0).
    \end{equation}
    Now, we show that such an estimate proves that, for all $t \in [0, \tau]$,
    \begin{equation} \label{eq:est_inf_m_stat_sol}
        \inf_{g \in G} \norm{m (t) - g.w_{h_0}}_{H^1} \geq \lambda \abs{\mathcal{E} (0)} e^{\lambda^2 t}.
    \end{equation}
    Indeed, let $t \in [0, \tau]$ and $g_0 \in G$. Since we already know that $\inf_{g \in G} \norm{m (t) - g.w_{h_0}}_{H^1} \leq \varepsilon < \delta_3$, then we can assume that $\norm{m (t) - g_0.w_{h_0}}_{H^1} < \delta_3$. Therefore, we can apply once again Lemma \ref{lem:summary_expand_energy} to $(-g_0).m (t)$, which gives
    \begin{align*}
        \norm{m (t) - g_0.w_{h_0}}_{H^1}^2 &= \norm{(-g_0).m(t) - w_{h_0}}_{H^1}^2 \\
            &\geq \lambda \abs{E_{h_0} ((- g_0 (t)).m (t)) - E_{h_0} (w_{h_0})} = \lambda \abs{E_{h_0} (m (t)) - E_{h_0} (w_{h_0})} = - \lambda \mathcal{E} (t).
    \end{align*}
    Thus, using \eqref{eq:est_lower_bound_energy} along with the fact that $\mathcal{E} (0) < 0$, we get $\norm{m (t) - g_0.w_{h_0}}_{H^1} \geq \lambda \abs{\mathcal{E} (0)} e^{\lambda^2 t}$.
    This estimate is true for general $g_0 \in G$, which shows \eqref{eq:est_inf_m_stat_sol}.

    Now, applying this estimate at $t = \tau$, and using the fact that $\inf_{g \in G} \norm{m (\tau) - g.w_{h_0}}_{H^1} \leq \varepsilon$ since $\tau < T$ and by definition of $T$, we get
    \begin{equation*}
        \lambda \abs{\mathcal{E} (0)} e^{\lambda^2 \tau} \leq \varepsilon,
    \end{equation*}
    which yields
    \begin{equation*}
        \tau \leq \frac{1}{\lambda^2} \ln{\Bigl( \frac{\varepsilon}{\lambda \abs{\mathcal{E} (0)}} \Bigr)}.
    \end{equation*}
    As this is true for general $\tau < T$, we obtain that $T$ satisfies the same upper bound, which obviously leads to \eqref{eq:upper_bound_min_time_instability} by continuity of $t \mapsto \inf_{g \in G} \norm{m (t) - g.w_{h_0}}_{H^1}$.
\end{proof}

\section{Numerical simulations} \label{sec:numerical}

In this section, we delve into the numerical approximations of stationary solutions and explore the evolution of solutions to \eqref{eq:llg} through numerical simulations, all conducted using Python. We begin by examining stationary solutions and proceed to investigate the behavior of solutions under small perturbations from these stationary states. These numerical investigations are supplemented with critical observations, related remarks, and comprehensive discussions regarding the general dynamics of \eqref{eq:llg}.

\subsection{Numerical approximations of the stationary solutions}

The stationary solutions are explicitly related to $\theta_{h_0}$ obtained as the solution of an ODE with explicit initial data given by Theorem \ref{th:stat_sol}.
This {function is} approximated numerically using the \textit{ode45} function from the \textit{matplotlib} library and displayed in Figures \ref{fig:stat_sol_h0_pos_plots} and \ref{fig:stat_sol_h0_neg_plots} for several $h_0$.

It is important to point out that direct use of \eqref{eq:ODE_theta_2_v2} is impractical in the case $h_0 \in (-1, 0)$.
This ODE {$\partial_x \theta = f(\theta)$} involves a function {$f$} that {lacks} $\mathscr{C}^1$ regularity, particularly at $\theta = \arccos{(-1 - 2 h_0)}$.
Even more problematic, the constant function $\arccos{(-1 - 2 h_0)}$ is a solution to \eqref{eq:ODE_theta_2_v2}, and it is probably the solution that a numerical approximation would find in this setting.
Thus, the 2nd order ODE \eqref{eq:ODE_theta_v1} will be used when $h_0 \in (-1, 0)$, whereas this problem does not arise in the case $h_0 > 0$, giving the ability to use \eqref{eq:ODE_theta_2_v1}.

\begin{figure}[htp]
\begin{center}
   \includegraphics[scale=1]{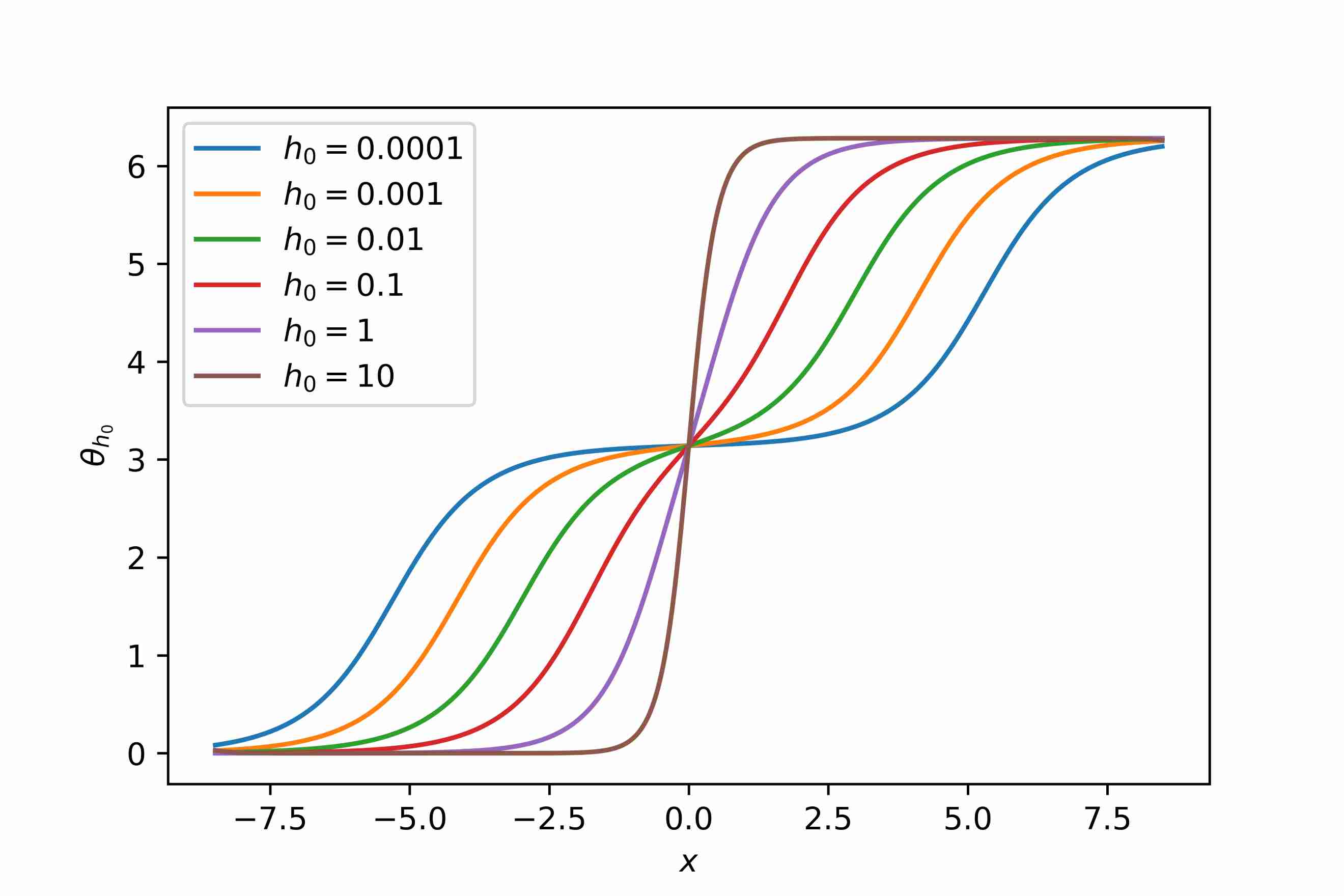}
   \includegraphics[scale=1]{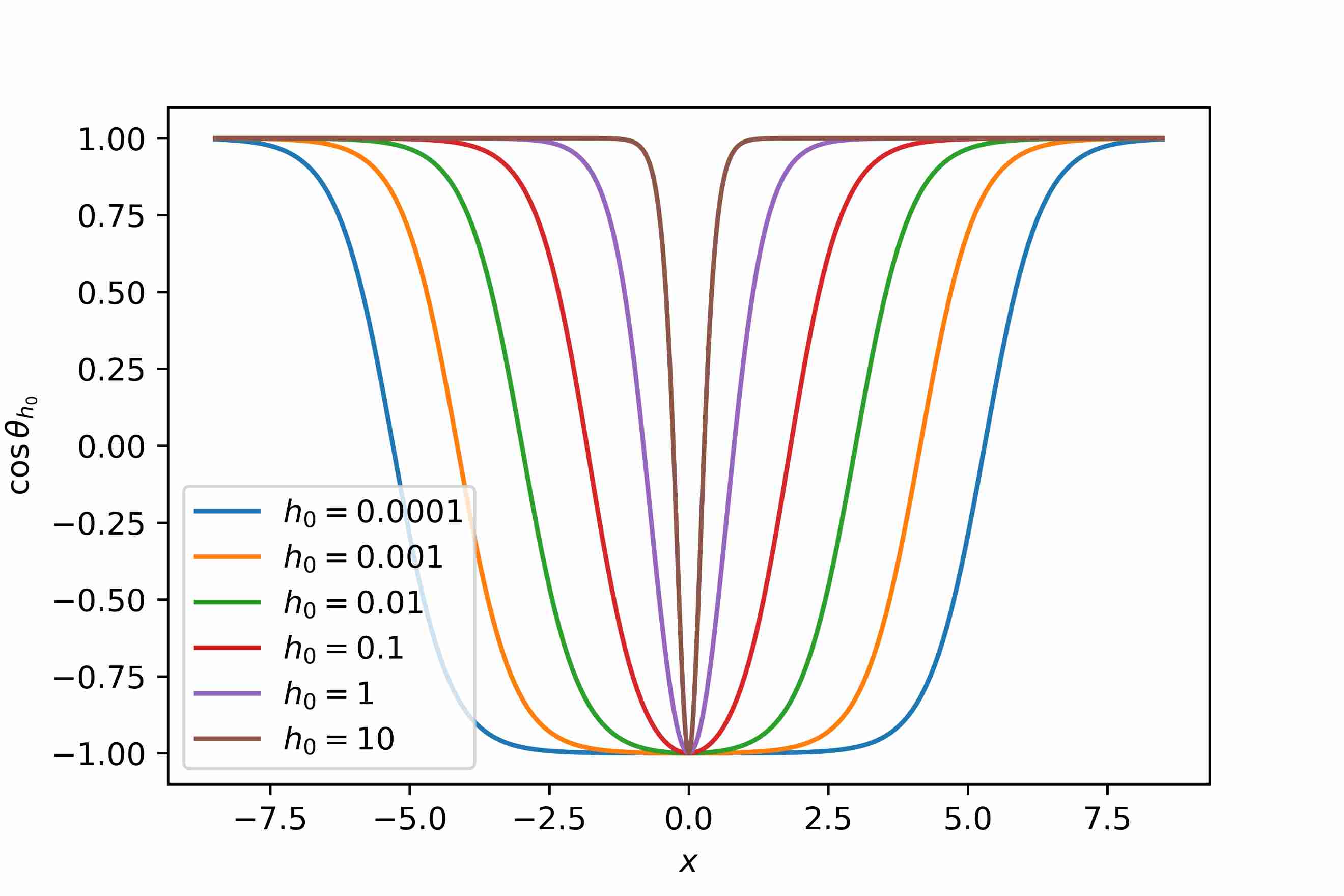}
   \caption{Plot of $\theta_{h_0}$ (on the left) and of $\cos \theta_{h_0}$, which is the first component of $m_{h_0}$ (on the right), for several positive $h_0$.} \label{fig:stat_sol_h0_pos_plots}
\end{center}
\end{figure}

\begin{figure}[htp]
\begin{center}
   \includegraphics[scale=1]{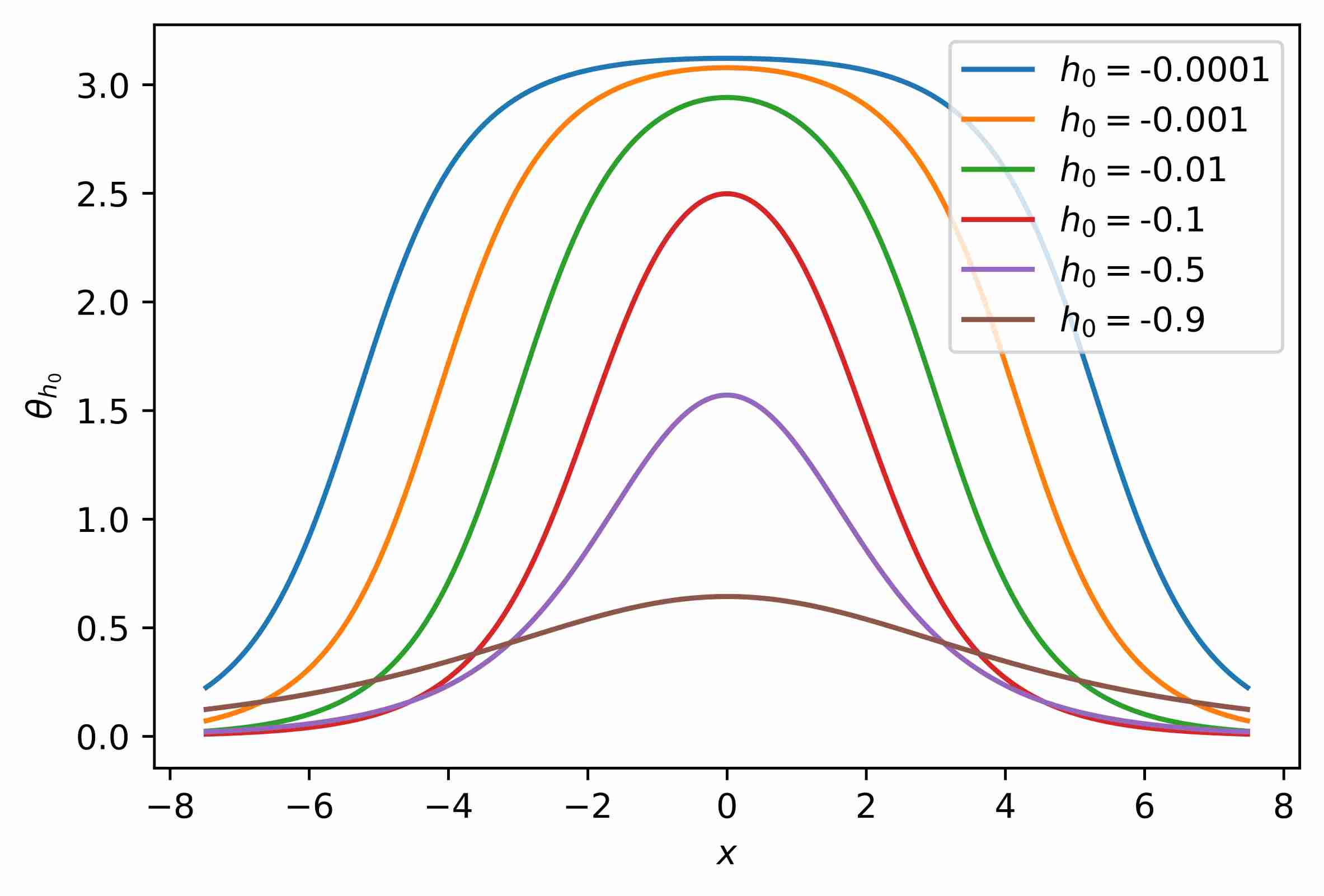}
   \includegraphics[scale=1]{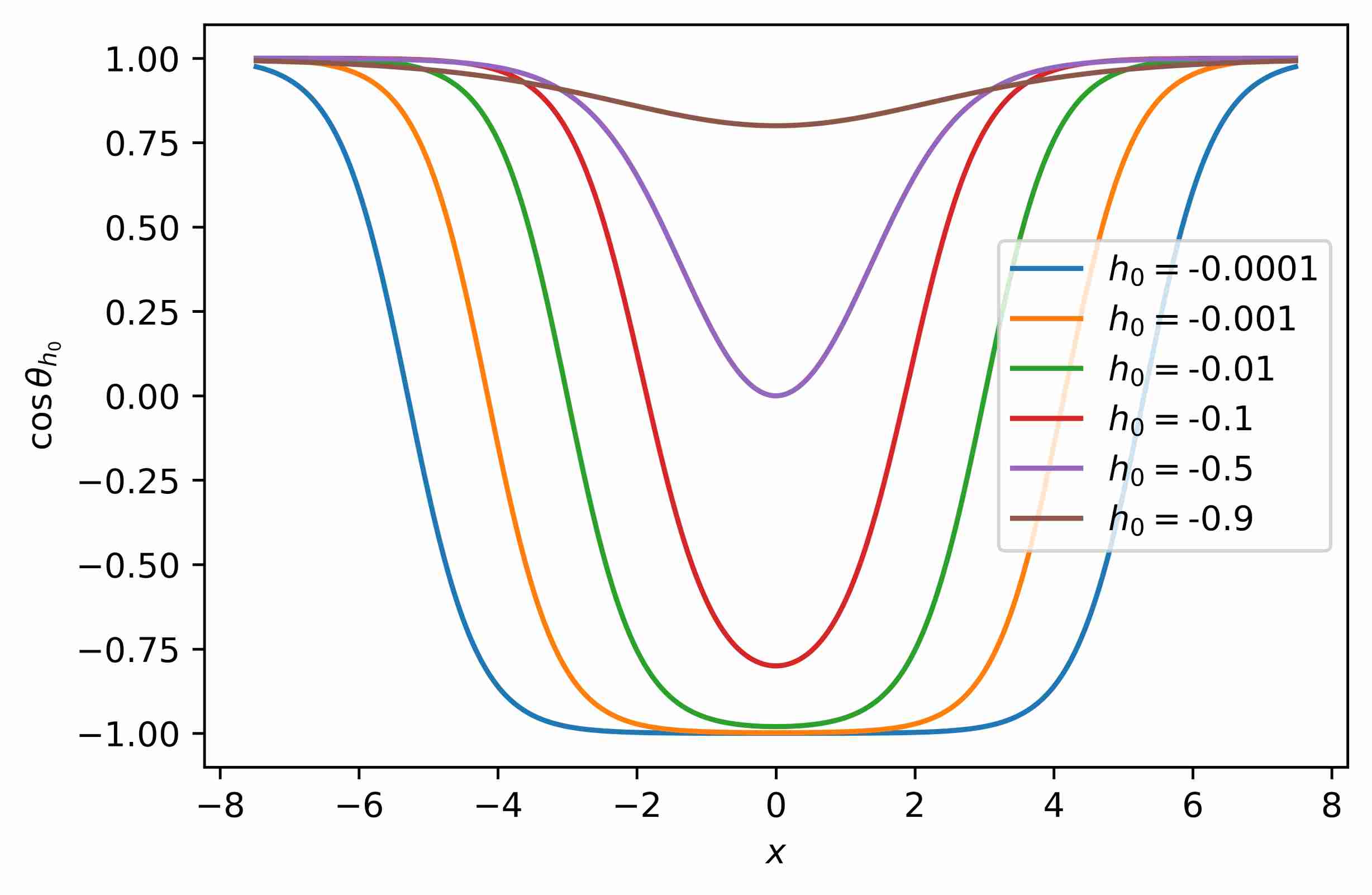}
   \caption{Plot of $\theta_{h_0}$ (on the left) and of $\cos \theta_{h_0}$, which is the first component of $m_{h_0}$ (on the right), for several negative $h_0$.} \label{fig:stat_sol_h0_neg_plots}
\end{center}
\end{figure}

For small $h_0$ (either positive or negative), the stationary solution $w_{h_0}$ looks like a 2-domain wall, as described in \cite{Cote_Ferriere__2DW}, separated from a distance of order roughly $\ln{\abs{h_0}}$.
The distinction between positive and negative $h_0$ lies in the direction of the transition between $- e_1$ and $+ e_1$ of the "domain wall" in $(0, + \infty)$: a rotation around $e_1$ of angle $\pi$ separates the two cases. In the case $h_0 \in (-1, 0)$, this transition takes the same values as the first transition. Conversely, for $h_0 > 0$, it mirrors the first transition with respect to a rotation around $e_1$ of angle $\pi$.

On the contrary, when $h_0 > 0$ is large, $w_{h_0}$ still performs a transition between $e_1$ and $-e_1$ back and forth, but the larger $h_0$, the more abrupt the transition and the shorter the interval of $w_{h_0}$ remaining close to $- e_1$. It is rather intriguing that a solution with such a fast transition with $h_0$ large can be a stationary solution to \eqref{eq:llg}.

There is no stationary solution when $h_0 \leq - 1$. As $h_0$ approaches $- 1^+$, the stationary solution is close to be the constant solution $e_1$. This is most probably related to the fact that the constant solution $e_1$ becomes unstable when $h_0 < 0$ becomes large, with the threshold seemingly occurring at $h_0 = -1$.

\subsection{Numerical simulation}

\subsubsection{Numerical scheme}

The numerical scheme used for the computation of the evolution of \eqref{eq:llg} is a simple explicit finite-difference scheme, on a finite interval $(-L, L)$ with Neumann boundary conditions.
Moreover, between each time step, the solution is renormalized so that $\abs{m} = 1$ anytime and anywhere.
The simulations we will present in the following sections are performed with $L = 15$, space step $dx = 0.2$ and time step $dt = 5.10^{-5}$.
All the plots displayed in the following sections concern $m_1$, as it is the most important component of the magnetization in our problem.

\subsubsection{Instability of the stationary solution}

In Figure \ref{fig:evolution_stat_sol_h0_01}, a numerical simulation of \eqref{eq:llg} with the aforementioned scheme has been performed for $h_0 = 0.1$. The initial data displayed here is the function $m_0$ constructed in Section \ref{subsec:V_delta_not_empty} with Remark \ref{rem:more_explicit_V_delta}, perturbation of the stationary solution, with $\varepsilon_0 = 0.1$. Theorem \ref{th:instability} establishes that for, for sufficiently small $\varepsilon_0$, such initial data subsequently evolves into a distinct state over time. The numerical simulation validates this statement: we see that the solution evolves, deviates from the stationary solution and, in this case, collapses into the constant solution $e_1$.

\begin{figure}[h]
\begin{center}
   \includegraphics[scale=1]{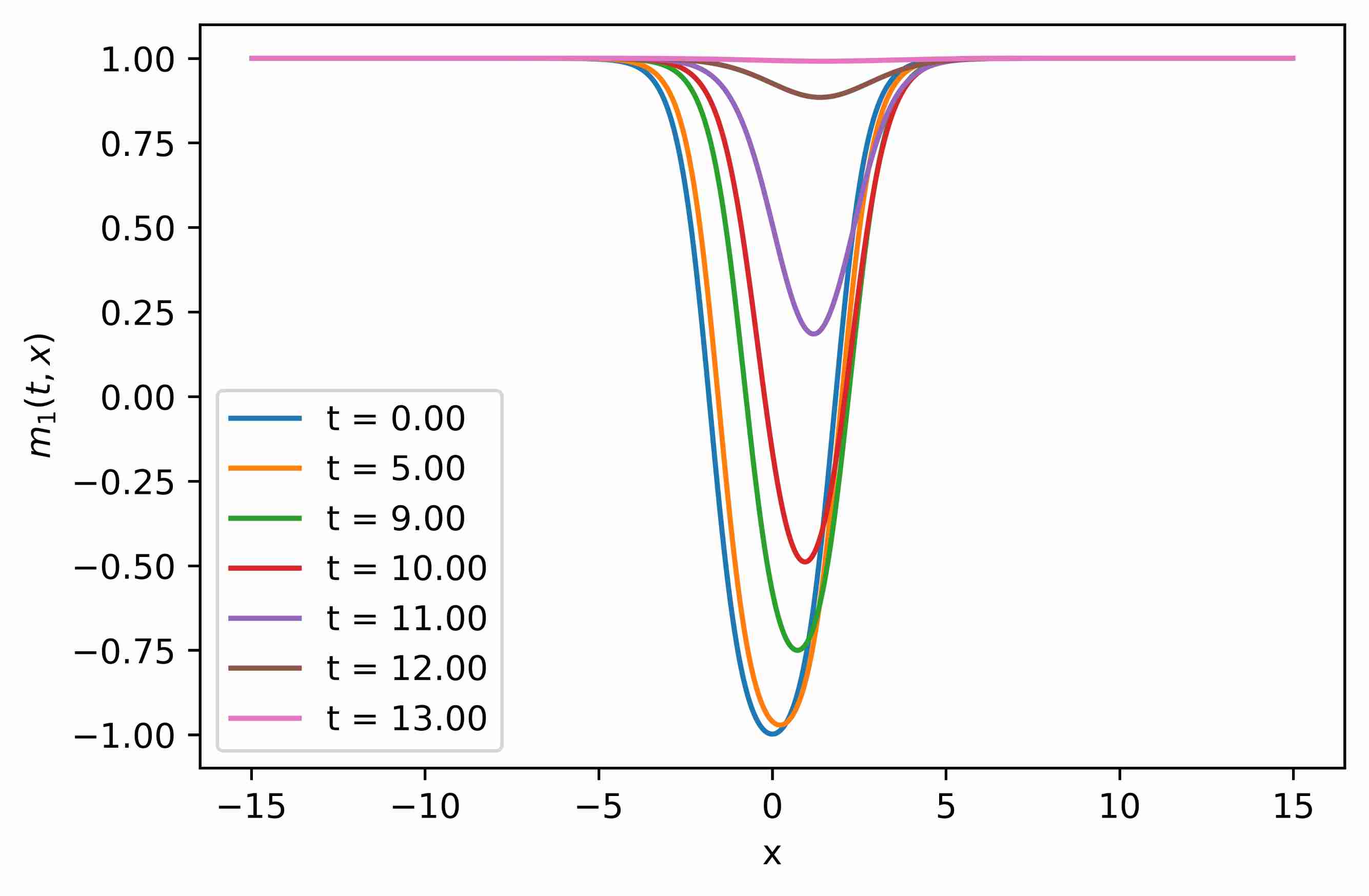}
   \includegraphics[scale=1]{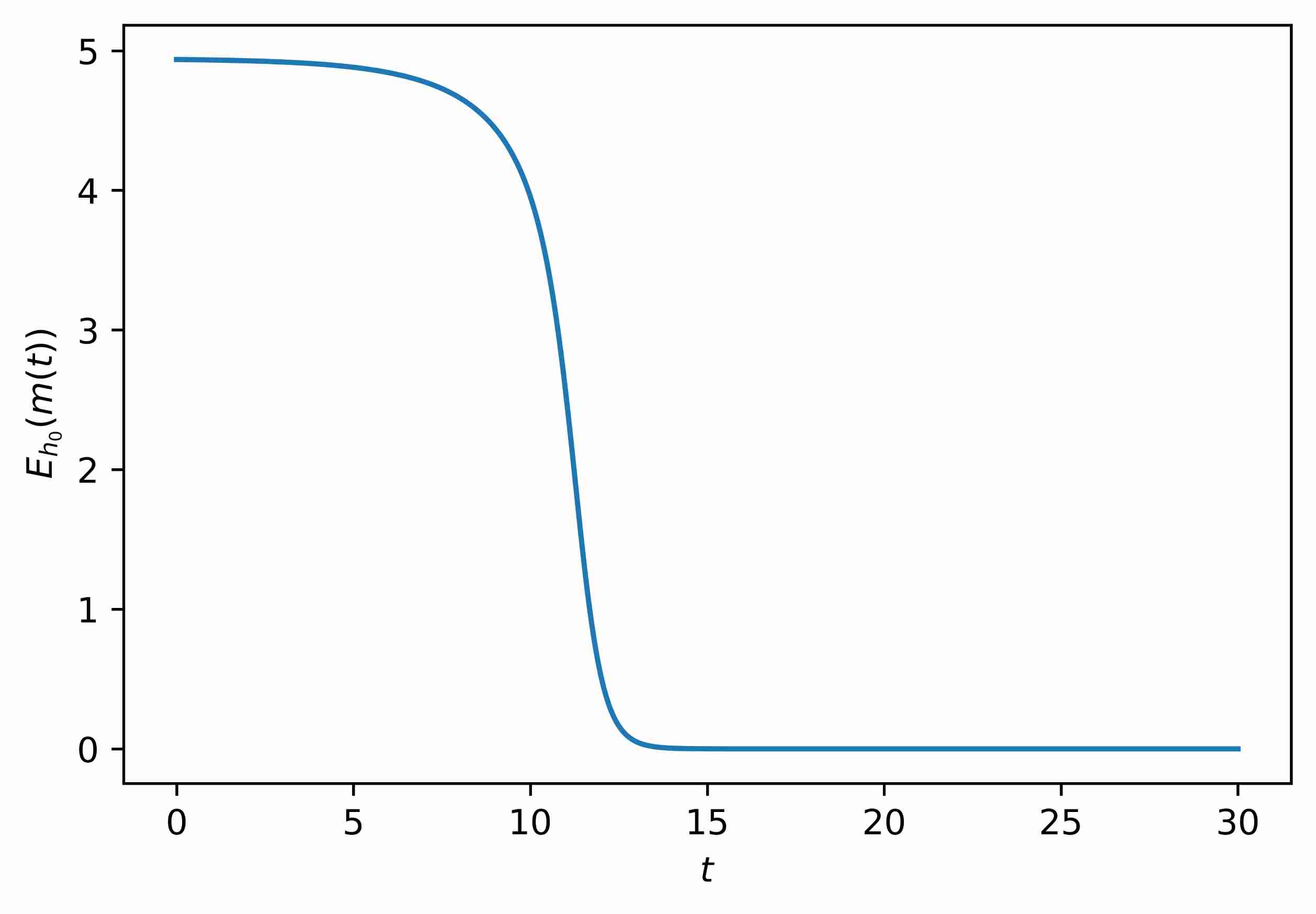}
   \caption{Left: Plot of $m_1 (t)$ for several time between $0$ and $13$. Right: Evolution of the total energy $E_{h_0} (m (t))$. The initial data is a perturbation of the stationary solution $w_{h_0}$ with $h_0 = 0.1$, as depicted in Section \ref{subsec:V_delta_not_empty}.} \label{fig:evolution_stat_sol_h0_01}
\end{center}
\end{figure}

We want to point out that this collapse is not symmetric, because the initial data is not symmetric anymore. While collapsing, the transition of the left-hand side moves to the right.
If we take $\varepsilon_0 = -0.1$, the evolution is symmetric with respect to the case $\varepsilon_0 = 0.1$: there is still collapse, but now the transition of the right-hand side moves to the left.
Last, the (discrete) total energy is indeed decreasing, as expected. More precisely, at first, $E_{h_0}$ is exponentially decreasing. After some time, this decrease starts being smaller and smaller, and converges exponentially to $0$ when $t \rightarrow \infty$. This corresponds to the fact that $E_{h_0} (m) \geq 0$ in the case $h_0 > 0$ and the minimum $0$ is reached at $m \equiv e_1$, which is indeed the limit of the magnetization as depicted in Figure \ref{fig:evolution_stat_sol_h0_01}.

In the case $h_0 = -0.1$, the situation is more involved and even more interesting at this point.
The initial data is again the function $m_0$ constructed in Section \ref{subsec:V_delta_not_empty} with Remark \ref{rem:more_explicit_V_delta}, with $\varepsilon_0 = \pm 0.1$.
However, the behavior of the solution is quite different depending on the sign of $\varepsilon$.
When $\varepsilon = -0.1$, the situation is rather similar to the previous case, with a collapse onto the constant solution $e_1$, although it is more symmetric now (Figure \ref{fig:evolution_stat_sol_h0_neg_01_neg_eps}).

\begin{figure}[h]
\begin{center}
   \includegraphics[scale=1]{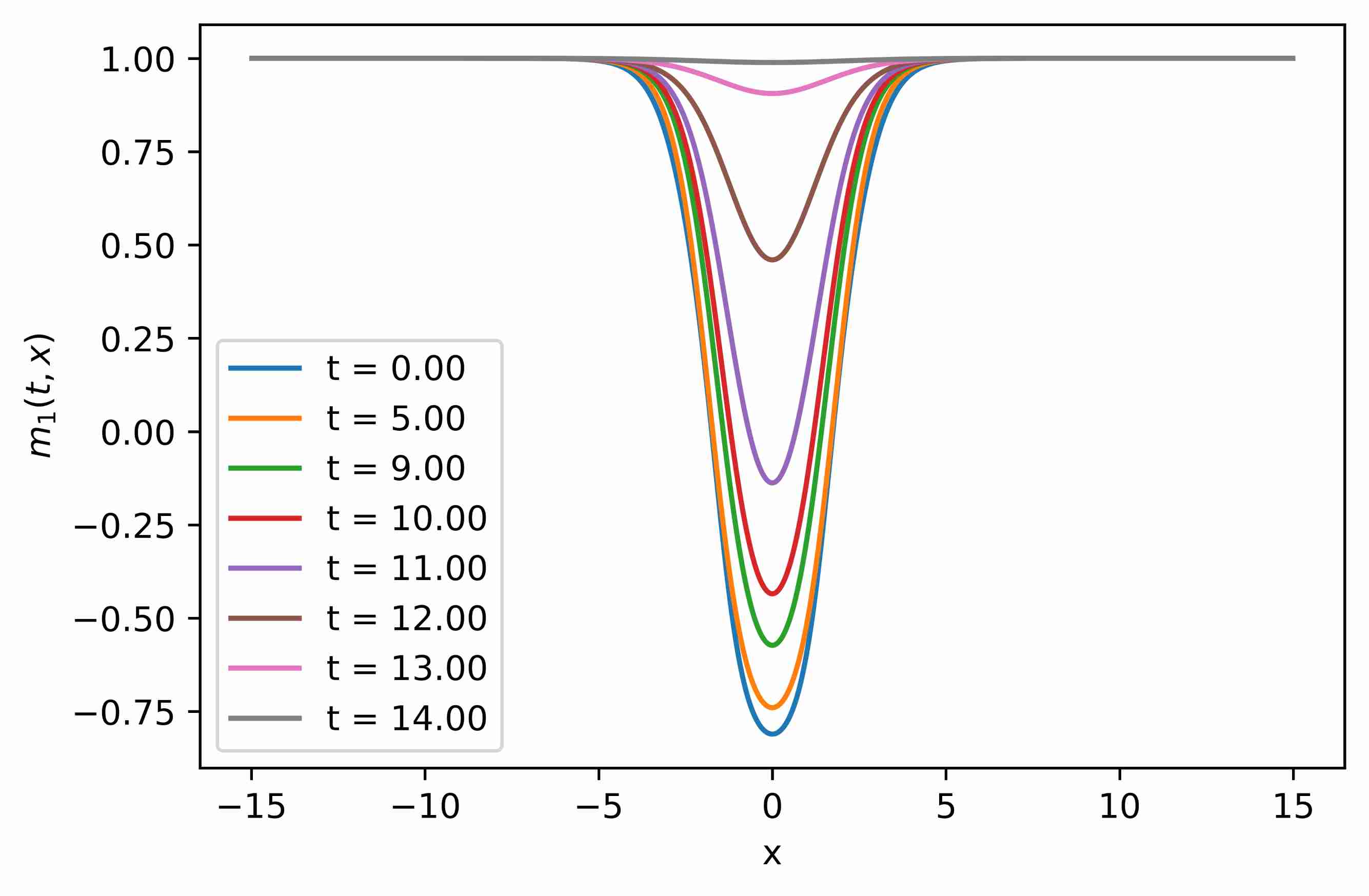}
   \includegraphics[scale=1]{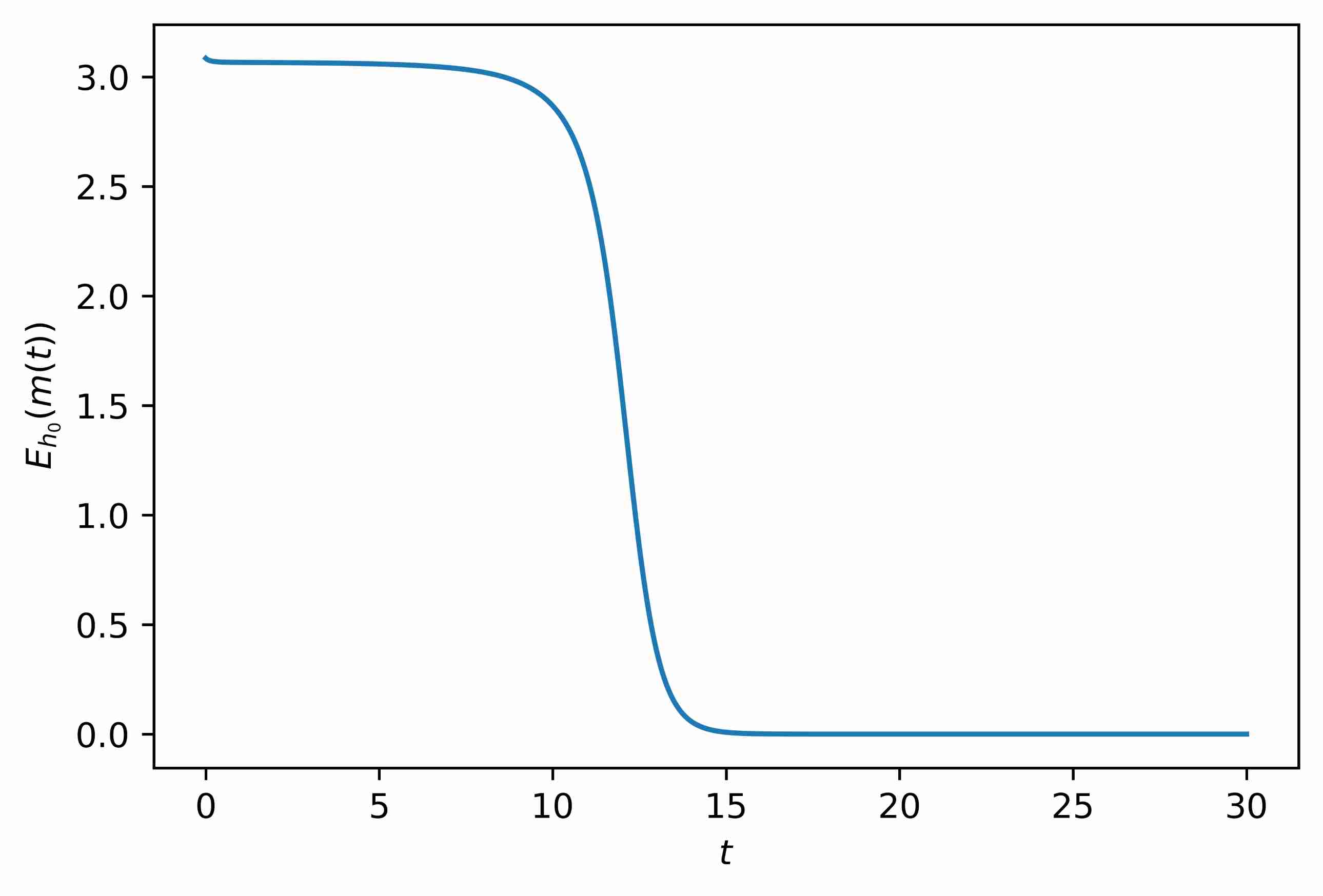}
   \caption{Left: Plot of $m_1 (t)$ for several time between $0$ and $13$. Right: Evolution of the total energy $E_{h_0} (m (t))$. The initial data is a perturbation of the stationary solution $w_{h_0}$ with $h_0 = -0.1$, as depicted in Section \ref{subsec:V_delta_not_empty} with $\varepsilon = - 0.1$.} \label{fig:evolution_stat_sol_h0_neg_01_neg_eps}
\end{center}
\end{figure}

On the other hand, when $\varepsilon = 0.1$ {(Figure \ref{fig:evolution_stat_sol_h0_neg_01})}, the evolution is rather different. Indeed, instead of collapsing, the structure evolves into a 2-domain wall structure, delimiting a magnetic domain of magnetization $- e_1$. Moreover, the domain walls seem to go away to infinity at a linear rate. This is also confirmed by the linear decreasing rate of the total energy, which shows that the area of magnetization $- e_1$ is growing linearly.

\begin{figure}[h]
\begin{center}
   \includegraphics[scale=1]{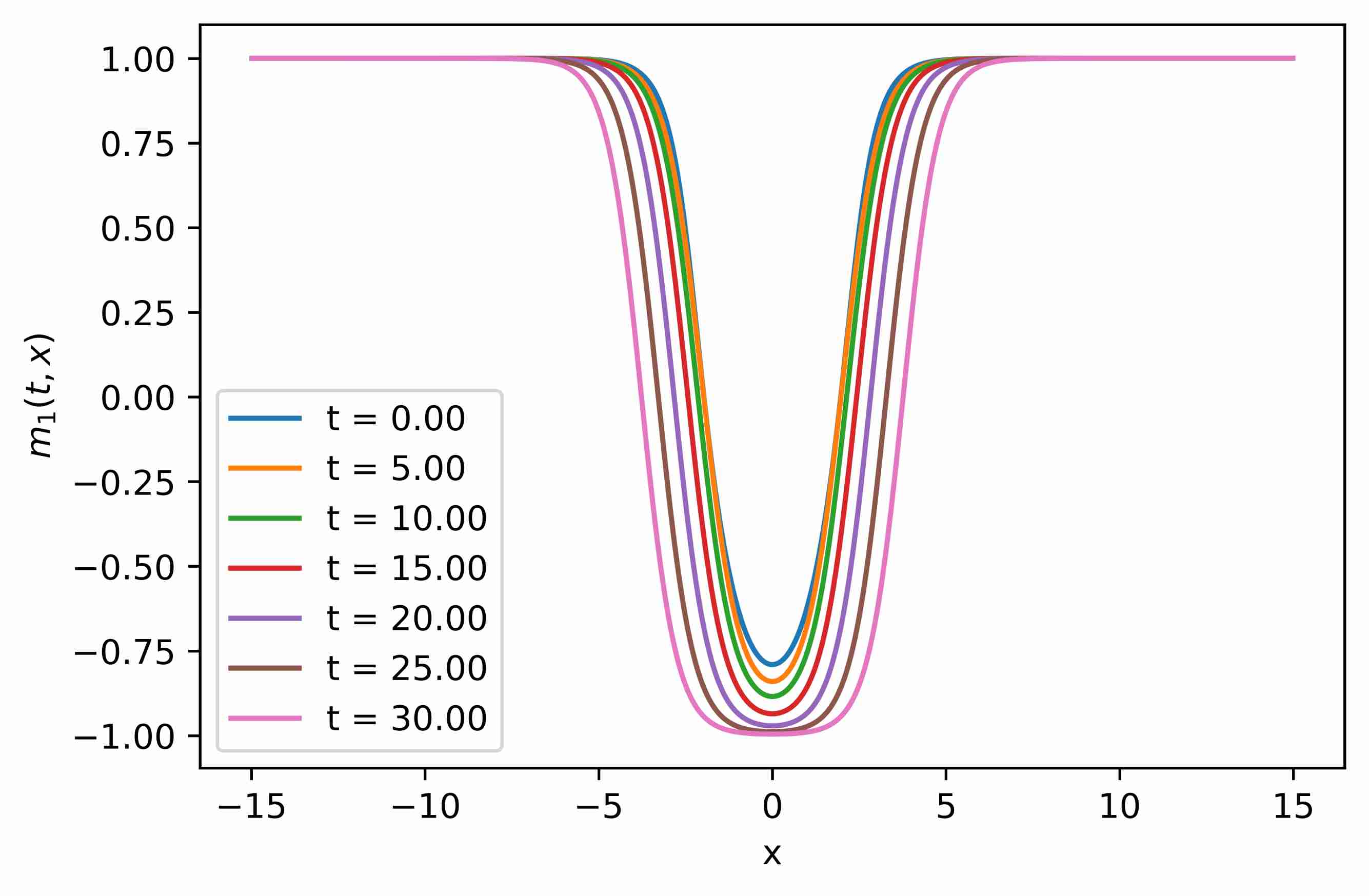}
   \includegraphics[scale=1]{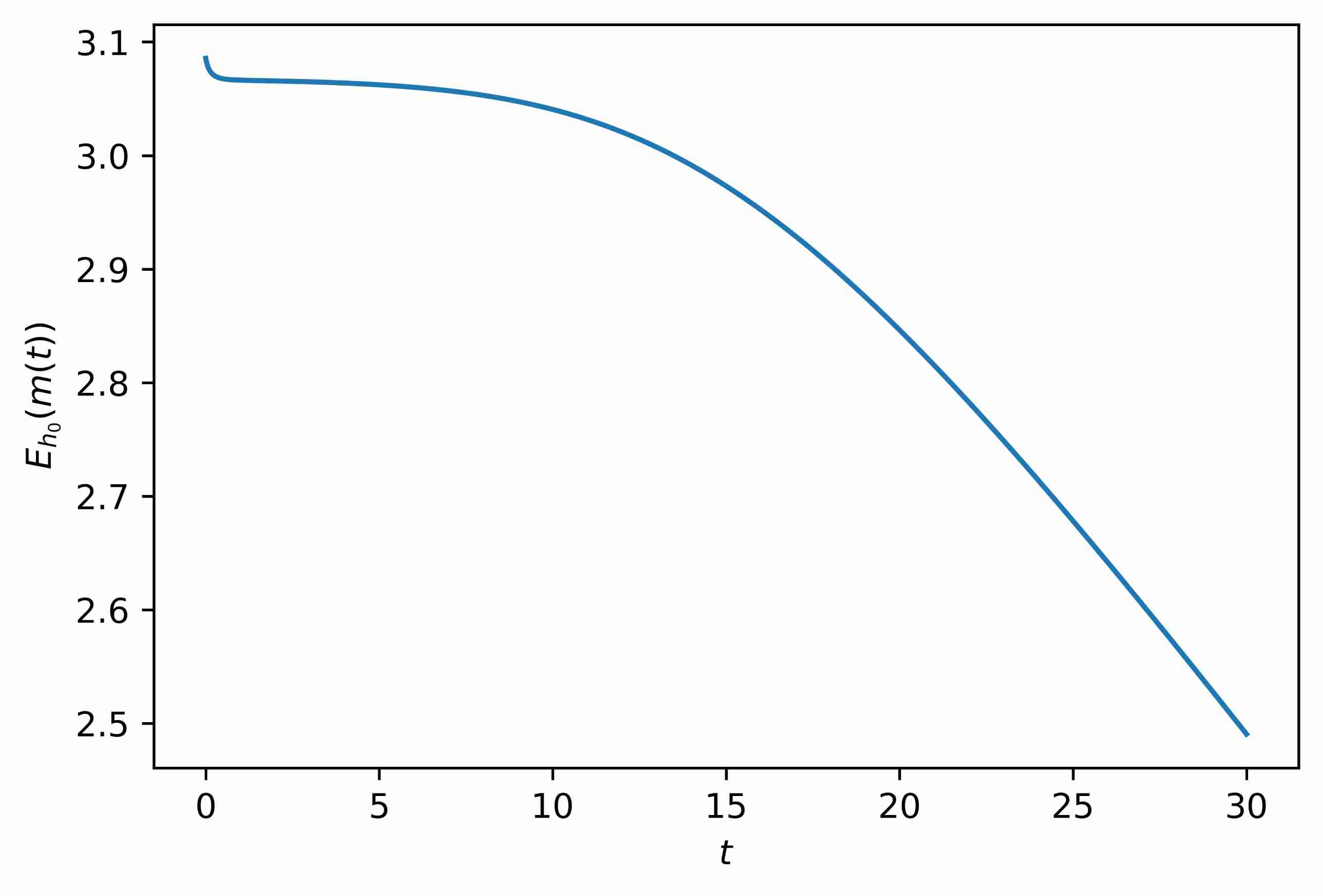}
   \caption{Left: Plot of $m_1 (t)$ for several time between $0$ and $13$. Right: Evolution of the total energy $E_{h_0} (m (t))$. The initial data is a perturbation of the stationary solution $w_{h_0}$ with $h_0 = -0.1$, as depicted in Section \ref{subsec:V_delta_not_empty} with $\varepsilon = 0.1$.} \label{fig:evolution_stat_sol_h0_neg_01}
\end{center}
\end{figure}

This linear evolution of a 2-domain wall confirms the result of \cite[Theorem~1.2]{Cote_Ferriere__2DW}. It also give some information about the assumptions regarding the initial distance between the two domain walls.
The minimal distance seems to be related to this stationary solutions, as the sign of $\varepsilon$ gives a different behavior of the solution (evolution to a 2-domain wall or collapsing).
This minimal distance thus seem to be of order $\ln{\abs{h_0}}$ when $h_0 \rightarrow 0^-$.
This conjecture might give some insights on the internal interactions between domain walls, especially their intensity, but these details are behind the scope of this paper.

\subsubsection{Evolution without external magnetic field}

We also look at the evolution of the solution whose initial data is a stationary solution $w_{h_0}$ for some $h_0 > 0$ or $h_0 \in (-1, 0)$, but with no external magnetic field in \eqref{eq:llg} (i.e. $H_{ext} = 0$).

In the case $h_0 = -0.1$ (figure \ref{fig:evolution_stat_without_ext_mag_sol_h0_neg_01}), the structure collapses into the constant solution $e_1$.
This behavior was expected: the external magnetic field $h_0 e_1$, maintaining the structure when triggered in \eqref{eq:llg}, is "pushing" it towards the magnetization $- e_1$ since $h_0 < 0$, countering the internal forces of the structure which push it towards $e_1$.
But when this external magnetic field is cut off, there is nothing which counters these internal interactions anymore.
This was mostly expected when $h_0$ is close to $- 1^+$, as the stationary solution was already close to $e_1$.
Other simulations with smaller $h_0$ shows the same behavior, although the collapse appear at larger time on which we cannot assure the validity of our simple numerical scheme.

\begin{figure}[h]
\begin{center}
   \includegraphics[scale=1]{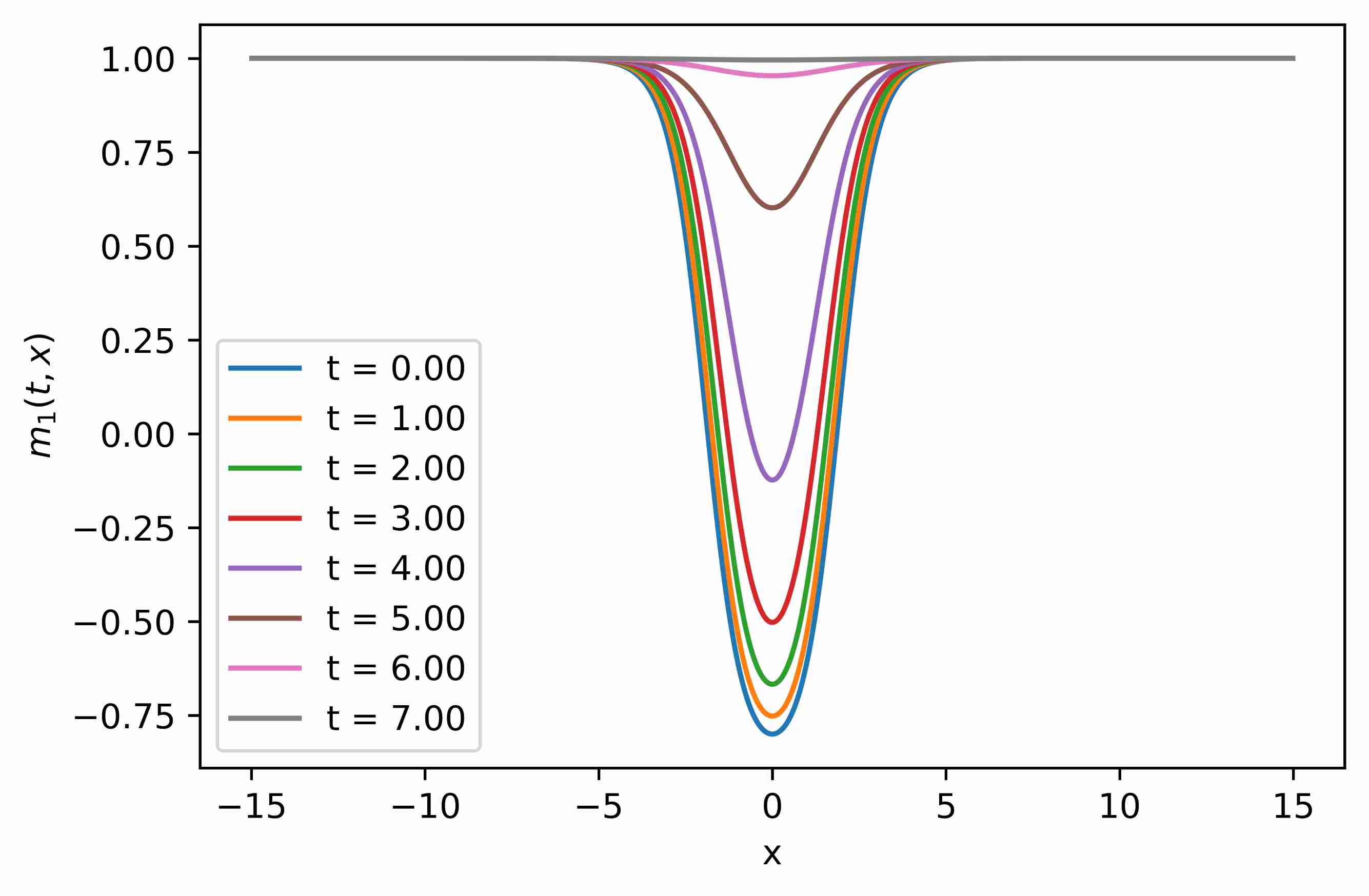}
   \includegraphics[scale=1]{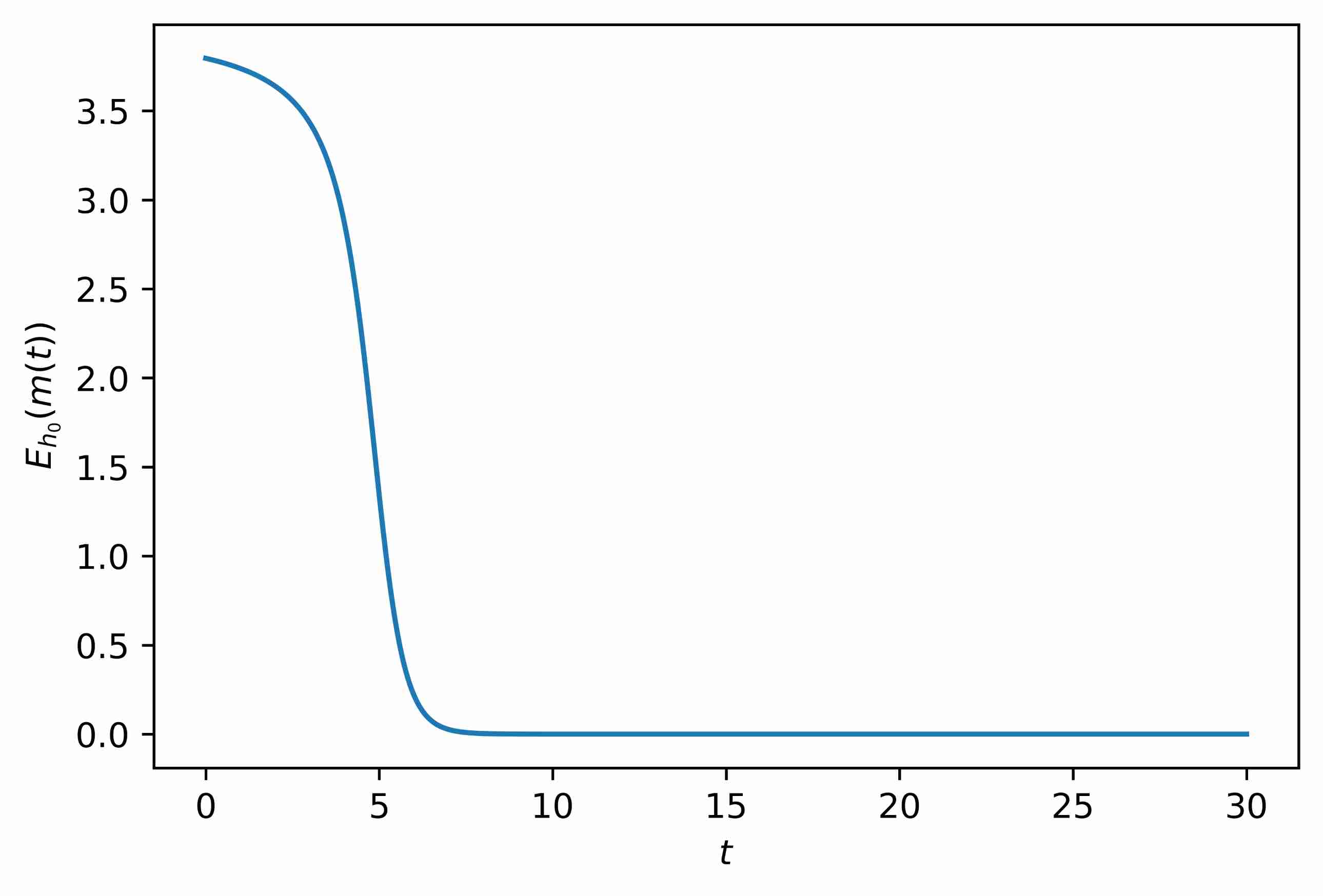}
   \caption{Left: Plot of the first component $m_1 (t)$ of $m(t)$, solution to \eqref{eq:llg} with $H_{ext} = 0$, for several time between $0$ and $7$. Right: Evolution of the energy $E (m (t))$. The initial data is the stationary solution $w_{h_0}$ with $h_0 = -0.1$.} \label{fig:evolution_stat_without_ext_mag_sol_h0_neg_01}
\end{center}
\end{figure}

\begin{figure}[h]
\begin{center}
   \includegraphics[scale=1]{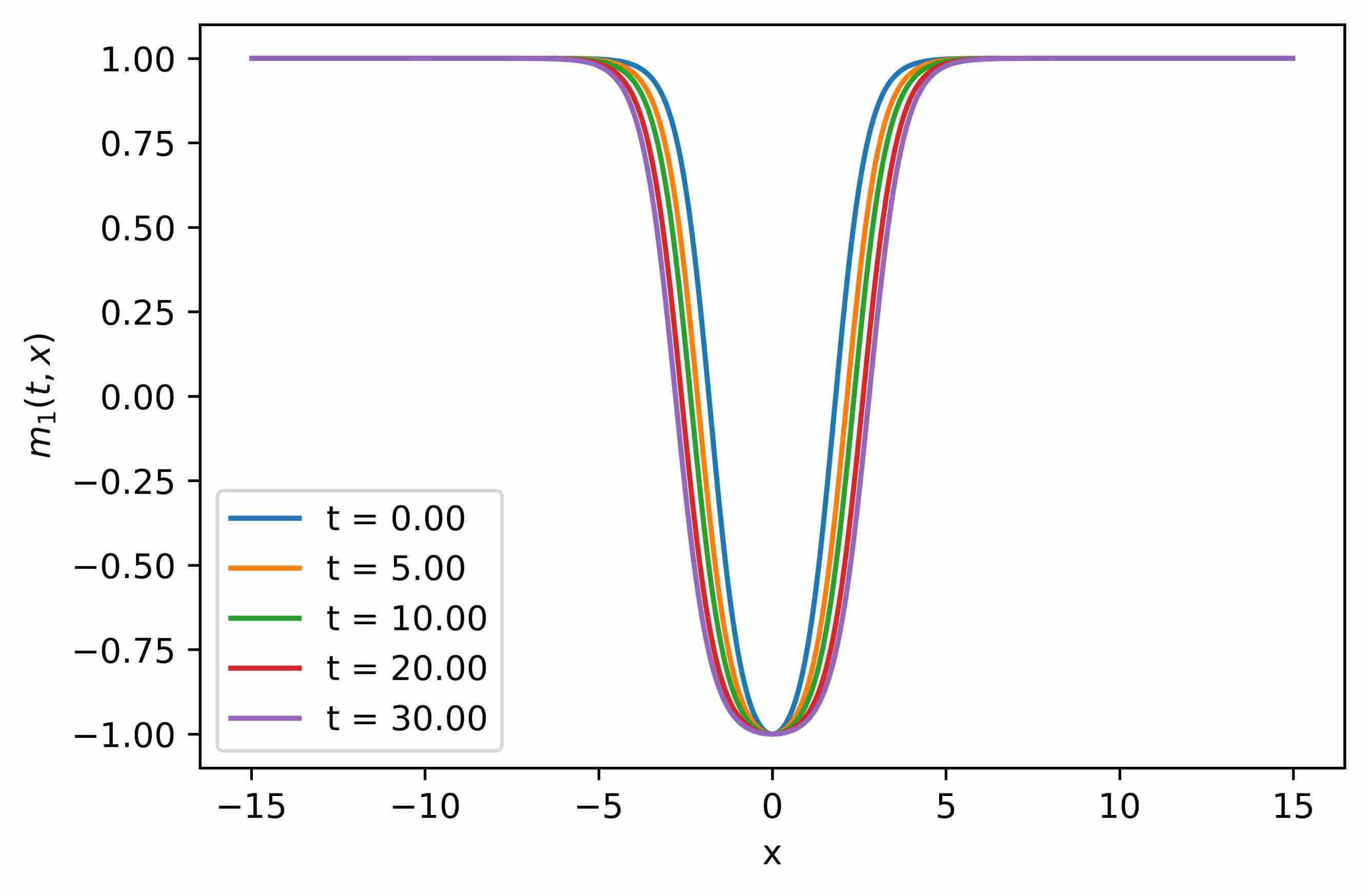}
   \includegraphics[scale=1]{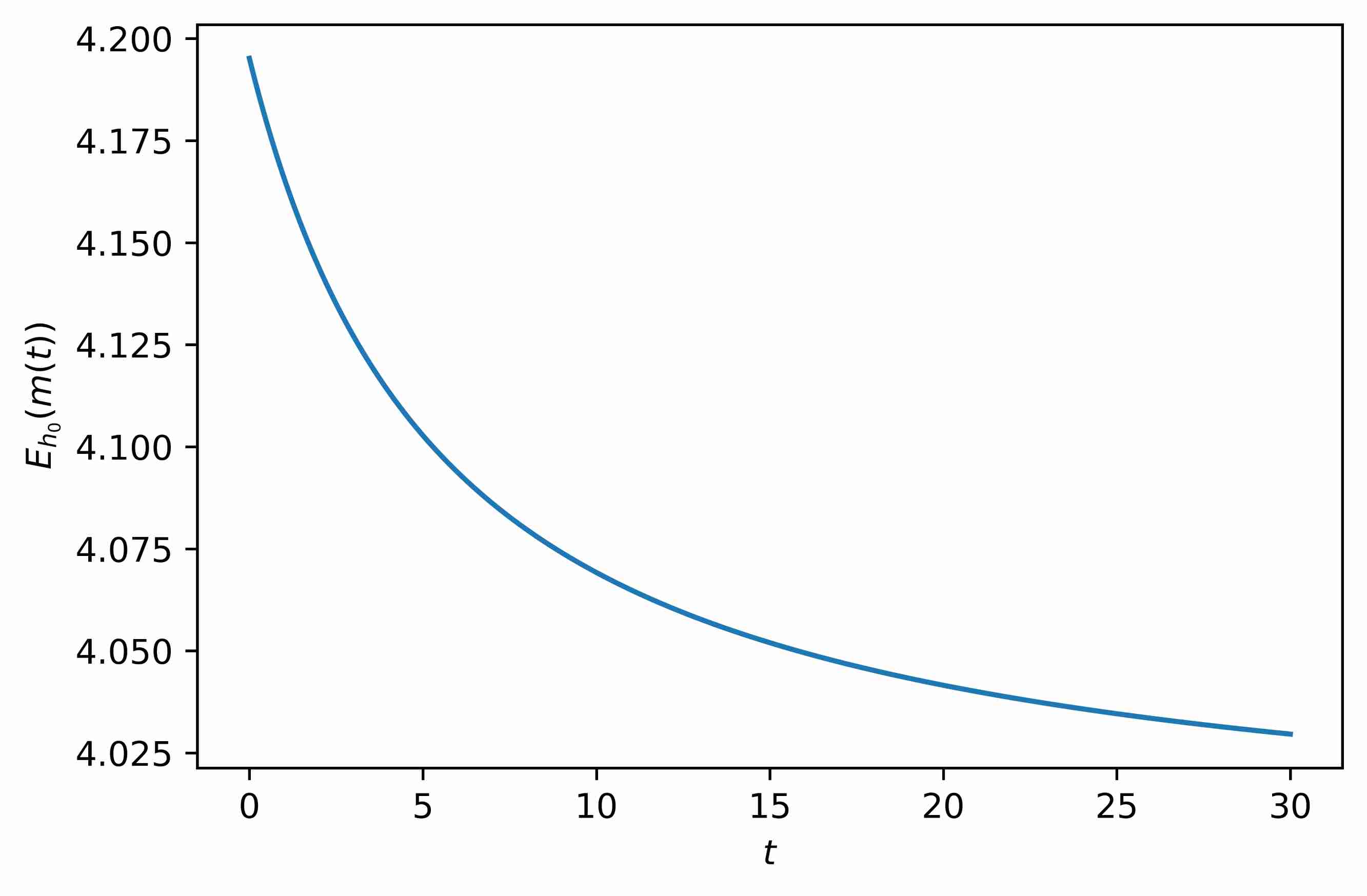}
   \caption{Left: Plot of the first component $m_1 (t)$ of $m(t)$, solution to \eqref{eq:llg} with $H_{ext} = 0$, for several time between $0$ and $30$. Right: Evolution of the energy $E (m (t))$. The initial data is the stationary solution $w_{h_0}$ with $h_0 = 0.1$.} \label{fig:evolution_stat_without_ext_mag_sol_h0_01}
\end{center}
\end{figure}

The case $h_0 > 0$ is more interesting. The external magnetic field $H_{ext}$ is now pushing towards $e_1$, and therefore one gets nontrivial dynamics when $t \rightarrow \infty$ if we cut off this external magnetic field.
When $h_0$ is small, so that $m_{h_0}$ looks like a 2-domain wall, the two domain walls might get closer one from each other.
This is what happens if $h_0 < 0$, as we saw previously.
On the contrary, if $h_0 > 0$, it looks like the internal interactions of this structure push them away from each other (see Figure \ref{fig:evolution_stat_without_ext_mag_sol_h0_01}).

When $h_0$ is large, the dynamics is also very interesting.
The initial data has a rapid transition from $e_1$ to $- e_1$ and back near $x = 0$, but is almost $- e_1$ everywhere else.
Contrarily to what one might expect, the structure does not collapse like previously.
It even grows, so that a real magnetic domain of magnetization $- e_1$ appear, delimited by two domain walls which move away from each other (see Figure \ref{fig:evolution_stat_without_ext_mag_sol_h0_10} for the simulation with $h_0 = 10$).
The evolution of the energy seems to converge to $4$, which is twice the energy of a domain wall, and is therefore in agreement with the previous discussion.
A small random regular perturbation of this initial data still leads to the same behavior, which {might indicate} that this phenomenon is not caused by symmetry.

\begin{figure}[h]
\begin{center}
   \includegraphics[scale=1]{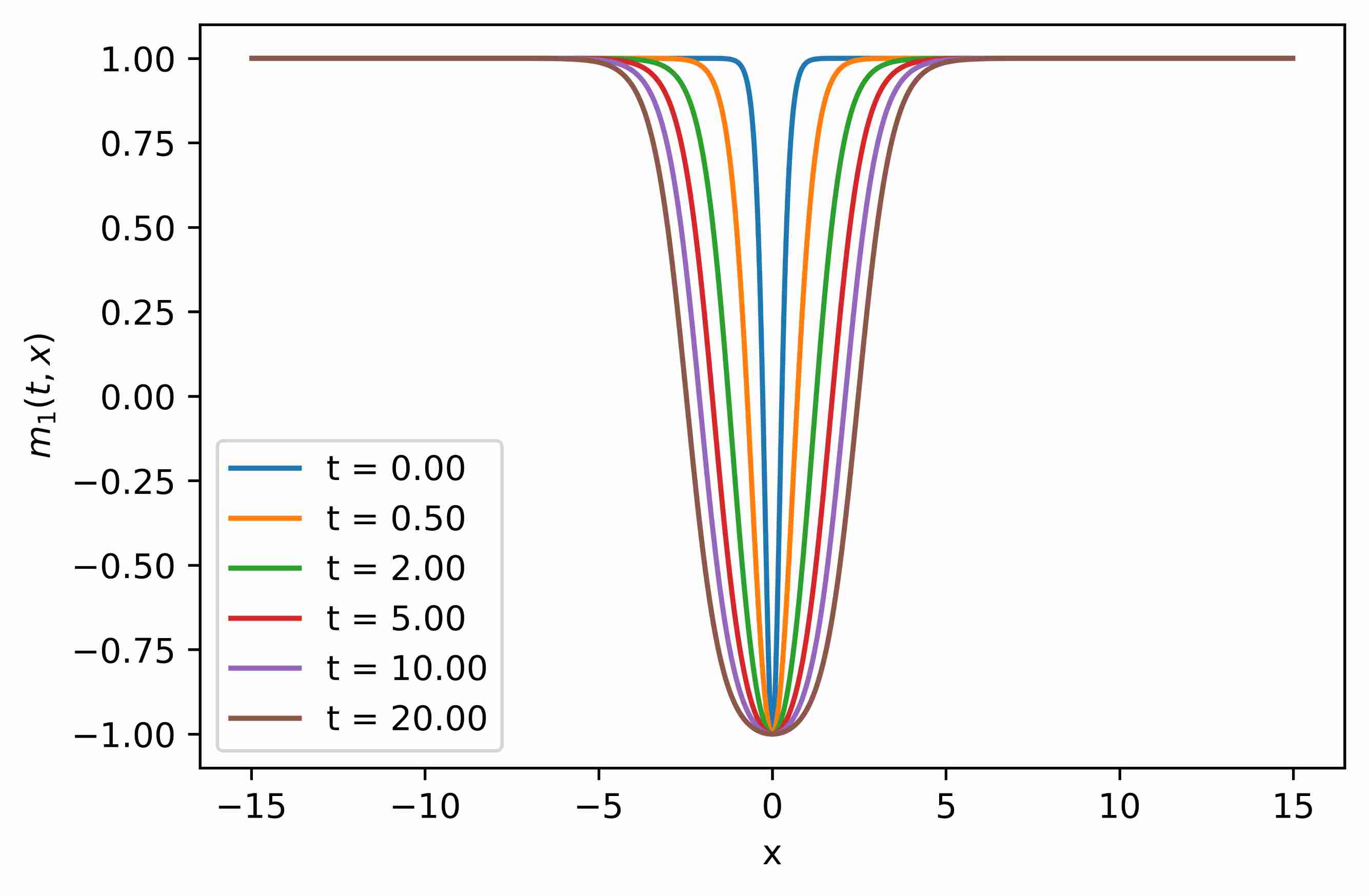}
   \includegraphics[scale=1]{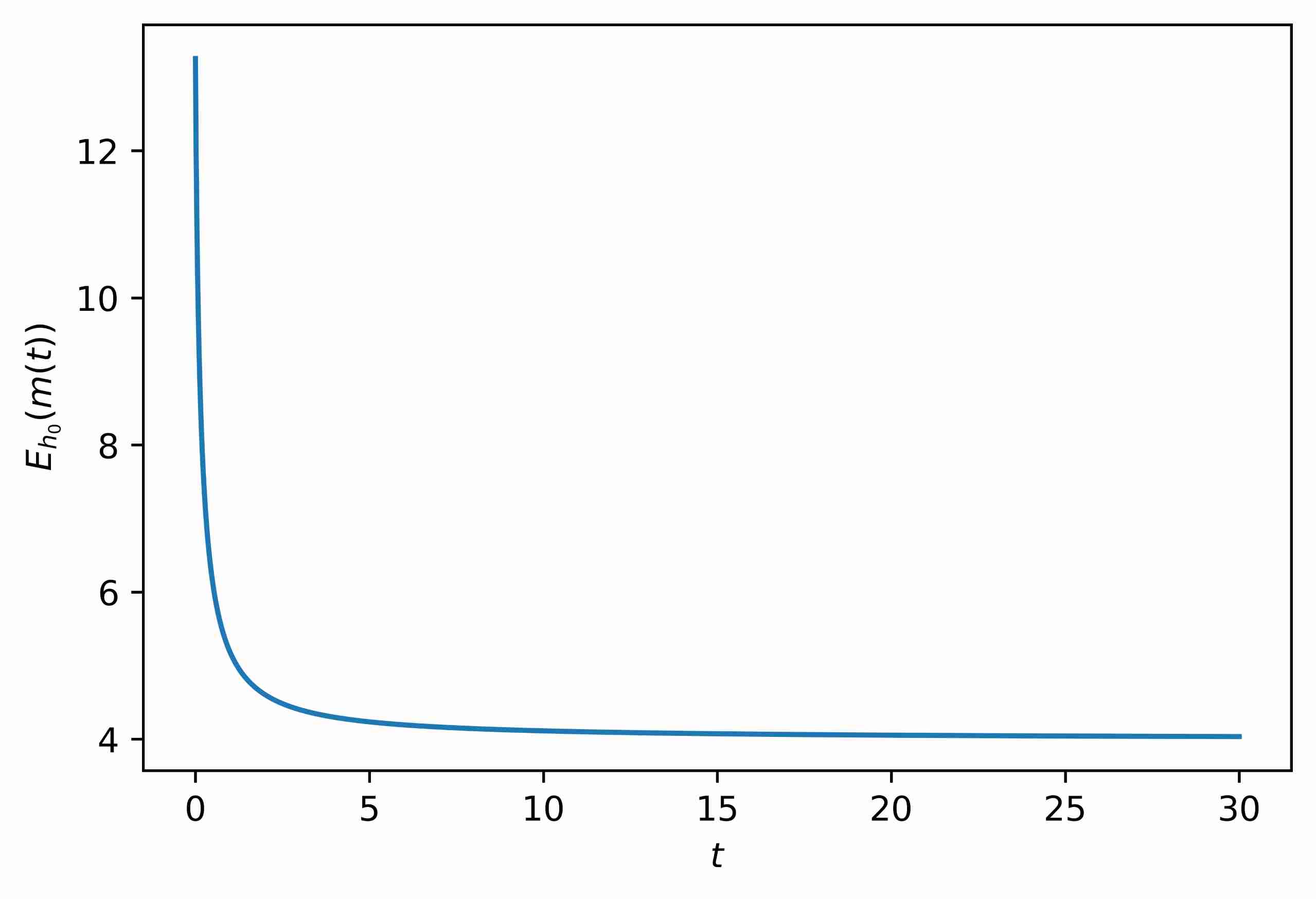}
   \caption{Left: Plot of $m_1 (t)$, solution to \eqref{eq:llg} with $H_{ext} = 0$, for several time between $0$ and $30$. Right: Evolution of the energy $E (m (t))$. The initial data is the stationary solution $w_{h_0}$ with $h_0 = 10$.} \label{fig:evolution_stat_without_ext_mag_sol_h0_10}
\end{center}
\end{figure}

\section{Discussion and open problems}

In this final section, we engage in a brief discussion regarding two intriguing aspects arising from our numerical investigation into the evolution of solutions of the Landau-Lifshitz-Gilbert equation.
First, we explore the nature of interactions between domain walls.
Secondly, we discuss about the spectrum of the operator derived from a direct linearization of the equation, which could suggest further insights of the stability properties of this stationary structure. 
These open problems offer avenues for future research, which might enrich our understanding of magnetization dynamics in nanoscale systems.

\subsection{Attractive and repulsive domain walls ?}

The main question arising from this work {concerns} the interactions between domain walls.
The result of \cite{Cote_Ferriere__2DW} is valid only when the two initial domain walls are far away enough and with a non-negligible external magnetic field which pushes them away from each other.
But it does not say anything about the interactions between these two structures, except that they are negligible enough compared to the external magnetic field in this framework.

On the other hand, the stationary solutions $m_{h_0}$ {look} like 2-domain walls when $h_0$ is small enough.
When $h_0$ is negative, one would expect the two domain walls to move away from each other in view of \cite[Theorem~1.2]{Cote_Ferriere__2DW}.
Since this structure does not move through the flow of \eqref{eq:llg}, this would mean that the forces exerted by the external magnetic field are countered by internal interactions which tend to make them get closer.
This is also shown by the simulation in Figure \ref{fig:evolution_stat_sol_h0_neg_01}.
On the other hand, when $h_0$ is positive, the exact opposite happens: one would expect the external magnetic field to make the two domain walls get closer, but it is probably countered by internal interactions which make them move away from each other.

In conclusion, a domain wall $w_* (x)$ and its opposite $w_* (-x)$ {seem to} have \textit{attractive} interactions, whereas the interactions with the latter being turned by a rotation $R_\pi$ around $e_1$ of angle $\pi$ {seem to} be \textit{repulsive}.
Yet, this interpretation needs to be attenuated. Such considerations are mostly on the positions of the domain walls, but it does not say anything on the evolution of the rotation between the two interfering domain walls.
In particular, for $h_0 > 0$, the analysis performed in this paper shows that the instability of the stationary solution comes from a term $L_2 \rho$, where $\rho$ is strongly related to rotations around $e_1$ ($\nu$ is related to translations, see also Remark \ref{rem:invariance_link}).
Thus, we cannot rule out a possible mechanism which would make the two domain walls to rotate at the same time as they move away from each other at first, which may lead to a new situation where their interactions become attractive.

In the end, these nonlinear dynamics regarding the interactions between domain walls remain an open question, and would require a detailed analysis, going further in the direction of \cite{Cote_Ferriere__2DW}.

\subsection{Spectrum of the linearized operator}

A direct linearization of \eqref{eq:llg} around a stationary solution $w_{h_0}$ is easily available. By decomposing the solution as in Lemma \ref{lem:expand_eta} and using the computations in Lemma \ref{lem:expand_energy}, we obtain that $(\eta, \rho)$ must satisfy
\begin{equation*}
    \frac{\diff}{\diff t}
    \begin{pmatrix}
        \eta \\ \rho
    \end{pmatrix}
    = - L
    \begin{pmatrix}
        \eta \\ \rho
    \end{pmatrix}
    + O_2^2(\eta),
\end{equation*}
where the operator $L$ is given by
\begin{equation*}
    L = 
    \begin{pmatrix}
        \alpha L_1 & -L_2 \\
        L_1 & \alpha L_2
    \end{pmatrix}
\end{equation*}
This operator is not self-adjoint. Yet, it is a relatively compact perturbation of the operators
\begin{equation*}
    \begin{pmatrix}
        \alpha L_1 & -L_1 \\
        L_1 & \alpha L_1
    \end{pmatrix}
    = L_1 M
    \qquad \text{and} \qquad
    \begin{pmatrix}
        \alpha L_2 & -L_2 \\
        L_2 & \alpha L_2
    \end{pmatrix}
    = L_2 M,
\end{equation*}
where $M$ is a matrix with constant coefficients
\begin{equation*}
    M \coloneqq
    \begin{pmatrix}
        \alpha & -1 \\
        1 & \alpha
    \end{pmatrix}.
\end{equation*}
This matrix has two complex eigenvalues:
\begin{itemize}
    \item $\alpha + i$, with related eigenvector $(1, i)$,
    \item $\alpha - i$, with related eigenvector $(1, -i)$.
\end{itemize}
We know that the potential in both $L_1$ and $L_2$ has a limit at $\pm \infty$, which is $1 + h_0$.
Thus, the essential spectrum of $L_1$ and $L_2$ is given by $\sigma_{ess} (L_1) = \sigma_{ess} (L_2) = [1 + h_0, \infty)$, and we obtain straightforwardly the essential spectrum of $L_1 M$ and $L_2 M$, which in turn is also the essential spectrum of $L$: $(\alpha + i) . [1 + h_0, \infty) \cup (\alpha - i) . [1 + h_0, \infty)$.
A similar analysis leads to the same result for the adjoint of this operator.

In particular, this essential spectrum is completely embedded in $\{ z \in \mathbb{C} \, | \, \Re z > 0 \}$. Thus, in order to study the linear stability, one should study the remaining eigenvalues of $L$, or rather of $L^\top$, and in particular the eigenvalues embedded in $\{ z \in \mathbb{C} \, | \, \Re z \leq 0 \}$.
The kernel of $L^\top$ can be easily computed thanks to the kernel of $L_1$ and $L_2$ given in Proposition \ref{prop:kernel_schro_op}:
\begin{equation*}
    \operatorname{ker} L^\top = \mathbb{C}
    \begin{pmatrix}
        \alpha \\ 1
    \end{pmatrix}
    \partial_x \theta_{h_0}
    + \mathbb{C}
    \begin{pmatrix}
        -1 \\ \alpha
    \end{pmatrix}
    \sin \theta_{h_0}.
\end{equation*}
These two functions are related to the invariance by translation and rotation around $e_1$, as already explained previously, and can therefore be put aside by some modulation.
In view of the results found above on the dynamics in the nonlinear setting, one would expect at least another eigenvalue with negative real part, and more probably two ($L$ is real, so any eigenvalue $\lambda \in \mathbb{C} \setminus \mathbb{R}$ would give rise to a second eigenvalue $\overline{\lambda}$).
However, since the operator $L$ is not self-adjoint, this analysis is much more difficult and it is still an open problem to the best of the author's knowledge.
A positive answer to this question would still give a more precise hint of the linear and nonlinear dynamics, as it would roughly say that most solutions to \eqref{eq:llg} close to a stationary solution would evolve into a different state, except maybe on a manifold of dimension $2$. This would be better than the result in Theorem \ref{th:instability}, which states this behaviour only on a subset in form of a cone near the stationary solution.

\section*{Acknowledgements}

The author acknowledges partial support by the ANR project MOSICOF ANR-21-CE40-0004.
The author expresses gratitude to Raphaël Côte for pointing out the existence of these structures and for engaging in several insightful discussions. While the subject looked poor on first sight, it actually revealed several very interesting behaviors of the dynamics of the Landau-Lifshitz-Gilbert equation on a nanowire, completing our previous study.
The author extends appreciation to Gilles Carbou for generously sharing his code, which served as the foundation of the code of the present article.
Special thanks are also extended to Quentin Chauleur for his meticulous review of the initial version of this article.

\appendix

\section{Rayleigh quotients and eigenfunctions} \label{sec:app_rayleigh}

In this appendix, we recall a particular case of the Rayleigh quotients, where the supremum over linear subspace of codimension fixed can be explicited thanks to the eigenfunctions of the previous eigenvalues.

\begin{lem} \label{lem:link_a_rayleigh_quotient}
    Let $\lambda_1 < \lambda_2$ be the first two eigenvalues of a self-adjoint bounded-from-below operator $(\operatorname{Dom} (L), L)$ on an Hilbert space $H$, associated with the quadratic form $Q (u) \coloneqq \langle L u, u \rangle_H$. Assume that these eigenvalues are simple with respecting normalized eigenvectors $\psi_1, \psi_2 \in \operatorname{Dom} (L) \subset \operatorname{Dom} (Q)$ and that the essential spectrum is included in $[\lambda_\infty, \infty)$ where $\lambda_\infty > \lambda_2$. Then
    \begin{align*}
        \inf_{u \in \operatorname{span} (\psi_1, \psi_2)^\perp \cap \operatorname{Dom} (Q) \setminus \{ 0 \}} \frac{\langle L u, u \rangle_H}{\norm{u}_{H}^2} \eqqcolon a = \mu_3 (L) \coloneqq \sup_{\phi_1, \phi_2} \inf_{u \in \operatorname{span} (\phi_1, \phi_2)^\perp \cap \operatorname{Dom} (Q) \setminus \{ 0 \}} \frac{\langle L u, u \rangle_H}{\norm{u}_H^2},
    \end{align*}
    where $\mu_3 (L)$ is the third Rayleigh quotient of $L$.
\end{lem}

\begin{proof}
    By definition of $\mu_3 (L)$, we obviously have $a \leq \mu_3 (L)$.
    By contradiction, assume that $a < \mu_3 (L)$.
    By definition of $a$, there exists a non-trivial vector $\psi_3 \in \operatorname{Dom} (Q) \setminus \{ 0 \}$ such that $\langle \psi_3, \psi_1 \rangle_H = \langle \psi_3, \psi_2 \rangle_H = 0$ and $Q (\psi_3) < \lambda \norm{\psi_3}_H^2$ for any $\lambda \in (a, \mu_3 (L))$ to be fixed later. We can assume that $\norm{\psi_3}_{H} = 1$.
    Let $\mu_2 (L)$ be the second Rayleigh quotient
    \begin{equation*}
        \mu_2 (L) \coloneqq \sup_{\phi_1} \inf_{u \in \operatorname{span} (\phi_1)^\perp \cap \operatorname{Dom} (Q) \setminus \{ 0 \}} \frac{\langle L u, u \rangle_H}{\norm{u}_H^2}.
    \end{equation*}
    It is known that $\mu_2 (L) = \lambda_2$ (see \cite{Cheverry_Raymond_book}). Moreover, since $\lambda_2$ is simple and since the bottom of the essential spectrum is larger than $\lambda_\infty > \lambda_2$, we also know that $\mu_2 (L) < \mu_3 (L)$.
    Then, we can assume furthermore $\lambda \geq \mu_2 (L) = \lambda_2$.

    Now, let $\phi_1, \phi_2 \in H$. We know that $\operatorname{span} (\psi_1, \psi_2, \psi_3)$ is a vector space with dimension $3$ which is included in $\operatorname{Dom} (Q) \subset H$, whereas $\operatorname{span} (\phi_1, \phi_2)^\perp$ is a vector space with codimension (at least) $2$ in $H$.
    Therefore, there exists a non-trivial vector $u \in \operatorname{span} (\psi_1, \psi_2, \psi_3) \cap \operatorname{span} (\phi_1, \phi_2)^\perp$, and this $u$ is obviously in $\operatorname{Dom} (Q)$.
    We can decompose $u$ in the basis $(\psi_1, \psi_2, \psi_3)$:
    \begin{equation*}
        u = \langle u, \psi_1 \rangle_H \psi_1 + \langle u, \psi_2 \rangle_H \psi_2 + \langle u, \psi_3 \rangle_H \psi_3.
    \end{equation*}
    In particular, we can compute $Q (u)$ with this expansion. Indeed, since $\psi_1$ and $\psi_2$ are eigenvectors, and as $(\psi_1, \psi_2, \psi_3)$ are orthogonal to each other in $H$, all the cross terms vanish:
    \begin{align*}
        Q (u) &= \langle u, \psi_1 \rangle_H^2 Q (\psi_1) + \langle u, \psi_2 \rangle_H^2 Q (\psi_2) + \langle u, \psi_3 \rangle_H^2 Q (\psi_3) \\
            &= \underbrace{\lambda_1}_{< \lambda_2 \leq \lambda} \langle u, \psi_1 \rangle_H^2 + \underbrace{\lambda_2}_{\leq \lambda} \langle u, \psi_2 \rangle_H^2 + \langle u, \psi_3 \rangle_H^2 \underbrace{Q (\psi_3)}_{< \lambda} \\
            &\leq \lambda \Bigl( \langle u, \psi_1 \rangle_H^2 + \langle u, \psi_2 \rangle_H^2 + \langle u, \psi_3 \rangle_H^2 \Bigr) = \lambda \norm{u}_H^2.
    \end{align*}
    Therefore, $\inf_{u \in \operatorname{span} (\phi_1, \phi_2)^\perp \cap \operatorname{Dom} (Q) \setminus \{ 0 \}} \frac{\langle L u, u \rangle_H}{\norm{u}_H^2} \leq \lambda$. Since this is true for any $\phi_1, \phi_2 \in H$, we get $\mu_3 (L) \leq \lambda < \mu_3 (L)$, and thus the contradiction.
\end{proof}

\section{Proof of the modulation lemma} \label{sec:proof_mod_lem}

In this section, we prove Lemma \ref{lem:modulation} which decomposes the magnetization with a nice gauge satisfying two orthogonal equalities. In brief, this result is a consequence of the implicit function theorem.
%
%
We start with defining and studying the appropriate function: let
\begin{equation*}
    \mathfrak{F}: 
    \begin{aligned}[t]
        {(L^\infty)^3} \times G &\rightarrow \mathbb{R}^2 \\
        (m, g) &\mapsto 
            \begin{pmatrix}
                \int m \cdot g.\partial_x w_{h_0} \diff x \\
                \int m \cdot (e_1 \wedge g.w_{h_0}) \diff x
            \end{pmatrix}.
        \end{aligned}
\end{equation*}

This integrals are well-defined since both $\partial_x w_{h_0}$ and $e_1 \wedge w_{h_0}$ converge exponentially to $0$ at $\pm \infty$ (they are also of class $W^{k, 1}$ for all $k \in \mathbb N$).
We begin by some obvious properties on this function.

\begin{lem} \label{lem:obv_prop_F}
    The function $\mathfrak{F}$ is linear in the first variable and of class $\mathscr{C}^\infty$. Moreover, there also holds
    \begin{equation*}
        \mathfrak{F} (m, g) = 
            \begin{pmatrix}
                \int (- g). m \cdot \partial_x w_{h_0} \diff x \\
                \int (-g) . m \cdot (e_1 \wedge w_{h_0}) \diff x
            \end{pmatrix}
             = 
            \begin{pmatrix}
                \int ((- g). m - w_{h_0}) \cdot \partial_x w_{h_0} \diff x \\
                \int ((-g) . m - w_{h_0}) \cdot (e_1 \wedge w_{h_0}) \diff x
            \end{pmatrix}.
    \end{equation*}
\end{lem}

Then, we show that $\partial_g \mathfrak{F} (w_{h_0}, (0, 0))$ is invertible.

\begin{lem} \label{lem:invertible_diff}
    Writing $g= (y, \phi)$, there holds $\mathfrak{F} (w_{h_0}, (0,0)) = 0$ and
    \begin{equation*}
        \partial_y \mathfrak{F} (w_{h_0}, (0,0)) = 
            \begin{pmatrix}
                - \int \abs{\partial_x w_{h_0}}^2 \diff x \\
                0
            \end{pmatrix},
        \qquad
        \partial_\phi \mathfrak{F} (w_{h_0}, (0,0)) = 
            \begin{pmatrix}
                0 \\
                - \int \abs{e_1 \wedge  w_{h_0}}^2 \diff x
            \end{pmatrix}.
    \end{equation*}
    Therefore, $\partial_g \mathfrak{F} (w_{h_0}, (0,0))$ is invertible.
\end{lem}

\begin{proof}
    Since $\partial_y ((-g).m) = (-g).\partial_x m$ and $\partial_\phi ((-g).m) = (-g).(e_1 \wedge m)$, we get
    \begin{equation*}
        \partial_y \mathfrak{F} (w_{h_0}, (0,0)) = 
            \begin{pmatrix}
                - \int \abs{\partial_x w_{h_0}}^2 \diff x \\
                - \int \partial_x w_{h_0} \cdot (e_1 \wedge  w_{h_0})
            \end{pmatrix},
        \qquad
        \partial_\phi \mathfrak{F} (w_{h_0}, (0,0)) = 
            \begin{pmatrix}
                - \int \partial_x w_{h_0} \cdot (e_1 \wedge  w_{h_0}) \\
                - \int \abs{e_1 \wedge  w_{h_0}}^2 \diff x
            \end{pmatrix}.
    \end{equation*}
    On the other hand, we know that $\partial_x w_{h_0} = \partial_x \theta_{h_0} \, n_{h_0}$ and $e_1 \wedge w_{h_0} = \sin \theta_{h_0} \, e_3$, so that $\partial_x w_{h_0} \cdot (e_1 \wedge  w_{h_0}) \equiv 0$ on $\mathbb R$. Last, we know that $\int \abs{\partial_x w_{h_0}}^2 \diff x > 0$ and $\int \abs{e_1 \wedge  w_{h_0}}^2 \diff x > 0$, which leads to the conclusion.
\end{proof}

Last before proving Lemma \ref{lem:modulation}, we show that $g.f$ is close to $f$ when $g$ is close to $(0,0)$.

\begin{lem} \label{lem:diff_g_w_h_0}
    There exists $C > 0$ such that, for any $g \in G$, there holds
    \begin{equation*}
        \norm{g.w_{h_0} - w_{h_0}}_{H^1} \leq C \abs{g}.
    \end{equation*}
\end{lem}

\begin{proof}
    First, we know that, for any $f \in \mathcal{H}^2$ and $y \in \mathbb R$, there holds
    \begin{equation*}
        \norm{\tau_y.f - f}_{H^1} \leq C \abs{y} \norm{\partial_x f}_{H^1}.
    \end{equation*}
    Moreover, using the expression of $w_{h_0}$ in terms of $\theta_{h_0}$, there holds
    \begin{equation*}
        R_\phi . w_{h_0} - w_{h_0} = \sin \theta_{h_0}
            \begin{pmatrix}
                0 \\
                \cos \phi - 1 \\
                \sin \phi
            \end{pmatrix},
    \end{equation*}
    so that
    \begin{equation*}
        \norm{R_\phi . w_{h_0} - w_{h_0}}_{H^1} \leq C \abs{\phi}.
    \end{equation*}
    Therefore,
    \begin{align*}
        \norm{g.w_{h_0} - w_{h_0}}_{H^1} &\leq \norm{\tau_y R_\phi.w_{h_0} - R_\phi.w_{h_0}}_{H^1} + \norm{R_\phi.w_{h_0} - w_{h_0}}_{H^1} \\
            &\leq C \abs{y} \norm{w_{h_0}}_{H^1} + C \abs{\phi} \leq C \abs{g}. \qedhere
    \end{align*}
\end{proof}

\begin{proof}[Proof of Lemma \ref{lem:modulation}]
    By Lemma \ref{lem:invertible_diff}, the implicit function theorem gives some $\tilde \delta_1 > {0}$ {and $C > 0$} such that, for all ${m} \in L^\infty$ satisfying
    \begin{equation*}
        \norm{m - w_{h_0}}_{L^\infty} < \tilde \delta_1,
    \end{equation*}
    there exists a unique $g \in G$ satisfying $\abs{g} {\, \leq \,} C \norm{m - w_{h_0}}_{L^\infty}$ such that
    \begin{equation*}
        \int m \cdot g.\partial_x w_{h_0} \diff x = \int m \cdot (e_1 \wedge g.w_{h_0}) \diff x = 0.
    \end{equation*}
    Then, Lemma \ref{lem:obv_prop_F} implies that $\eta \coloneqq (-g).m - w_{h_0}$ satisfies
    \begin{equation*}
        \int \eta \cdot \partial_x w_{h_0} \diff x = \int \eta \cdot (e_1 \wedge w_{h_0}) \diff x = 0.
    \end{equation*}
    Now, we use the embedding of $H^1$ into $L^\infty$, that is the existence of $C_S > 0$ such that for all $f \in H^1$,
    \begin{equation*}
        \norm{f}_{L^\infty} {\, \leq \,} C_S \norm{f}_{H^1}.
    \end{equation*}
    Therefore, a similar conclusion can be reached if $m - w_{h_0} \in H^1$ satisfies \eqref{eq:m_close_m_h_0} for $\delta_1 \coloneqq \frac{\tilde \delta_1}{C_S}$.
    Last, in such a case, we can estimate $\norm{\eta}_{H^1}$ as follows:
    \begin{align*}
        \norm{\eta}_{H^1} &= \norm{m - g.w_{h_0}}_{H^1} \\
            &\leq \norm{m - w_{h_0}}_{H^1} + \norm{g.w_{h_0} - w_{h_0}}_{H^1} \\
            &\leq \norm{m - w_{h_0}}_{H^1} + C \abs{g} \\
            &\leq \norm{m - w_{h_0}}_{H^1} + C \norm{m - w_{h_0}}_{L^\infty} \\
            &\leq (1 + C C_S) \norm{m - w_{h_0}}_{H^1},
    \end{align*}
    where we used Lemma \ref{lem:diff_g_w_h_0} from the second to the third line.
\end{proof}

\bibliographystyle{abbrv}
\bibliography{sample}

\begin{thebibliography}{1}

\bibitem{Alouges_Soyeur__weak_LLG}
F.~Alouges and A.~Soyeur.
\newblock On global weak solutions for {L}andau-{L}ifshitz equations: existence
  and nonuniqueness.
\newblock {\em Nonlinear Anal.}, 18(11):1071--1084, 1992.

\bibitem{Angulo_Pava_nonlin_disp}
J.~Angulo~Pava.
\newblock {\em Nonlinear dispersive equations}, volume 156 of {\em Mathematical
  Surveys and Monographs}.
\newblock American Mathematical Society, Providence, RI, 2009.
\newblock Existence and stability of solitary and periodic travelling wave
  solutions.

\bibitem{Carbou_Labbe__StaticWalls}
G.~Carbou and S.~Labb{\'e}.
\newblock {Stability for Static Walls in Ferromagnetic Nanowires}.
\newblock {\em {Discrete and Continuous Dynamical Systems - Series S}}, 6:pp.
  273--290, n. 2, 2006.

\bibitem{Cheverry_Raymond_book}
C.~Cheverry and N.~Raymond.
\newblock {\em A guide to spectral theory---applications and exercises}.
\newblock Birkh\"{a}user Advanced Texts: Basler Lehrb\"{u}cher. [Birkh\"{a}user
  Advanced Texts: Basel Textbooks]. Birkh\"{a}user/Springer, Cham, [2021]
  \copyright 2021.
\newblock With a foreword by Peter D. Hislop.

\bibitem{Cote_Ignat__stab_DW_LLG_DM}
R.~C\^{o}te and R.~Ignat.
\newblock Asymptotic stability of precessing domain walls for the
  {L}andau-{L}ifshitz-{G}ilbert equation in a nanowire with
  {D}zyaloshinskii-{M}oriya interaction.
\newblock {\em Comm. Math. Phys.}, 401(3):2901--2957, 2023.

\bibitem{Cote_Ferriere__2DW}
R.~Côte and G.~Ferriere.
\newblock {Asymptotic Stability of 2-Domain Walls for the
  Landau–Lifshitz–Gilbert Equation in a Nanowire With
  Dzyaloshinskii–Moriya Interaction}.
\newblock {\em International Mathematics Research Notices}, 2024(4):3551--3600,
  11 2023.

\bibitem{Gou_2011}
Y.~Gou, A.~Goussev, J.~M. Robbins, and V.~Slastikov.
\newblock Stability of precessing domain walls in ferromagnetic nanowires.
\newblock {\em Physical Review B}, 84(10), Sept 2011.

\bibitem{Gutierrez_deLaire__Cauchy_LLG}
S.~Guti\'{e}rrez and A.~de~Laire.
\newblock The {C}auchy problem for the {L}andau-{L}ifshitz-{G}ilbert equation
  in {BMO} and self-similar solutions.
\newblock {\em Nonlinearity}, 32(7):2522--2563, 2019.

\bibitem{Lakshmanan_Nakamura__LLG}
M.~Lakshmanan and K.~Nakamura.
\newblock {Landau-Lifshitz Equation of Ferromagnetism: Exact Treatment of the
  Gilbert Damping}.
\newblock {\em Phys. Rev. Lett.}, 53:2497--2499, Dec 1984.

\end{thebibliography}

\end{document}